\documentclass[a4paper]{amsart}

\usepackage{bbm,mathptmx}
\usepackage{enumerate,pstricks,pst-plot,pst-math,hyperref}

\allowdisplaybreaks

\theoremstyle{plain}
\newtheorem{thm}{Theorem}[section]
\newtheorem{prop}[thm]{Proposition}
\newtheorem{cor}[thm]{Corollary}
\newtheorem{lem}[thm]{Lemma}

\theoremstyle{definition}
\newtheorem{dfn}[thm]{Definition}
\newtheorem{rem}[thm]{Remark}

\newtheorem{exm}[thm]{Example}

\newcounter{abb}{}
\newcommand{\bildnummer}{\stepcounter{abb}\emph{Fig.~\arabic{abb}}\vspace{2mm}}

\renewcommand{\phi}{\varphi}
\newcommand{\N}{\mathbb{N}}
\newcommand{\R}{\mathbb{R}}

\newcommand{\W}{\mathbb{W}}
\newcommand{\extT}{\mathrm{ext}_T}
\newcommand{\E}{\mathrm{E}}
\renewcommand{\epsilon}{\varepsilon}
\newcommand{\eps}{\varepsilon}
\newcommand{\<}{\left\langle}
\renewcommand{\>}{\right\rangle}
\newcommand{\1}{\mathbbmss{1}}
\newcommand{\rhoO}{\rho_{\overline\Omega}}
\newcommand{\Diag}{\mathrm{Diag}}

\newcommand{\ph}{\widehat p}

\newcommand{\be}{\begin{equation}}
\newcommand{\ee}{\end{equation}}

\begin{document}

\title{Wiener Measures on Riemannian Manifolds and the Feynman-Kac Formula}

\author{Christian B\"ar and Frank Pf\"aff\-le}

\address{Institut f\"ur Mathematik,
Universit\"at Potsdam,
Am Neuen Palais 10,
14469 Potsdam, Germany}

\subjclass[2010]{58J65, 58J35}
\keywords{Wiener measure, conditional Wiener measure, Brownian motion, Brownian bridge, Riemannian manifold, normal Riemannian covering, heat kernel, Laplace-Beltrami operator, stochastic completeness, Feynman-Kac formula}

\email{baer@math.uni-potsdam.de \textrm{and} pfaeffle@math.uni-potsdam.de}

\begin{abstract}
This is an introduction to Wiener measure and the Feynman-Kac formula on general Riemannian manifolds for Riemannian geometers with little or no background in stochastics.
We explain the construction of Wiener measure based on the heat kernel in full detail and we prove the Feynman-Kac formula for Schr\"odinger operators with $L^\infty$-potentials.
We also consider normal Riemannian coverings and show that projecting and lifting of paths are inverse operations which respect the Wiener measure.
\end{abstract}

\maketitle

\section{Introduction}

\noindent
This paper is meant as a service to the community.
It is an introduction to Wiener measure, hence path integration, and to the Feynman-Kac formula on Riemannian manifolds.
The reader should be familiar with Riemannian geometry but no background in stochastics is required.
Most results are not new; either they are contained somewhere in the literature or they are considered as folklore knowledge.

There are excellent introductions and textbooks which treat stochastic analysis on manifolds, e.g.\ \cite{Gangolli,McKean,Bismut,Ikeda,Elworthy,Emery,HT,Stroock,Hsu}.
They tend to be written from the probabilist's point of view who wants to extend stochastic analysis on Euclidean space to manifolds.
Therefore embeddings of the manifold into a high-dimensional Euclidean space are important or the frame bundle is used to transfer Brownian motion in $\R^n$ to manifolds via the so-called Eells-Elworthy-Malliavin construction.

We choose a different route.
Embeddings and the frame bundle make no appearance; Euclidean space occurs only as a special case.
The necessary measure theoretic and stochastic background is kept to a minimum and almost fully developed.
The concept of stochastic differential equations will not be used.
The starting point is the heat kernel $p: S \times S \times \R \to \R$ canonically associated to an arbitrary Riemannian manifold $S$ or, more generally, a suitable ``transition function'' on a metric measure.
For Euclidean space this is the classical Gaussian normal distribution.

We show that if a certain abstract criterion on such a transition function on a metric measure space $S$ is satisfied, then this transition function induces a measure with good properties on the set $C_{x_0}\left([0,T];S\right)$ of continuous paths emanating from a fixed point $x_0\in S$ and being parametrized on $[0,T]$.
There is also a version for the set $C_{x_0}^{y_0}\left([0,T];S\right)$ of continuous paths with fixed initial and end point.

We check that the criterion is met for the heat kernel of a closed (i.e., compact and boundaryless) Riemannian manifold $S$.
The case that $S$ is compact but does have boundary can be reduced to the case of closed manifolds by a doubling trick.
Finally, if $S$ is an arbitrary Riemannian manifold, which need not be geodesically complete, a limiting procedure involving an exhaustion of $S$ by compact subsets with smooth boundary yields the desired measures also in this case.
The measure that one obtains on $C_{x_0}\left([0,T];S\right)$ is known as Wiener measure and the corresponding stochastic process is Brownian motion.
For $C_{x_0}^{y_0}\left([0,T];S\right)$ one obtains the conditional Wiener measure and the Brownian bridge.

If the end point is not fixed, we can let $T\to\infty$ and another limiting procedure yields Wiener measure and Brownian motion on $C_{x_0}\left([0,\infty);S\right)$.
This is important if one wants to study long time asymptotic properties of random paths.
It requires an assumption however; the manifold needs to be stochastically complete.
The stochastically incomplete case can also be dealt with by passing to the 1-point compactification of $S$, but we will be very brief on this.

One nice feature of the approach based on transition functions is its extensibility to more singular spaces such as Riemannian orbifolds.

The paper is organized as follows.
In the next section we use a classical tool due to Kolmogorov to construct stochastic processes given good transition functions.
We then develop the Kolmogorov-Chentsov criterion which ensures that these processes have continuous paths.
The paths are actually H\"older continuous of a suitable order.

The criterion cannot be applied directly to the heat kernel on an arbitrary Riemannian manifold because the necessary uniform estimates will not hold in general.
This is why we first consider closed manifolds, then compact manifolds with boundary, and finally pass to a limit to treat general manifolds.
This is done in Section~\ref{wienerheatkernel}.

In the subsequent section we consider normal Riemannian coverings.
A typical example is the standard covering of Euclidean space over a flat torus or a flat cylinder.
We show that projecting and lifting of paths are inverse operations which respect the Wiener measure.

In the sixth section we compare the expectation value for the distance of a random path from the initial point after time $t$ for Euclidean and for hyperbolic space.
It turns out that for small time $t\searrow 0$ the expectation values have the same asymptotic behavior but for $t\to\infty$ it grows much faster for hyperbolic space.
This is plausible because the volume of metric balls grows exponentially fast as a function of the radius in hyperbolic space while it only grows polynomially in Euclidean space.
Therefore a random path in hyperbolic space is less likely to return to the relatively small neighborhood of the initial point than in Euclidean space.

In the last section we provide a proof of the Feynman-Kac formula for Schr\"odinger operators with $L^\infty$-potentials.
This is not the most general class of potentials one can treat, but we wanted to keep the technical effort at a reasonable size.

There are three appendices which contain some technical material which would have interrupted the exposition of the main ideas.
\medskip

\noindent
{\bf Acknowledgments.}
It is a pleasure to thank SFB 647 ``Raum-Zeit-Materie'' funded by Deutsche Forschungsgemeinschaft and the Max Planck Institute for Mathematics in Bonn for financial support.
Special thanks go to A.~Grigor'yan, S.~R{\oe}lly and A.~Thalmaier for helpful discussion.

\section{Constructing substochastic processes with continuous paths}\label{section_2}

\noindent
In this section we develop the necessary measure theoretic background.
The aim is to show how transition functions with suitable properties on metric measure spaces lead to substochastic processes with continuous paths.

\begin{dfn}
Let $(S,\mathcal{B})$ be a measurable space\footnote{This means that $S$ is a set and that $\mathcal{B}$ is a  $\sigma$-algebra on $S$.} and let $(\Omega,\mathcal{A},P)$ be a measure space.\footnote{This means that $(\Omega,\mathcal{A})$ is a measurable space and that $P$ is a measure on $\mathcal{A}$.}
We assume $P(\Omega)\le 1$.
Let $I\subset\R$. 
A family $(X_t)_{t\in I}$ of measurable maps $X_t:(\Omega,\mathcal{A})\to (S,\mathcal{B})$  is called a \emph{substochastic process} on $\Omega$  with values in $S$ (and with index set $I$).
\end{dfn}

This generalizes the usual terminology of a stochastic process where one assumes $P(\Omega)=1$.
The construction of substochastic processes which we will demonstrate here is based on transition functions.

\begin{dfn}\label{def:transfunc}
Let $(S,\mathcal{B},\mu)$ be a measure space
and let $T>0$.
A function $p: (0,T]\times S\times S\to[0,\infty]$, $(t,x,y)\mapsto p_t(x,y)$, is called \emph{substochastic transition function} if for all $s,t>0$ and $x,z\in S$ one has
\begin{enumerate}[(a)]
\item\label{trans0}
the map $S\times S\to [0,\infty], (y,w)\mapsto p_t(y,w)$ is measurable with respect to the product $\sigma$-algebra of $S\times S$,
\item\label{trans1} 
$\int_S p_t(y,x)\, d\mu(y)\le 1$ , and
\item\label{trans2} 
$\int_S p_t(z,y)\,p_s(y,x)\,d\mu(y)=p_{t+s}(z,x)$.
\end{enumerate}
\end{dfn}

The next definition follows \cite[p.~113]{Gromov} where metric measure spaces are called mm~spaces.

\begin{dfn}
A triple $(S,\rho,\mu)$ is called a \emph{metric measure space} if $(S,\rho)$ is a complete separable metric space and $\mu$ is a $\sigma$-finite measure on the Borel $\sigma$-algebra\footnote{This is the $\sigma$-algebra generated by the open subsets of $S$.} of $S$.
\end{dfn}
\begin{exm}\label{exm_complete_mm}
Let $S$ be a connected (possibly non-compact) differentiable manifold equipped with a $\sigma$-finite measure $\mu$ on the Borel $\sigma$-algebra.
Then there is a complete Riemannian metric on $S$, and the induced Riemannian distance $\rho$ makes $(S,\rho,\mu)$ a metric measure space.
\end{exm}

The main result of this section are the following two existence theorems:

\begin{thm}\label{thm_existence}
Let $(S,\rho,\mu)$ be a metric measure space, let  $T>0$, and let $p$ be a substochastic transition function on $S$.
We fix $x_0\in S$.
Suppose there are constants $a,b,C,\eps>0$ such that for any $y\in S$ and for any $\tau\in(0,\varepsilon)$ one has
\begin{equation}\label{kernelcondition}
\int_S \rho(z,y)^a \;p_{\tau}(z,y)\,d\mu(z)\le C\cdot \tau^{1+b}.
\end{equation}
Then there exists a measure space $(\Omega,\mathcal{A},P)$ with $P(\Omega)\le 1$ and a substochastic process $(Y_t)_{t\in [0,T]}$ on $\Omega$  with values in $S$, which has the following three properties:
\begin{enumerate}[(i)]
\item\label{kernelcondition1} 
$Y_0\equiv x_0$;
\item\label{kernelcondition2} 
for any $n\in \N$, any $0<t_1<\ldots<t_n= T$ and any Borel set $B\subset S^n=S\times\cdots\times S$ ($n$ times) one has
\begin{align*}
 &P´\left((Y_{t_1},\ldots,Y_{t_n})\in B \right)
=P\left(\left\{\omega\in\Omega \,\big|\,
  (Y_{t_1}(\omega),\ldots,Y_{t_n}(\omega))\in B\right\} \right)  \\
 &=\int_{S^n}\1_B(x_1,\ldots,x_n)p_{t_n-t_{n-1}}(x_n,x_{n-1})\cdots
 p_{t_2-t_1}(x_2,x_1) p_{t_1}(x_1,x_0)d\mu(x_1)\cdots d\mu(x_n);
\end{align*}
\item\label{kernelcondition3}
$(Y_t)_{t\in [0,T]}$ has H\"older continuous paths of order $\theta$ for any $\theta\in(0,b/a)$.
\end{enumerate}
\end{thm}

Here $\1_B$ denotes the characteristic function of $B\subset S^n$.
Property~\eqref{kernelcondition1} means that $Y_0(\omega)=x_0$ for all $\omega\in\Omega$ and similary \eqref{kernelcondition3} means that the map $Y_\bullet(\omega):[0,T]\to S$, given by $t\mapsto Y_t(\omega)$, is H\"older continuous of order $\theta$ for all $\omega\in\Omega$.
From \eqref{kernelcondition2} we immediately get that $P(\Omega)=P(\omega\in\Omega\mid Y_T(\omega)\in S)=\int_S p_T(x_1,x_0)d\mu(x_1)\le 1$, i.e., in general, $P$ will not be a probability measure.

In Theorem~\ref{thm_existence} we have prescribed the initial point $x_0$ of the paths of the substochastic process $(Y_t)_{t\in [0,T]}$.
In the second existence theorem we prescribe both the initial and the end point.
This requires stronger assumptions; the integral condition \eqref{kernelcondition} is replaced by the corresponding pointwise condition.

\begin{thm}\label{thm_existence_bridge}
Let $(S,\rho,\mu)$ be a metric measure space, let  $T>0$, and let $p$ be a substochastic transition function on $S$.
We fix $x_0,y_0\in S$.
Suppose there are constants $a,b,C,\eps>0$ such that for any $z,y\in S$ and for any $\tau\in(0,\varepsilon)$ one has
\begin{equation}\label{kernelcondition_b}
\rho(z,y)^a \;p_{\tau}(z,y) \le C\cdot \tau^{1+b}.
\end{equation}
Then there exists a measure space $(\Omega,\mathcal{A},P)$ with $P(\Omega)\le 1$ and a substochastic process $(Y_t)_{t\in [0,T]}$ on $\Omega$  with values in $S$, which has the following three properties:
\begin{enumerate}[(i)]
\item\label{kernelcondition_b1} 
$Y_0\equiv x_0$ and $Y_T\equiv y_0$;
\item\label{kernelcondition_b2} 
for any $n\in \N$, any $0<t_1<\ldots<t_n< T$ and any Borel set $B\subset S^n$ one has
\begin{align*}
 &P´\left((Y_{t_1},\ldots,Y_{t_n})\in B \right)
=P\left(\left\{\omega\in\Omega \,\big|\,
  (Y_{t_1}(\omega),\ldots,Y_{t_n}(\omega))\in B\right\} \right)  \\
 &=\int_{S^n}\1_B(x_1,\ldots,x_n)p_{T-t_{n}}(y_0,x_{n})p_{t_n-t_{n-1}}(x_n,x_{n-1})\cdots
 p_{t_2-t_1}(x_2,x_1) p_{t_1}(x_1,x_0)d\mu(x_1)\cdots d\mu(x_n);
\end{align*}
\item\label{kernelcondition_b3}
$(Y_t)_{t\in [0,T]}$ has H\"older continuous paths of order $\theta$ for any $\theta\in(0,b/a)$.
\end{enumerate}
\end{thm}

For the construction of $\Omega$, $P$ and $(Y_t)_{t\in [0,T]}$ we will use classical results of measure theory (with slight modifications).
For the convenience of the reader we will present a review of these results.

For any set $F$ let $S^F$ denote the set of all maps $F\to S$.
Given two subsets $G\subset F\subset [0,T]$, one has the natural projection
\[
\pi^{F}_{G}:S^{F}\to S^{G},
\]
given by restricting maps $F\to S$ to $G$.
For $t\in[0,T]$ we abbreviate $\pi^F_{t}:=\pi^F_{\{ t\}}$,
and if $F=[0,T]$ we write $\pi_G:=\pi^{F}_{G}$ and $\pi_{t}:=\pi_{\{ t\}}=\pi^{F}_{\{ t\}}$.

For $F\subset [0,T]$ the product $\sigma$-algebra $\mathcal{B}^F$ is defined as the smallest $\sigma$-algebra on $S^F$ for which all projections $\pi^F_t:S^F\to S$ are measurable maps, $t\in F$.
It follows that the natural projections $\pi^{F}_{G}:S^{F}\to S^{G}$ are measurable maps with respect to the $\sigma$-algebras $\mathcal{B}^F$ and $\mathcal{B}^G$.
\begin{rem}\label{rem_product_measurable}
For any substochastic process $(X_t)_{t\in [0,T]}$ on $(\Omega,\mathcal{A},P)$ with values in $(S,\mathcal{B})$ and any $\omega\in\Omega$ the path $X_\bullet (\omega)\in S^{[0,T]}$ is defined by $t\mapsto X_t(\omega)$.
Then $(X_t)_{t\in [0,T]}$ induces a measurable map
$X:(\Omega,\mathcal{A})\to (S^{[0,T]},\mathcal{B}^{[0,T]})$ given by $\omega\mapsto X_\bullet (\omega)$.
\end{rem}
\begin{dfn}
Denote the set of all finite subsets of $[0,T]$ by $\mathcal{P}_0(T)=\{F\subset [0,T]\mid F\mbox{ finite} \}$.
A family $(P_F)_{F\in \mathcal{P}_0(T)}$ of finite measures on $(S^F,\mathcal{B}^F)$ is called \emph{consistent} if 
\[ 
\left(\pi^F_G\right)_*P_F=P_G\quad \mbox{for all } G,F\in \mathcal{P}_0(T) \mbox{ with } G\subset F.
\]
In other words, $P_F\left((\pi^F_G)^{-1}(B) \right)=P_G(B)$ for any $B\in\mathcal{B}^G$.
\end{dfn}

\begin{rem}\label{rem_consistent_mass}
All measures in a consistent family $(P_F)_{F\in \mathcal{P}_0(T)}$ have the same total mass because if $G\subset F$, then $P_F(S^F)=P_F((\pi^F_G)^{-1}(S^G))=P_G(S^G)$.
\end{rem}

\begin{lem}\label{consistent_family_density}
Let $(S,\rho,\mu)$ be a metric measure space, let $p$ be a substochastic transition function on $S$, and let  $T>0$.
We fix $x_0,y_0\in S$.
Then for any finite subset $F=\{0\le t_1<\ldots t_n\le T\}\subset [0,T]$ one gets finite measures $Q_F$ and $\widehat Q_F$ on $(S^F,\mathcal{B}^F)$ by setting for any $B\in \mathcal{B}^F$
\begin{align}
Q_F(B)
 :=
\int_{S^{n+1}}\1_B(x_1,\ldots,x_n)p_{T-t_n}(z,x_n)p_{t_n-t_{n-1}}&(x_n,x_{n-1})\cdots
 p_{t_2-t_1}(x_2,x_1)\times\label{def_density_measure}\\
 \times& p_{t_1}(x_1,x_0)d\mu(x_1)\cdots d\mu(x_n)\, d\mu(z) \nonumber 
\end{align}
and 
\begin{align}
\widehat Q_F(B)
 :=
\int_{S^{n}}\1_B(x_1,\ldots,x_n)p_{T-t_n}(y_0,x_n)p_{t_n-t_{n-1}}&(x_n,x_{n-1})\cdots
 p_{t_2-t_1}(x_2,x_1)\times\label{def_density_measure_b}\\
 &\times p_{t_1}(x_1,x_0)d\mu(x_1)\cdots d\mu(x_n). \nonumber 
\end{align}
Furthermore, both families $(Q_F)_{F\in \mathcal{P}_0(T)}$ and $(\widehat Q_F)_{F\in \mathcal{P}_0(T)}$ are consistent.
\end{lem}

If $t_1=0$ in \eqref{def_density_measure} or in \eqref{def_density_measure_b}, then one uses the convention that $p_0(x_1,x_0)d\mu(x_1)=d \delta_{x_0}(x_1)$ means integration with respect to the Dirac measure supported at $x_0$.
Similarly, if $t_n=T$, one understands $p_0(z,x_n)d\mu(x_n)=d \delta_{z}(x_n)$.

\begin{proof}[Proof of Lemma~\ref{consistent_family_density}]
In order to check consistency of $(Q_F)_{F\in \mathcal{P}_0(T)}$ or of $(\widehat Q_F)_{F\in \mathcal{P}_0(T)}$ it suffices to consider finite subsets $G\subset F$ of $[0,T]$ of the form
$F=\left\{t_1<\ldots<t_{l-1}<s<t_l\ldots<t_n \right\}$ and $G=\left\{t_1<\ldots<t_{l-1}<t_l\ldots<t_n \right\}$. 
We abbreviate the corresponding projection $\pi^F_G$ by $\pi$. 
For any $B\in\mathcal{B}^G$ we note that $\1_B(x_1,\ldots,x_n)=\1_{\pi^{-1}(B)}(x_1\ldots,x_{l-1},y,x_l\ldots x_n)$ for all $x_1,\ldots,x_n, y\in S$.
We compute:
\begin{align*}
& Q_F(\pi^{-1}(B))\\
&= \int_{S^{n+2}}\1_{\pi^{-1}(B)}(x_1\ldots,x_{l-1},y,x_l\ldots x_n )
  p_{T-t_n}(z,x_n)\cdot p_{t_n-t_{n-1}}(x_n,x_{n-1})\cdots
   p_{t_l-s}(x_l,y)\times\\
&\qquad\qquad
 \times p_{s-t_{l-1}}(y,x_{l-1})\cdots  p_{t_1}(x_1,x_0)
   d\mu(x_1)\cdots d\mu(x_{l-1})d\mu(y) d\mu(x_l)\ldots d\mu(x_n) d\mu(z)\\
&=\int_{S^{n+1}}\1_B(x_1,\ldots,x_n)p_{T-t_n}(z,x_n)\cdot p_{t_n-t_{n-1}}(x_n,x_{n-1})\cdots
 p_{t_1}(x_1,x_0)d\mu(x_1)\ldots d\mu(x_n) d\mu(z)\\
&= Q_G(B),
\end{align*}
where we used Property~\eqref{trans2} of the substochastic transition function.
The family $(\widehat Q_F)_{F\in \mathcal{P}_0(T)}$ is treated similary.
\end{proof}

The next theorem is a classical tool for the construction of (sub)stochastic processes when their finite-dimensional distributions are given in terms of consistent families.

\begin{thm}[Kolmogorov]\label{thm_projectivelimit}
Let $S$ be a complete separable metric space with Borel $\sigma$-algebra $\mathcal{B}$.
Let $T>0$, and let $(P_F)_{F\in \mathcal{P}_0(T)}$ be a consistent family of finite measures on $(S^F,\mathcal{B}^F)$.
Then there exists a unique finite measure $P$ on ($S^{[0,T]}$,$\mathcal{B}^{[0,T]}$) such that
\begin{equation}\label{eqn_projectivelimit}
\left(\pi_F \right)_*P=P_F\quad\mbox{ for all finite }F\subset[0,T].
\end{equation}
\end{thm} 
\begin{proof}
By Remark~\ref{rem_consistent_mass}, all measures $P_F$ have the same total mass $m$, say.
If $m=0$, then the trivial measure $P=0$ satisfies \eqref{eqn_projectivelimit}. 
We notice that the $\sigma$-algebra $\mathcal{B}^{[0,T]}$ is generated by 
\[
\mathcal{G}:=\Big\{\bigcap_{i=1}^n \pi_{t_i}^{-1}(U_i)\,\Big|\,n\in\N,\, 0\le t_1\le\ldots\le t_n\le T,\mbox{ and } U_1,\ldots, U_n\subset S\mbox{ open}\Big\},
\]
which is stable under $\cap$.
Since in the case $m=0$ any measure satisfying \eqref{eqn_projectivelimit} is zero on $\mathcal{G}$,
$P=0$ is the only such measure.

For the case $m=1$, the proof of this theorem can be found e.g.\ in \cite[Thm.~35.3]{Bauer} or in \cite[Thm.~12.1.2]{Dud}. 
Then the measure $P$ has again total mass $1$.
The fact that $S$ is a complete separable metric space enters when showing the $\sigma$-additivity of $P$.

In the general case $m>0$, we can consider the family $(\frac{1}{m}P_F)_F$ and thus reduce to the case $m=1$.
\end{proof}

\begin{cor}\label{step1}
Let $(S,\rho,\mu)$ be a metric measure space, let  $T>0$, and let $p$ be a substochastic transition function on $S$.
We fix $x_0,y_0\in S$.

Then there exists a measure space $(\Omega,\mathcal{A},P)$ with $P(\Omega)\le 1$ and a substochastic process $(X_t)_{t\in [0,T]}$ on $\Omega$  with values in $S$ having Property~\eqref{kernelcondition2} in Theorem~\ref{thm_existence} and such that $X_0(\omega)=x_0$ holds for almost all $\omega\in\Omega$.

Moreover, there exists a measure $\widehat P$ on $(\Omega,\mathcal{A})$ with $\widehat P(\Omega)\le 1$ and a substochastic process $(\widehat X_t)_{t\in [0,T]}$ on $\Omega$  with values in $S$ having Property~\eqref{kernelcondition_b2} in Theorem~\ref{thm_existence_bridge} and such that $\widehat X_0(\omega)=x_0$ and $\widehat X_T(\omega)=y_0$ hold for almost all $\omega\in\Omega$.
\end{cor}

\begin{proof}
We choose $(\Omega,\mathcal{A}):=(S^{[0,T]},\mathcal{B}^{[0,T]})$.
We apply Theorem~\ref{thm_projectivelimit} to the consistent family $(Q_F)_{F\in \mathcal{P}_0(T)}$ from Lemma~\ref{consistent_family_density} and we get a measure $P$ on $(\Omega,\mathcal{A})$.
Setting $X_t:=\pi_t:S^{[0,T]}\to S$ we obtain an $S$-valued substochastic process $(X_t)_{t\in [0,T]}$.
When we insert (\ref{def_density_measure}) into condition (\ref{eqn_projectivelimit}), we recover exactly Property~\eqref{kernelcondition2} from Theorem~\ref{thm_existence}.
From
\begin{align*}
P(\{\omega\in\Omega&\mid X_0(\omega)=x_0\})
=
P(\pi_0^{-1}(x_0))
=
(\pi_0)_*P(\{x_0\})
=
Q_0(\{x_0\}) \\
&=
\int_{S^2}\1_{\{x_0\}}(x_1) p_{T-0}(z,x_1)d\delta_{x_0}(x_1)d\mu(z)
=
\int_S p_T(z,x_0)d\mu(z)
=
Q_0(S)
=
P(\Omega)
\end{align*}
we see that $X_0(\omega)=x_0$ holds for almost all $\omega\in\Omega$.

Similarly, applying Theorem~\ref{thm_projectivelimit} to the consistent family $(\widehat Q_F)_{F\in \mathcal{P}_0(T)}$ from Lemma~\ref{consistent_family_density} and we get a measure $\widehat P$ on $(\Omega,\mathcal{A})$.
Again putting $\widehat X_t:=\pi_t:S^{[0,T]}\to S$ we obtain an $S$-valued substochastic process $(\widehat X_t)_{t\in [0,T]}$ having Property~\eqref{kernelcondition_b2} in Theorem~\ref{thm_existence_bridge}.
As above one checks that $\widehat X_0(\omega)=x_0$ and $\widehat X_T(\omega)=y_0$ hold for almost all $\omega\in\Omega$.
\end{proof}

\begin{rem}\label{cylinderfunctions}
In the situation of Corollary~\ref{step1} we consider $F=\{0\le t_1<\cdots < t_n \le T \}$ and 
a function $h:S^F=S^n \to  [-\infty,\infty]$ which is integrable with respect to the measure $Q_F$ or nonnegative measurable.
We set $f:=h\circ\pi_F: S^{[0,T]}\to [-\infty,\infty]$.
Functions of this type are sometimes called {\em cylindrical functions}.
From (\ref{def_density_measure}) we get by the general transformation formula that
\begin{align*}
&\int_{S^{[0,T]}} f\,dP
 = \int_{S^{n+1}} h\,d((\pi_F)_*P) 
 = \int_{S^{n+1}} h\,dQ_F \\
 &=\int_{S^{n+1}} h(x_1,\ldots,x_n)p_{T-t_n}(z,x_n)p_{t_n-t_{n-1}}(x_n,x_{n-1})\cdots
 p_{t_1}(x_1,x_0)d\mu(x_1)\cdots d\mu(x_n)\, d\mu(z).
\end{align*}
Similarly, using the measure $\widehat P$ instead of $P$, we get
\begin{align*}
\int_{S^{[0,T]}} f&\,d\widehat P
 = \int_{S^{n}} h\,d\widehat Q_F \\
 &=\int_{S^{n}} h(x_1,\ldots,x_n)p_{T-t_n}(y_0,x_n)p_{t_n-t_{n-1}}(x_n,x_{n-1})\cdots
 p_{t_1}(x_1,x_0)d\mu(x_1)\cdots d\mu(x_n).
\end{align*}
\end{rem}

\noindent
Next, we want to modify this substochastic process such that its paths are continuous.
\begin{dfn}
Let $X=(X_t)_{t\in [0,T]}$ and $Y=(Y_t)_{t\in [0,T]}$  be substochastic processes on $(\Omega,\mathcal{A},P)$ with values in $(S,\mathcal{B})$.
One calls $Y$ a \emph{version} of $X$ if $X_t=Y_t$ almost surely for every $t\in [0,T]$, i.e.,
\[
P\left(\left\{ \omega\in\Omega\, \big|\, X_t(\omega)\ne Y_t(\omega)  \right\} \right) =0.
\]
\end{dfn}
\begin{rem}\label{version_doesnot_change_i}
Any version $(Y_t)_{t}$ of the substochastic process $(X_t)_{t}$ constructed in Corollary~\ref{step1} again has Property~\eqref{kernelcondition2} from Theorem~\ref{thm_existence}.
\end{rem}
\begin{exm}\label{Cauchy_process}
We consider $S=\R$ equipped with the euclidean metric and the Lebesgue measure $\mu$. 
Then, for any $T>0$ the function $p:(0,T]\times \R\times\R\to [0,\infty)$ given by 
\[p_t(x,y)=\frac{t}{\pi(t^2+(x-y)^2)}\]
is a substochastic transition function with $\int_\R p_t(y,x)d\mu(y)=1$ for any $x,y\in\R$ and $t>0$.
The substochastic process $(X_t)_t$ constructed out of $p$ as in  Corollary~\ref{step1} is called {\it Cauchy process}.
It is an example for a L\'evy process which coincides with its associated jump process (see e.g.~\cite[Chap.\ I.4]{Protter} for the terminology).
Hence $(X_t)_t$ is a pure jump process and does not possess any version with continuous paths. 
\end{exm}

Example~\ref{Cauchy_process} shows that generally versions with continuous paths need not exist.
A classical criterion for that is given by the following theorem.
To that end we define the {\em substochastic expectation} of a measurable maps $Z:(\Omega,\mathcal{A},P)\to [0,\infty)$  as
\[
\E[Z]:=\int_{\Omega} Z(\omega)\, dP(\omega).
\]

\begin{thm}[Kolmogorov, Chentsov]\label{thm_continuousversion}
Let $(X_t)_{t\in [0,T]}$ be a substochastic process with values in a metric measure space $(S,\rho,\mu)$.
Suppose there are constants $a,b,C,\varepsilon>0$ such that 
\begin{equation}\label{Chentsov_condition}
\E\left[\rho(X_s,X_t)^a \right]\le C\cdot \left|s-t\right|^{1+b}\quad\mbox{ whenever } |t-s|<\varepsilon.
\end{equation}
Then there is a version $(Y_t)_{t\in [0,T]}$ of $(X_t)_{t\in [0,T]}$ having H\"older continuous paths of any order $\theta\in (0,\frac{b}{a})$.
\end{thm}
The proof is a modification of that of \cite[Thm.~3.23]{Kal} and we give it in Appendix~\ref{appendix_Kolmogorov_Chentsov}.
Now we are in the position to prove Theorems~\ref{thm_existence} and \ref{thm_existence_bridge}.
We show Theorem~\ref{thm_existence}, the proof of Theorem~\ref{thm_existence_bridge} being analogous.

\begin{proof}[Proof of Theorem~\ref{thm_existence}]
We take the substochastic process constructed in Corollary~\ref{step1} and verify that (\ref{kernelcondition}) implies (\ref{Chentsov_condition}):
\begin{align*}
\E\left[\rho(X_s,X_t)^a \right]&=\int_{S\times S \times S}\rho(z,y)^a \;p_{T-t}(w,z)\;p_{t-s}(z,y)\;p_s(y,x_0)\,d\mu(y)\,d\mu(z)\,d\mu(w)\\
&\le \int_S \left(\int_S\rho(z,y)^a \;p_{t-s}(z,y)\; d\mu(z)\right)p_s(y,x_0)\, d\mu(y)\\
&\le C\cdot \left|s-t\right|^{1+b}\cdot \int_Sp_s(y,x_0)\, d\mu(y)\\
&\le C\cdot \left|s-t\right|^{1+b}
\end{align*}
whenever $|t-s|<\varepsilon$.
Hence Theorem~\ref{thm_continuousversion} can be applied and yields a version $(Y_t)_{t\in [0,T]}$ of $(X_t)_{t\in [0,T]}$ having H\"older continuous paths of any order $\theta\in (0,\frac{b}{a})$, thus proving \eqref{kernelcondition3}.

From Corollary~\ref{step1} we know that $X_0(\omega)=x_0$ and hence $Y_0(\omega)=x_0$ for almost all $\omega\in\Omega$.
By removing a null set from $\Omega$ we can therefore achieve $Y_0\equiv x_0$, proving \eqref{kernelcondition1}.

Finally, Remark~\ref{version_doesnot_change_i} shows that $(Y_t)_{t\in [0,T]}$ has Property~\eqref{kernelcondition2}.
\end{proof}

Given $T>0$, a point $x_0\in S$ and a substochastic transition function $p$ satisfying (\ref{kernelcondition}),
we call an $S$-valued substochastic process $(Y_t)_{t\in [0,T]}$ as given in Theorem~\ref{thm_existence} a \emph{diffusion process} generated by $p$ with starting point $x_0$.
Now, we interprete any such diffusion process $(Y_t)_{t\in [0,T]}$ as a measurable map $Y$ with values in $(S^{[0,T]},\mathcal{B}^{[0,T]})$ as in Remark~\ref{rem_product_measurable}.
From Theorem~\ref{thm_existence} it is clear that $Y$ takes values in the set of continuous maps starting at $x_0$,
\[
C_{x_0}\left([0,T];S\right)
:=
\left\{w:[0,T]\to S \,\mid\, w \mbox{ continuous and } w(0)=x_0\right\}.
\]
\begin{rem}\label{rem:erzeugersigma}
A priori, $C_{x_0}\left([0,T];S\right)$ carries two natural $\sigma$-algebras:
the Borel $\sigma$-algebra $\mathcal{C}_1$ generated by the compact-open topology and the trace $\sigma$-algebra $\mathcal{C}_2$ which is given as
\[
\mathcal{C}_2=\left\{ B\cap C_{x_0}\left([0,T];S\right)\,\big|\, B\in\mathcal{B}^{[0,T]}\right\}.
\]
For both of these $\sigma$-algebras 
\begin{align*}
{\mathcal E}:=\Big\{ E\cap C_{x_0}\left([0,T];S\right)\,\big|\,& E=\bigcap_{i=1}^n \pi_{t_i}^{-1}(U_i)\mbox{ for some }0\le t_1<\ldots<t_n\le T \\
&\mbox{ and open subsets }U_1,\ldots,U_n\subset S\Big\}
\end{align*}
forms a generator.
Hence both $\sigma$-algebras coincide, $\mathcal{C}:=\mathcal{C}_1=\mathcal{C}_2$.
\end{rem}
Therefore $Y$ is a measurable map $(\Omega,\mathcal{A})\to \left(C_{x_0}\left([0,T];S\right),\mathcal{C}\right)$ and $Y$ induces a measure $\W_{x_0}$ on $\left(C_{x_0}\left([0,T];S\right),\mathcal{C}\right)$ by $\W_{x_0}:=Y_*P$.
In other words,
\[
\W_{x_0}(C)=P\left(\{\omega\,\big|\,Y_\bullet(\omega)\in C \}\right),
\]
for any $C\in\mathcal{C}$.
We conclude:
\begin{cor}\label{wienermeasureexists}
Let $(S,\rho,\mu)$ be a metric measure space, let $T>0$ and let $p$ be a substochastic transition function.
We fix $x_0\in S$.
Suppose there are constants $a,b,C,\varepsilon>0$ such that for any $y\in S$ and for any $\tau\in(0,\varepsilon)$ one has
\begin{equation}\label{WienerChentsov}
\int_S \rho(z,y)^a \;p_{\tau}(z,y)\,d\mu(z)\le C\cdot \tau^{1+b}.
\end{equation}
Then there exists a unique measure $\W_{x_0}$ on $\left(C_{x_0}\left([0,T];S\right),\mathcal{C}\right)$ such that
\begin{align}\label{wienerbed}
\W_{x_0}&\left(\left\{w\in C_{x_0}\left([0,T];S\right) \,\big|\,
  w(t_1)\in U_1,\ldots,w(t_n)\in U_n\right\} \right)  \\
 =\int_{S^n}\1_{U_1\times\ldots\times U_n}&(x_1,\ldots,x_n)\,p_{t_n-t_{n-1}}(x_n,x_{n-1})\cdots
 p_{t_2-t_1}(x_2,x_1) p_{t_1}(x_1,x_0)d\mu(x_1)\cdots d\mu(x_n), \nonumber
\end{align}
for any $n\in \N$, any $0<t_1<\ldots<t_n= T$ and any open subsets $U_1,\ldots,U_n\subset S$.

Moreover, for any $\theta\in(0,b/a)$ the set of H\"older continuous paths of order $\theta$ has full measure in $\left(C_{x_0}\left([0,T];S\right),\mathcal{C},\W_{x_0}\right)$.
\end{cor}

\begin{proof}
The existence of $\W_{x_0}$ is clear by the above discussion.
Since $\mathcal E$ is stable under $\cap$ and the values of $\W_{x_0}$ for elements in $\mathcal{E}$ are given by (\ref{wienerbed}), there is at most one measure $\W_{x_0}$ as in the theorem.
\end{proof}
For $x_0,y_0\in S$ we put
\[
C_{x_0}^{y_0}\left([0,T];S\right)
:=
\left\{w\in C_{x_0}\left([0,T];S\right)\mid w(T)=y_0\right\}.
\]
Again, the Borel $\sigma$-algebra and the trace $\sigma$-algebra on $C_{x_0}^{y_0}\left([0,T];S\right)$ coincide and are denoted by $\mathcal{C}$.
Then we get similary

\begin{cor}\label{wienerbridgeexists}
Let $(S,\rho,\mu)$ be a metric measure space, let  $T>0$, and let $p$ be a substochastic transition function on $S$.
We fix $x_0,y_0\in S$.
Suppose there are constants $a,b,C,\eps>0$ such that for any $z,y\in S$ and for any $\tau\in(0,\varepsilon)$ one has
\begin{equation}\label{WienerChentsovBridge}
\rho(z,y)^a \;p_{\tau}(z,y) \le C\cdot \tau^{1+b}.
\end{equation}
Then there exists a unique measure $\W_{x_0}^{y_0}$ on $\left(C_{x_0}^{y_0}\left([0,T];S\right),\mathcal{C}\right)$ such that
\begin{align}\label{wienerbridgebed}
\W_{x_0}^{y_0}&\left(\left\{w\in C_{x_0}^{y_0}\left([0,T];S\right) \,\big|\,
  w(t_1)\in U_1,\ldots,w(t_n)\in U_n\right\} \right)  \\
 =\int_{S^n}\1_{U_1\times\ldots\times U_n}&(x_1,\ldots,x_n)\,p_{T-t_{n}}(y_0,x_{n})\cdots
 p_{t_2-t_1}(x_2,x_1) p_{t_1}(x_1,x_0)d\mu(x_1)\cdots d\mu(x_n),\nonumber 
\end{align}
for any $n\in \N$, any $0<t_1<\ldots<t_n< T$ and any open subsets $U_1,\ldots,U_n\subset S$.

Moreover, for any $\theta\in(0,b/a)$ the set of H\"older continuous paths of order $\theta$ has full measure in $\left(C_{x_0}^{y_0}\left([0,T];S\right),\mathcal{C},\W_{x_0}^{y_0}\right)$.\qed
\end{cor}

\begin{dfn}
Let $(S,\mathcal{B},\mu)$ be a measure space, let $T>0$,
and let $p$ be a substochastic transition function on $S$, furthermore let $x_0,y_0\in S$.
A measure $\W_{x_0}$ on $\left(C_{x_0}\left([0,T];S\right),\mathcal{C}\right)$ with Property~(\ref{wienerbed}) is called {\em Wiener measure} induced by $p$.

A measure $\W_{x_0}^{y_0}$ on $\left(C_{x_0}^{y_0}\left([0,T];S\right),\mathcal{C}\right)$ with Property~(\ref{wienerbridgebed}) is called {\em conditional Wiener measure} induced by $p$.
\end{dfn}
\begin{rem}
The proof of Corollary~\ref{wienermeasureexists} shows that given a substochastic transition function, there is at most one Wiener measure for each $x_0$.
If it exists, it has total mass $\W_{x_0}\left( C_{x_0}\left([0,T];S\right)\right)=\int_S p_T(z,x_0) d\mu(z)\le 1$.

Similarly, there is at most one conditional Wiener measure for each $x_0$ and $y_0$.
It has total mass $\W_{x_0}^{y_0}\left( C_{x_0}^{y_0}\left([0,T];S\right)\right)=p_T(y_0,x_0)$.
\end{rem}

\begin{rem}\label{wienercylinders}
Suppose the Wiener measure exists for the substochastic transition function $p$. 
We consider a cylindrical function $f=h\circ\pi_F:C_{x_0}\left([0,T];S\right) \subset S^{[0,T]}\to [-\infty,\infty]$ as in Remark~\ref{cylinderfunctions}.
Then we have 
\begin{align*}
&\int_{C_{x_0}\left([0,T];S\right)} f(w)\,d\W_{x_0}(w)
=\int_{C_{x_0}\left([0,T];S\right)} h(w(t_1),\ldots,w(t_n) )\,d\W_{x_0}(w) \\
& =\int_{S^n} h(x_1,\ldots,x_n)\,p_{t_n-t_{n-1}}(x_n,x_{n-1})\cdots p_{t_2-t_1}(x_2,x_1) p_{t_1}(x_1,x_0) d\mu(x_1)\cdots d\mu(x_n).
\end{align*}
Similarly, for a cylindrical function $f=h\circ\pi_F:C_{x_0}^{y_0}\left([0,T];S\right) \subset S^{[0,T]}\to [-\infty,\infty]$ we get
\begin{align*}
&\int_{C_{x_0}^{y_0}\left([0,T];S\right)} f(w)\,d\W_{x_0}^{y_0}(w)
=\int_{C_{x_0}^{y_0}\left([0,T];S\right)} h(w(t_1),\ldots,w(t_n) )\,d\W_{x_0}^{y_0}(w) \\
& =\int_{S^n} h(x_1,\ldots,x_n)\,p_{T-t_{n}}(y_0,x_{n})\cdots p_{t_2-t_1}(x_2,x_1) p_{t_1}(x_1,x_0) d\mu(x_1)\cdots d\mu(x_n).
\end{align*}
\end{rem}
\noindent
The next lemma says that the Wiener measure can be obtained from the conditional Wiener measure by integration over the endpoints.
\begin{lem}\label{disintegration}
Let $(S,\mathcal{B},\mu)$ be a measure space, let $T>0$,
and let $p$ be a substochastic transition function on $S$.
Suppose the induced Wiener measures $\W_{x_0}$ and the induced relative Wiener measures $\W_{x_0}^{y_0}$ exist for all $x_0,y_0\in S$.
Let $f:C_{x_0}\left([0,T];S\right)\to [-\infty,+\infty]$ be integrable with respect to $\W_{x_0}$ or nonnegative measurable.
Then we have
\begin{equation}\label{eq_disintegration}
\int_{C_{x_0}\left([0,T];S\right)} f(w)\,d\W_{x_0}(w)
=
\int_S \int_{C_{x_0}^{y_0}\left([0,T];S\right)} f(w)\,d\W_{x_0}^{y_0}(w) d\mu(y_0).
\end{equation}
\end{lem}
\begin{proof}
For every $t\in [0,T]$ the evaluation map $\pi_t:C_{x_0}([0,T];S)\to S$, defined by $w\mapsto w(t)$, is continuous and therefore measurable w.r.t.~$\mathcal{C}$.
This shows that
\[
C_{x_0}^{y_0}\left([0,T];S\right) = {\pi_T}^{-1}(y_0)\subset C_{x_0}\left([0,T];S\right)
\]
is a measurable subset, and hence the restriction of any measurable function $f$ on $C_{x_0}\left([0,T];S\right)$ yields a measurable function on $C_{x_0}^{y_0}\left([0,T];S\right)$.

The next argument is routine in measure theory; it is known as the {\it good sets principle}.
We put
\[
\mathcal{D}:=\left\{ A\in\mathcal{C}\,\mid\, \mbox{\eqref{eq_disintegration} holds for }f=\1_A\right\}
\]
and notice that $\mathcal{D}$ is a Dynkin system.\footnote{This means $\emptyset\in\mathcal{D}$, for any $A\in\mathcal{D}$ one has $A^c\in\mathcal{D}$ and for any sequence $(A_n)_{n\ge 1}$ of pairwise disjoint sets in $\mathcal{D}$ one has $\bigcup\limits_{n\ge 1}A_n\in \mathcal{D}$.}
A generator of the $\sigma$-algebra $\mathcal{C}$ is given by
\[
\mathcal{E}=\Big\{\bigcap_{i=1}^n \pi_{t_i}^{-1}(U_i)\subset C_{x_0}\left([0,T];S\right) \,\Big|\,n\in\N,\, 0 < t_1<\ldots < t_n\le T,\mbox{ and } U_1,\ldots, U_n\subset S\mbox{ open}\Big\}.
\]
From the formulas in Remark~\ref{wienercylinders} we get that $\mathcal{E}\subset\mathcal{D}$.
This implies $\mathcal{C}=\mathcal{D}$ since $\mathcal{E}$ is stable under $\cap$ und generates $\mathcal{C}$ as a $\sigma$-algebra.\footnote{Here we have used a standard fact from measure theory \cite[Satz~2.4]{Bauer_masstheorie}: Let a set system $\mathcal{E}$ be stable under $\cap$, then the Dynkin system generated by $\mathcal{E}$ coincides with the $\sigma$-algebra generated by $\mathcal{E}$.
From that we conclude in the above situation that $\mathcal{C}\subset\mathcal{D}$, and therefore $\mathcal{C}=\mathcal{D}$.}

By linearity of integrals it follows that \eqref{eq_disintegration} is true for step functions, i.e. functions of the form $\sum_{i=1}^k \alpha_i \1_{A_i}$ with $\alpha_i\in\R$ and $A_i\in\mathcal{C}$. 
If $f$ is nonnegative measurable we approximate $f$ by nonnegative step functions pointwise and monotonically from below, and monotone convergence shows that \eqref{eq_disintegration} also holds for nonnegative measurable functions.

Finally, let $f$ be integrable with respect to $\W_{x_0}$. 
We apply \eqref{eq_disintegration} to $|f|$ and get from integrability that, for $\mu$-almost all $y_0\in S$,
\[
\int_{C_{x_0}^{y_0}\left([0,T];S\right)} |f(w)|\,d\W_{x_0}^{y_0}(w) <\infty.
\]
We approximate $f$ by step functions $f_n$ such that $|f_n|\le|f|$ for any $n$ and $f_n\to f$ pointwise,
\eqref{eq_disintegration} holds for every $f_n$ and dominated convergence concludes the proof.
\end{proof}

\noindent
The next lemma is a slight generalization of the Lemma on p.~279 in \cite{ReedSimon} whose proof is sketched as exercise 65 in \cite[p.~347]{ReedSimon}.
It states that any given null set is avoided by almost all paths for almost all the time.
\begin{lem}\label{nullsetlemma}
Let $(S,\mathcal{B},\mu)$ be a measure space, let $x_0\in S$, let $T>0$ 
and let $p$ be a substochastic transition function on $S$.
Suppose the induced Wiener measure $\W_{x_0}$ on $C_{x_0}([0,T];S)$ exists.
Let $B\in\mathcal{B}$ be a null set, $\mu(B)=0$.
Denote the Lebesgue measure on $[0,T]$ by $\lambda$.
We put
\[
W_B= \left\{ w\in C_{x_0}([0,T];S)\,\mid \, \lambda(w^{-1}(B))=0 \right\}.
\]
Then its complement $C_{x_0}([0,T];S)\setminus W_B$ is a $\W_{x_0}$-null set.
\end{lem}
\begin{proof}
On $C_{x_0}([0,T];S)\times [0,T]$ we consider the product $\sigma$-algebra of $\mathcal{C}$ and the Borel $\sigma$-algebra of $[0,T]$.
For any $n\ge 1$ we define the map $F_n:C_{x_0}([0,T];S)\times [0,T]\to S$ by
\[
F_n(w,t)=w(\tfrac kn)\quad\mbox{ for }\;\tfrac kn\le t <\tfrac{k+1}n.
\]
In order to see that $F_n$ is measurable we argue as follows: For any $\tau\in[0,T]$ denote the evaluation map $\pi_\tau$ as in the proof of Lemma~\ref{disintegration}.
The product sets of the form $A\times [a,b]$, where $A\in \mathcal{C}$ and $0\le a<b\le T$, generate the $\sigma$-algebra of $C_{x_0}([0,T];S)\times [0,T]$.
Hence for any $C\in \mathcal{B}$ the preimage
\[
{F_n}^{-1}(C)=\bigcup_{k\ge 0} (\pi_{k/n})^{-1}(C)\times\left(\left[\tfrac kn, \tfrac{k+1}n\right)\cap [a,b]\right)
\]
is a measurable subset of $C_{x_0}([0,T];S)\times [0,T]$.
Therefore $F_n$ is a measurable map.

Now we consider the map $F:C_{x_0}([0,T];S)\times [0,T]\to S$ given by $(w,t)\mapsto w(t)$.
We note that $F$ is the pointwise limit of the sequence of measurable maps $(F_n)_{n\ge 1}$, and hence $F$ itself is a measurable map.
This implies that 
\[
F^{-1}(B)=\left\{(w,t) \,\mid\, w(t)\in B\right\}\subset C_{x_0}([0,T];S)\times [0,T]
\]
is a measurable subset.
Since $B$ is a null set, property \eqref{wienerbed} of the Wiener measure gives for every $t\in[0,T]$
\[
\W_{x_0}\left({\pi_t}^{-1}(B) \right)= \int_S\int_B p_{T-t}(x_2,x_1)p_t(x_1,x_0)d\mu(x_1)d\mu(x_2) =0.
\]
We denote the product measure by $\W_{x_0}\otimes\lambda$ and apply Fubini's Theorem twice:
\begin{align*}
0&= \int_0^T \W_{x_0}\left({\pi_t}^{-1}(B) \right)dt\\
&= \W_{x_0}\otimes\lambda\left(F^{-1}(B) \right)\\
&= \int_{C_{x_0}([0,T];S)} \lambda\left(w^{-1}(B) \right)d\W_{x_0}(w),
\end{align*}
which shows that $C_{x_0}([0,T];S)\setminus W_B=\{ w\,\mid\, \lambda (w^{-1}(B) )>0 \}$ is a $\W_{x_0}$-null set.
\end{proof}

\section{Wiener Measures on Riemannian Manifolds}\label{wienerheatkernel}

\noindent
From now on the metric measure space will be a connected Riemannian manifold $S$, possibly with nonempty boundary.
Let $\rho$ be the Riemannian distance function on $S$ and $d\mu$ the Riemannian volume measure induced by $g$.
The Laplace-Beltrami operator $\Delta$ acts on smooth functions with compact support in the interior of $S$.
In local coordinates, $\Delta$ is given by
\[
\Delta = \frac{1}{\sqrt{\det(g)}} \sum_{i,j}\frac{\partial}{\partial x^i}\left(\sqrt{\det(g)}\, g^{ij} \, \frac{\partial}{\partial x^j}\right).
\]
In case of a closed Riemannian manifold, $\Delta$ is essentially selfadjoint in the Hilbert space $L^2(S,d\mu)$ of square-integrable functions.
In general, there always exists a selfadjoint extension, known as the Friedrichs extension, because $\Delta$ is a nonpositive operator.
If $S$ is a compact Riemannian manifold with boundary, the Friedrichs extension coincides with the Laplace-Beltrami operator with Dirichlet boundary conditions imposed.
In the following we will always use the Friedrichs extension and denote it again by $\Delta$.

For $t>0$ the bounded selfadjoint operator $e^{t\Delta}$ on $L^2(S,d\mu)$ can be defined using functional calculus. 
By elliptic theory, $e^{t\Delta}$ is smoothing and its Schwartz kernel $p_t(x,y)$ depends smoothly on all variables $x,y\in S$ and $t>0$.
The kernel $p_t(x,y)$ is called the {\em heat kernel} because $e^{t\Delta}$ is the solution operator for the heat equation.\footnote{In stochastics it is customary to consider the kernel of $e^{t\Delta/2}$ instead of $e^{t\Delta}$. For all that follows this modification is irrelevant.}
It has the following properties:
\begin{eqnarray*}
p_t(x,y) &>& 0; \nonumber\\
\frac{\partial}{\partial t}p_t(x,y) &=& \Delta_x p_t(x,y); \nonumber\\
p_t(x,y) &=& p_t(y,x);\nonumber\\
p_{t+s}(x,y) &=& \int_M p_t(x,z) \,p_s(z,y)\,d\mu(z); \nonumber\\
\int_M p_t(x,z)\,d\mu(z) &\leq& 1 
\end{eqnarray*}
for all $x,y\in S$ and $t>0$.
In particular, the heat kernel is a substochastic transition function.
Moreover, the heat kernel approximates the delta function as $t\searrow 0$ in the sense that for any compactly supported continuous function $u:S\to \R$ and any $y$ in the interior of $S$ we have
$$
\lim_{t\searrow 0}\int_S u(z)p_t(z,y)\, d\mu(z) = u(y).
$$

\begin{dfn}
If the heat kernel $p$ of Riemannian manifold $S$ satisfies the conservation property
\[
\int_S p_t(y,x)d\mu(y)=1
\]
for some $x\in S$ and some $t>0$, then one calls $S$ {\em stochastically complete}.
\end{dfn}

\begin{rem}
If $S$ is a stochastically complete Riemannian manifold, the conservation property holds for all $x\in S$ and all $t>0$.
In \cite{Grigoryan} several criteria for stochastic completeness are discussed.
For example, geodesically complete manifolds with a lower Ricci curvature bound are stochastically complete.
This applies in particular to closed Riemannian manifolds.
\end{rem}

\subsection{Closed Riemannian manifolds}

We start with the simplest case where the manifold is compact and has no boundary.

\begin{prop}\label{closedmanifold}
Let $S$ be a closed connected Riemannian manifold.
Then its heat kernel $p_t(x,y)$ satisfies the estimate (\ref{WienerChentsov}) from Corollary~\ref{wienermeasureexists} for any $b\in\N$ and $a=2b+2$.
\end{prop}

\begin{proof}
There exists a $\delta>0$ such that $\rho(x,y)^2$ is a smooth function on the set of $(x,y)$ with $\rho(x,y) < 2\delta$.
We chose a smooth function $\tilde\rho:M\times M\to [0,\infty)$ such that $\tilde\rho(x,y)$ coincides with $\rho(x,y)^2$ if $\rho(x,y)\le \delta$ and $\tilde\rho\ge\rho^2$ everywhere.
It suffices to show
\begin{equation}\label{eqWclosed1}
\int_S \tilde\rho(z,y)^{a/2} \;p_{\tau}(z,y)\,d\mu(z)=\int_S \tilde\rho(z,y)^{b+1} \;p_{\tau}(z,y)\,d\mu(z) \le C\cdot \tau^{1+b}
\end{equation}
for all $y\in S$ and all positive $\tau$.

We fix $y\in S$ and put
$$
f(t) := \int_S \tilde\rho(z,y)^{b+1} \;p_{t}(z,y)\,d\mu(z).
$$
Since the heat kernel approximates the delta function as $t\searrow 0$ and $\tilde\rho(y,y) = \rho(y,y)^2=0$, we have
$$
f(t) \to 0 \quad\mbox{ as }t\searrow 0.
$$
We compute
\begin{align*}
\dot{f}(t) &= \int_S \tilde\rho(z,y)^{b+1} \;\frac{\partial p_{t}}{\partial t}(z,y)\,d\mu(z) \\
&= \int_S \tilde\rho(z,y)^{b+1} \;\Delta_zp_{t}(z,y)\,d\mu(z) \\
&= \int_S \Delta_z\big(\tilde\rho(z,y)^{b+1}\big) \;p_{t}(z,y)\,d\mu(z) \, .
\end{align*}
Here $\Delta_z$ denotes the Laplace-Beltrami operator applied to the $z$-variable.
Since $z\mapsto\tilde\rho(z,y)^{b+1}=\rho(z,y)^{2b+2}$ vanishes to order $2b+2$ at $z=y$, the function $z \mapsto \Delta_z\left(\tilde\rho(z,y)^{b+1}\right)$ also vanishes at $z=y$ and we get again
$$
\dot{f}(t) \to 0 \quad\mbox{ as }t\searrow 0.
$$
Inductively we get for the $k^\mathrm{th}$ derivative of $f$:
$$
f^{(k)}(t) = \int_S \Delta_z^k\big(\tilde\rho(z,y)^{b+1}\big) \;p_{t}(z,y)\,d\mu(z)
$$
and thus
\be \label{eqWclosed2}
f^{(k)}(t)\to 0 \quad\mbox{ as }t\searrow 0
\ee 
for all $k\le b$.
By compactness of $M$, there is a constant $C$ such that 
$$
\big|\Delta_z^{b+1}\big(\tilde\rho(z,y)^{b+1}\big)\big| \le C
$$
for all $y,z \in S$.
Hence 
\be \label{eqWclosed4}
|f^{(b+1)}(t)| \le \int_S \big|\Delta_z^{b+1}\big(\tilde\rho(z,y)^{b+1}\big)\big| \;p_{t}(z,y)\,d\mu(z) \le C.
\ee
Now \eqref{eqWclosed2} and \eqref{eqWclosed4} combine to give
$$
f(\tau) = \int_0^\tau \int_0^{t_1}\cdots \int_0^{t_b} f^{(b+1)}(t_{b+1})\,dt_{b+1}\cdots dt_1 \le C\cdot \tau^{b+1} \, ,
$$
thus proving inequality \eqref{eqWclosed1} with a constant $C$ independent of $y$ and $\tau$.
\end{proof}

\begin{rem}
Proposition~\ref{closedmanifold} and its proof carry over without changes to closed connected Riemannian \emph{orbifolds}.
See \cite{Chiang} for basics on the Laplacian and its heat kernel on orbifolds.
\end{rem}

By Corollary~\ref{wienermeasureexists}, this implies:

\begin{cor}
For any closed connected Riemannian manifold $S$ and any $x_0\in S$, the heat kernel induces a Wiener measure $\W_{x_0}$ on $\left(C_{x_0}\left([0,T];S\right),\mathcal{C}\right)$ .

For any $\theta\in(0,1/2)$ the set of H\"older continuous paths of order $\theta$ has full measure in $\left(C_{x_0}\left([0,T];S\right),\mathcal{C},\W_{x_0}\right)$.
\end{cor}

\begin{proof}
The statement on H\"older continuous paths is true for all $\theta\in(0,1/2)$ because $\frac{b}{a}=\frac{b}{2b+2} \to \frac12$ as $b\to \infty$.
\end{proof}

Next we check Condition~\eqref{WienerChentsovBridge} in order to apply Corollary~\ref{wienerbridgeexists}.

\begin{prop}\label{closedmanifoldbridge}
Let $S$ be a closed connected Riemannian manifold.
Then its heat kernel $p_t(x,y)$ satisfies the estimate (\ref{WienerChentsovBridge}) from Corollary~\ref{wienerbridgeexists} for any $b\in\N$ if the dimension $n$ of $S$ is even and $b\in\N+\frac12$ if $n$ is odd and $a=2b+n+2$.
\end{prop}

\begin{proof}
The proof of Propostion~\ref{closedmanifold} is based on repeated integration by parts and does not yield the required pointwise estimate.
Therefore we follow a different approach based on the asymptotic heat kernel expansion of Minakshisundaram and Pleijel \cite{MP}.
It says that there are smooth functions $a_j:S\times S \to \R$ such that for all $N\in\N$
$$
p_t(x,y) = (4\pi t)^{-n/2}\cdot\exp\left(-\frac{\rho(x,y)^2}{4t}\right)\cdot \sum_{j=0}^N a_j(x,y)t^j + \mathrm{O}(t^{N+1-n/2})
\quad\mbox{ as }t\searrow 0.
$$
The constant in the $\mathrm{O}(t^{N+1-n/2})$-term is uniform in $x,y\in S$.

Given $b$ let $a=2b+n+2$.
Putting $N:=b+\frac{n}{2}$ we get for all $x,y\in S$ and $t\in(0,T]$:
\begin{align*}
\rho(x,y)^a p_t(x,y)
&=
(4\pi t)^{-n/2}\cdot\rho(x,y)^{a}\cdot\exp\left(-\frac{\rho(x,y)^2}{4t}\right)\cdot \sum_{j=0}^N a_j(x,y)t^j + \mathrm{O}(t^{N+1-n/2}) \\
&\le
C_1\cdot \left(\frac{\rho(x,y)^2}{4t}\right)^{b+n/2+1}\cdot\exp\left(-\frac{\rho(x,y)^2}{4t}\right)\cdot t^{b+1}+ \mathrm{O}(t^{N+1-n/2}) \\
&\le
C_2\cdot t^{b+1} + \mathrm{O}(t^{b+1}) \\
&\le
C_3\cdot t^{b+1}.
\end{align*}
For the second to last inequality we used that the function $[0,\infty) \to \R$, $x\mapsto x^{b+n/2+1}\exp(-x)$, is bounded.
\end{proof}

Since again $\frac{b}{a}=\frac{b}{2b+n+2} \to \frac12$ as $b\to\infty$ we get, using Corollary~\ref{wienerbridgeexists},

\begin{cor}
For any closed connected Riemannian manifold $S$ and any $x_0,y_0\in S$, the heat kernel induces a conditional Wiener measure $\W_{x_0}^{y_0}$ on $\left(C_{x_0}^{y_0}\left([0,T];S\right),\mathcal{C}\right)$.

For any $\theta\in(0,1/2)$ the set of H\"older continuous paths of order $\theta$ has full measure in $\left(C_{x_0}^{y_0}\left([0,T];S\right),\mathcal{C},\W_{x_0}^{y_0}\right)$.\qed
\end{cor}

\begin{rem}
By integration over the end point, Proposition~\ref{closedmanifoldbridge} implies Proposition~\ref{closedmanifold} in a slightly weaker form concerning the conditions on the constants $a$ and $b$.
For our applications concerning the construction of Wiener measure this would be sufficient.
Nevertheless, we have included the direct proof of Proposition~\ref{closedmanifold} because it is more elementary and does not require any knowledge about heat kernel asymptotics.
\end{rem}

\subsection{Compact Riemannian manifolds with boundary}
Next we consider the case that the Riemannian manifold $S$ is compact and connected and has a nonempty smooth boundary.
The first lemma says that if $S$ is contained in a larger Riemannian manifold of equal dimension, then it does not matter whether we use the instrinsic distance function of $S$ or the restriction of the distance function on the larger manifold.

\begin{lem}\label{lem:MetrikVergleich}
Let $M$ be a connected Riemannian manifold and let $\Omega\subset M$ be a connected relatively compact open subset with smooth boundary.
Denote by $\rho_M$ the Riemannian distance function on $M$ and by $\rhoO$ the one on $\overline\Omega$.

Then there is a constant $C>0$ such that 
$$
\rho_M(x,y) \le \rhoO(x,y) \le C\cdot\rho_M(x,y)
$$
for all $x,y\in\overline\Omega$.
\end{lem}

For the proof see Appendix~\ref{app:LemmaProof}.

\begin{prop}\label{Dirichletmanifold}
Let $S$ be a compact connected Riemannian manifold with smooth boundary.
Then its heat kernel $p_t(x,y)$ (for Dirichlet boundary conditions) satisfies estimates (\ref{WienerChentsov}) and \eqref{WienerChentsovBridge} with the same exponents $a$ and $b$ as in Proposition~\ref{closedmanifold} and Proposition~\ref{closedmanifoldbridge}, respectively.
\end{prop}

\begin{proof}
We isometrically embed $S$ into a closed Riemannian manifold $M$ of equal dimension.
For instance, for $M$ we can take the topological double of $S$, i.e., the closed manifold obtained by gluing two copies of $S$ along the boundary, and then choose a smooth metric on $M$ such that $S\subset M$ inherits its original metric from $M$.
\begin{center}
\psset{xunit=.7pt,yunit=.7pt,runit=.7pt}
\begin{pspicture}(-40,40)(420,260)
{
\newrgbcolor{curcolor}{0 0 0}
\pscustom[linewidth=1,linecolor=curcolor]
{
\newpath
\moveto(56.692913,134.6456707)
\curveto(77.952756,177.1653507)(43.032186,190.4949007)(56.692913,212.5984207)
\curveto(65.023727,226.0779607)(83.366446,233.8582707)(99.212598,233.8582707)
\curveto(115.05875,233.8582707)(132.94241,225.7832107)(141.73228,212.5984207)
\curveto(155.90551,191.3385807)(120.47244,177.1653507)(141.73228,134.6456707)
}
}
{
\newrgbcolor{curcolor}{0 0 0}
\pscustom[linewidth=1,linecolor=curcolor]
{
\newpath
\moveto(70.866142,205.5118107)
\curveto(85.03937,184.2519707)(113.38583,184.2519707)(127.55906,205.5118107)
}
}
{
\newrgbcolor{curcolor}{0 0 0}
\pscustom[linewidth=1,linecolor=curcolor]
{
\newpath
\moveto(77.952756,198.4252007)
\lineto(77.952756,198.4252007)
\closepath
}
}
{
\newrgbcolor{curcolor}{0 0 0}
\pscustom[linewidth=1,linecolor=curcolor]
{
\newpath
\moveto(77.952756,198.4252007)
\curveto(99.212598,212.5984207)(99.212598,212.5984207)(120.47244,198.4252007)
}
}
{
\newrgbcolor{curcolor}{0.50196081 0.50196081 0.50196081}
\pscustom[linestyle=none,fillstyle=solid,fillcolor=lightgray]
{
\newpath
\moveto(58.992619,137.6335807)
\curveto(58.992619,137.8223707)(59.331417,138.7612807)(59.745503,139.7200307)
\curveto(62.616506,146.3673807)(64.239158,154.2082207)(64.23007,161.3899707)
\curveto(64.21969,169.5935707)(63.15813,173.9168307)(57.650439,188.1858607)
\curveto(54.172841,197.1954307)(53.519081,201.9157107)(55.029329,207.1107807)
\curveto(57.204948,214.5946507)(65.631162,222.8775307)(75.987752,227.7127007)
\curveto(91.596171,234.9997707)(108.26291,234.7577407)(123.46153,227.0232807)
\curveto(132.19684,222.5779507)(139.14453,216.2248707)(142.33891,209.7614807)
\curveto(144.01347,206.3732507)(144.09137,206.0051407)(144.0868,201.5022307)
\curveto(144.0818,196.3736507)(144.1627,196.6627907)(138.81718,182.7233207)
\curveto(135.22333,173.3517007)(134.31695,168.9944207)(134.33074,161.1555307)
\curveto(134.34344,153.9562407)(134.95633,150.3548007)(137.31532,143.6186407)
\curveto(138.09171,141.4016507)(138.84828,139.2285407)(138.99658,138.7895007)
\curveto(139.21288,138.1491807)(139.00158,138.0100307)(137.92799,138.0862007)
\curveto(137.19196,138.1384007)(136.44384,138.3846407)(136.26551,138.6333407)
\curveto(136.08718,138.8820407)(135.45421,139.1186407)(134.85891,139.1591307)
\curveto(134.26362,139.1996307)(133.40733,139.3879307)(132.95604,139.5776207)
\curveto(132.50476,139.7673107)(131.83268,139.8063007)(131.46253,139.6642207)
\curveto(131.09237,139.5221707)(130.72314,139.6014207)(130.64201,139.8402807)
\curveto(130.56091,140.0791507)(129.17759,140.3310307)(127.56803,140.4000207)
\curveto(125.95848,140.4690207)(124.63979,140.6903007)(124.63763,140.8917807)
\curveto(124.63563,141.0932607)(123.94798,141.2472107)(123.10988,141.2338807)
\curveto(122.27179,141.2205807)(120.10915,141.3443807)(118.30402,141.5090407)
\curveto(116.49889,141.6737007)(111.79253,141.8656707)(107.84544,141.9356407)
\curveto(103.89835,142.0056407)(100.54011,142.1916607)(100.38268,142.3490907)
\curveto(99.879712,142.8520607)(98.46378,142.6559607)(98.217599,142.0492507)
\curveto(98.033689,141.5960107)(97.877312,141.6096207)(97.52765,142.1093507)
\curveto(97.271683,142.4751507)(96.823365,142.5996607)(96.494361,142.3963307)
\curveto(96.174727,142.1987907)(93.091795,142.0488807)(89.6434,142.0632007)
\curveto(86.195004,142.0775007)(83.37359,141.9309207)(83.37359,141.7374007)
\curveto(83.37359,141.5438907)(82.193998,141.4563907)(80.752274,141.5429707)
\curveto(79.31055,141.6295707)(77.924734,141.5729707)(77.672683,141.4171407)
\curveto(77.420632,141.2613707)(76.015607,141.0447107)(74.550405,140.9356907)
\curveto(73.085203,140.8266607)(71.569918,140.6715007)(71.183104,140.5908807)
\curveto(70.796291,140.5102807)(69.424861,140.2774407)(68.135483,140.0735007)
\curveto(66.846105,139.8695607)(65.58017,139.5778507)(65.322294,139.4252507)
\curveto(65.064418,139.2726607)(64.398855,139.2065107)(63.843265,139.2782507)
\curveto(63.287675,139.3499507)(62.695003,139.1852507)(62.526216,138.9121307)
\curveto(62.357429,138.6390307)(61.746508,138.4155807)(61.168613,138.4155807)
\curveto(60.590718,138.4155807)(59.864708,138.1623907)(59.555257,137.8529407)
\curveto(59.245806,137.5434907)(58.992619,137.4447707)(58.992619,137.6335707)
\closepath
\moveto(108.92672,190.3468007)
\curveto(112.56723,191.5141407)(115.2753,192.9021007)(118.83957,195.4274107)
\curveto(121.80007,197.5249507)(126.3074,202.1619607)(127.43648,204.2716707)
\curveto(128.60349,206.4522507)(127.40308,206.1886907)(125.37062,203.8180907)
\curveto(124.2671,202.5309907)(122.85864,200.8942507)(122.24072,200.1808907)
\lineto(121.11721,198.8838707)
\lineto(119.00732,200.3166207)
\curveto(113.98934,203.7241507)(106.56806,207.9663707)(104.09869,208.8388207)
\curveto(102.07656,209.5532507)(100.66486,209.7301807)(98.306488,209.5647607)
\curveto(94.425971,209.2925707)(91.078386,207.7793107)(83.157519,202.7167007)
\lineto(77.31876,198.9848707)
\lineto(75.540311,200.9427507)
\curveto(74.562163,202.0195807)(73.200624,203.6354507)(72.514666,204.5335907)
\curveto(71.630162,205.6916807)(71.118526,206.0230607)(70.755285,205.6731107)
\curveto(70.123559,205.0645007)(71.503089,203.0986607)(75.302676,199.1930507)
\curveto(79.671246,194.7025807)(84.949261,191.6049307)(91.050161,189.9508807)
\curveto(95.416991,188.7669707)(104.63663,188.9711607)(108.92672,190.3468007)
\lineto(108.92672,190.3468007)
\closepath
}
}
{
\newrgbcolor{curcolor}{0.50196081 0.50196081 0.50196081}
\pscustom[linestyle=none,fillstyle=solid,fillcolor=gray]
{
\newpath
\moveto(94.981923,127.8835507)
\curveto(94.915293,127.9404507)(91.590446,128.1267707)(87.593373,128.2976407)
\curveto(78.051211,128.7055407)(77.169002,128.7885707)(70.010942,129.9524107)
\curveto(63.412164,131.0253207)(61.51377,131.5199207)(59.002937,132.8204107)
\curveto(56.938392,133.8897407)(56.821144,134.2890607)(58.150227,135.7245307)
\curveto(60.062454,137.7898207)(68.95356,139.8194207)(79.153806,140.5190807)
\curveto(81.603625,140.6871207)(84.662968,140.9227207)(85.952346,141.0426307)
\curveto(91.7713,141.5838207)(113.99951,141.2367007)(119.47618,140.5191107)
\curveto(120.37875,140.4008507)(122.27765,140.1793407)(123.69596,140.0268607)
\curveto(127.2365,139.6462207)(128.96213,139.4173107)(129.79121,139.2182907)
\curveto(130.17802,139.1254907)(131.65494,138.7960307)(133.07326,138.4862707)
\curveto(139.59058,137.0628907)(142.51551,135.0350407)(140.45788,133.3665107)
\curveto(139.17574,132.3268307)(132.22716,130.4217607)(127.21245,129.7350607)
\curveto(123.21257,129.1873307)(121.0702,128.9183307)(119.53365,128.7708807)
\curveto(115.37515,128.3718407)(95.261464,127.6449107)(94.981923,127.8835507)
\lineto(94.981923,127.8835507)
\closepath
}
}
{
\newrgbcolor{curcolor}{0.70588237 0 0}
\pscustom[linewidth=0.46849564,linecolor=black,linestyle=dashed,dash=0.93699127 0.46849564]
{
\newpath
\moveto(94.981923,127.8835507)
\curveto(94.915293,127.9404507)(91.590446,128.1267707)(87.593373,128.2976407)
\curveto(78.051211,128.7055407)(77.169002,128.7885707)(70.010942,129.9524107)
\curveto(63.412164,131.0253207)(61.51377,131.5199207)(59.002937,132.8204107)
\curveto(56.938392,133.8897407)(56.821144,134.2890607)(58.150227,135.7245307)
\curveto(60.062454,137.7898207)(68.95356,139.8194207)(79.153806,140.5190807)
\curveto(81.603625,140.6871207)(84.662968,140.9227207)(85.952346,141.0426307)
\curveto(91.7713,141.5838207)(113.99951,141.2367007)(119.47618,140.5191107)
\curveto(120.37875,140.4008507)(122.27765,140.1793407)(123.69596,140.0268607)
\curveto(127.2365,139.6462207)(128.96213,139.4173107)(129.79121,139.2182907)
\curveto(130.17802,139.1254907)(131.65494,138.7960307)(133.07326,138.4862707)
\curveto(139.59058,137.0628907)(142.51551,135.0350407)(140.45788,133.3665107)
\curveto(139.17574,132.3268307)(132.22716,130.4217607)(127.21245,129.7350607)
\curveto(123.21257,129.1873307)(121.0702,128.9183307)(119.53365,128.7708807)
\curveto(115.37515,128.3718407)(95.261464,127.6449107)(94.981923,127.8835507)
\lineto(94.981923,127.8835507)
\closepath
}
}

\rput(200,0){
{
\newrgbcolor{curcolor}{0 0 0}
\pscustom[linewidth=1,linecolor=curcolor]
{
\newpath
\moveto(56.692913,134.6456707)
\curveto(77.952756,177.1653507)(43.032186,190.4949007)(56.692913,212.5984207)
\curveto(65.023727,226.0779607)(83.366446,233.8582707)(99.212598,233.8582707)
\curveto(115.05875,233.8582707)(132.94241,225.7832107)(141.73228,212.5984207)
\curveto(155.90551,191.3385807)(120.47244,177.1653507)(141.73228,134.6456707)
}
}
{
\newrgbcolor{curcolor}{0 0 0}
\pscustom[linewidth=1,linecolor=curcolor]
{
\newpath
\moveto(70.866142,205.5118107)
\curveto(85.03937,184.2519707)(113.38583,184.2519707)(127.55906,205.5118107)
}
}
{
\newrgbcolor{curcolor}{0 0 0}
\pscustom[linewidth=1,linecolor=curcolor]
{
\newpath
\moveto(77.952756,198.4252007)
\lineto(77.952756,198.4252007)
\closepath
}
}
{
\newrgbcolor{curcolor}{0 0 0}
\pscustom[linewidth=1,linecolor=curcolor]
{
\newpath
\moveto(77.952756,198.4252007)
\curveto(99.212598,212.5984207)(99.212598,212.5984207)(120.47244,198.4252007)
}
}
{
\newrgbcolor{curcolor}{0.50196081 0.50196081 0.50196081}
\pscustom[linestyle=none,fillstyle=solid,fillcolor=lightgray]
{
\newpath
\moveto(58.992619,137.6335807)
\curveto(58.992619,137.8223707)(59.331417,138.7612807)(59.745503,139.7200307)
\curveto(62.616506,146.3673807)(64.239158,154.2082207)(64.23007,161.3899707)
\curveto(64.21969,169.5935707)(63.15813,173.9168307)(57.650439,188.1858607)
\curveto(54.172841,197.1954307)(53.519081,201.9157107)(55.029329,207.1107807)
\curveto(57.204948,214.5946507)(65.631162,222.8775307)(75.987752,227.7127007)
\curveto(91.596171,234.9997707)(108.26291,234.7577407)(123.46153,227.0232807)
\curveto(132.19684,222.5779507)(139.14453,216.2248707)(142.33891,209.7614807)
\curveto(144.01347,206.3732507)(144.09137,206.0051407)(144.0868,201.5022307)
\curveto(144.0818,196.3736507)(144.1627,196.6627907)(138.81718,182.7233207)
\curveto(135.22333,173.3517007)(134.31695,168.9944207)(134.33074,161.1555307)
\curveto(134.34344,153.9562407)(134.95633,150.3548007)(137.31532,143.6186407)
\curveto(138.09171,141.4016507)(138.84828,139.2285407)(138.99658,138.7895007)
\curveto(139.21288,138.1491807)(139.00158,138.0100307)(137.92799,138.0862007)
\curveto(137.19196,138.1384007)(136.44384,138.3846407)(136.26551,138.6333407)
\curveto(136.08718,138.8820407)(135.45421,139.1186407)(134.85891,139.1591307)
\curveto(134.26362,139.1996307)(133.40733,139.3879307)(132.95604,139.5776207)
\curveto(132.50476,139.7673107)(131.83268,139.8063007)(131.46253,139.6642207)
\curveto(131.09237,139.5221707)(130.72314,139.6014207)(130.64201,139.8402807)
\curveto(130.56091,140.0791507)(129.17759,140.3310307)(127.56803,140.4000207)
\curveto(125.95848,140.4690207)(124.63979,140.6903007)(124.63763,140.8917807)
\curveto(124.63563,141.0932607)(123.94798,141.2472107)(123.10988,141.2338807)
\curveto(122.27179,141.2205807)(120.10915,141.3443807)(118.30402,141.5090407)
\curveto(116.49889,141.6737007)(111.79253,141.8656707)(107.84544,141.9356407)
\curveto(103.89835,142.0056407)(100.54011,142.1916607)(100.38268,142.3490907)
\curveto(99.879712,142.8520607)(98.46378,142.6559607)(98.217599,142.0492507)
\curveto(98.033689,141.5960107)(97.877312,141.6096207)(97.52765,142.1093507)
\curveto(97.271683,142.4751507)(96.823365,142.5996607)(96.494361,142.3963307)
\curveto(96.174727,142.1987907)(93.091795,142.0488807)(89.6434,142.0632007)
\curveto(86.195004,142.0775007)(83.37359,141.9309207)(83.37359,141.7374007)
\curveto(83.37359,141.5438907)(82.193998,141.4563907)(80.752274,141.5429707)
\curveto(79.31055,141.6295707)(77.924734,141.5729707)(77.672683,141.4171407)
\curveto(77.420632,141.2613707)(76.015607,141.0447107)(74.550405,140.9356907)
\curveto(73.085203,140.8266607)(71.569918,140.6715007)(71.183104,140.5908807)
\curveto(70.796291,140.5102807)(69.424861,140.2774407)(68.135483,140.0735007)
\curveto(66.846105,139.8695607)(65.58017,139.5778507)(65.322294,139.4252507)
\curveto(65.064418,139.2726607)(64.398855,139.2065107)(63.843265,139.2782507)
\curveto(63.287675,139.3499507)(62.695003,139.1852507)(62.526216,138.9121307)
\curveto(62.357429,138.6390307)(61.746508,138.4155807)(61.168613,138.4155807)
\curveto(60.590718,138.4155807)(59.864708,138.1623907)(59.555257,137.8529407)
\curveto(59.245806,137.5434907)(58.992619,137.4447707)(58.992619,137.6335707)
\closepath
\moveto(108.92672,190.3468007)
\curveto(112.56723,191.5141407)(115.2753,192.9021007)(118.83957,195.4274107)
\curveto(121.80007,197.5249507)(126.3074,202.1619607)(127.43648,204.2716707)
\curveto(128.60349,206.4522507)(127.40308,206.1886907)(125.37062,203.8180907)
\curveto(124.2671,202.5309907)(122.85864,200.8942507)(122.24072,200.1808907)
\lineto(121.11721,198.8838707)
\lineto(119.00732,200.3166207)
\curveto(113.98934,203.7241507)(106.56806,207.9663707)(104.09869,208.8388207)
\curveto(102.07656,209.5532507)(100.66486,209.7301807)(98.306488,209.5647607)
\curveto(94.425971,209.2925707)(91.078386,207.7793107)(83.157519,202.7167007)
\lineto(77.31876,198.9848707)
\lineto(75.540311,200.9427507)
\curveto(74.562163,202.0195807)(73.200624,203.6354507)(72.514666,204.5335907)
\curveto(71.630162,205.6916807)(71.118526,206.0230607)(70.755285,205.6731107)
\curveto(70.123559,205.0645007)(71.503089,203.0986607)(75.302676,199.1930507)
\curveto(79.671246,194.7025807)(84.949261,191.6049307)(91.050161,189.9508807)
\curveto(95.416991,188.7669707)(104.63663,188.9711607)(108.92672,190.3468007)
\lineto(108.92672,190.3468007)
\closepath
}
}
{
\newrgbcolor{curcolor}{0 0 0}
\pscustom[linewidth=1,linecolor=curcolor]
{
\newpath
\moveto(56.692913,134.6456707)
\curveto(56.692913,144.5454707)(141.73228,144.0766107)(141.73228,134.6456707)
}
}
{
\newrgbcolor{curcolor}{0 0 0}
\pscustom[linewidth=1,linecolor=curcolor]
{
\newpath
\moveto(56.692913,134.6456707)
\curveto(42.519685,106.2992107)(21.259842,63.7795307)(99.212598,63.7795307)
\curveto(177.16535,63.7795307)(155.90551,106.2992107)(141.73228,134.6456707)
}
}
{
\newrgbcolor{curcolor}{0 0 0}
\pscustom[linewidth=1,linecolor=curcolor]
{
\newpath
\moveto(70.866142,106.2992107)
\curveto(92.125984,85.0393707)(106.29921,85.0393707)(127.55906,106.2992107)
}
}
{
\newrgbcolor{curcolor}{0 0 0}
\pscustom[linewidth=1,linecolor=curcolor]
{
\newpath
\moveto(77.952756,99.2126007)
\curveto(99.212598,113.3858307)(99.212599,113.3858207)(120.47244,99.2126007)
}
}}
\rput(100,170){\psframebox*[framearc=.6]{$S$}}
\rput(100,115){\psframebox*[framearc=.6]{$\partial S$}}
\rput(350,170){\psframebox*[framearc=.6]{$M$}}
\end{pspicture}

\bildnummer
\end{center}

Let $q$ be the heat kernel of $M$, let $\rho_M$ be the Riemannian distance function on $M$ and $\rho_S$ the one on $S$.
By Proposition~\ref{closedmanifoldbridge}, we have \eqref{WienerChentsovBridge} for $\rho_M$ and $q$, i.e.,
$$
\rho_M(z,y)^a \;q_{\tau}(z,y) \le C\cdot \tau^{1+b}.
$$
The maximum principle implies that $p\le q$, see \cite[Thm.~1 on p.~181]{Chavel}.
Lemma~\ref{lem:MetrikVergleich} says $\rho_S\le C\cdot\rho_M$.
Hence \eqref{WienerChentsovBridge} also holds for $\rho_S$ and $p$ with the same exponents $a$ and $b$.

The argument for estimate \eqref{WienerChentsov} is the same.
\end{proof}

As in the closed case, we find

\begin{cor}\label{Dirichletwienermeasure}
Let $S$ be a compact connected Riemannian manifold with smooth boundary, let $x_0,y_0\in S$.
Then the heat kernel (for Dirichlet boundary conditions) induces a Wiener measure $\W_{x_0}$ on $\left(C_{x_0}\left([0,T];S\right),\mathcal{C}\right)$ and a conditional Wiener measure $\W_{x_0}^{y_0}$ on $\left(C_{x_0}^{y_0}\left([0,T];S\right),\mathcal{C}\right)$.

In both cases, the set of H\"older continuous paths of any order $\theta\in(0,1/2)$ has full measure.
\end{cor}

\subsection{Open Riemannian manifolds}
Now we pass to arbitrary connected Riemannian manifolds.
Note that geodesic completeness is not assumed.

The heat kernel on an arbitrary Riemannian manifold can be characterized as follows \cite{Dodziuk}:
Let $S$ be a connected Riemannian manifold of dimension $m$.
Let $S^i$ be an exhaustion by compact connected $m$-dimensional submanifolds with smooth boundary.

\begin{center}
\psset{xunit=.5pt,yunit=.5pt,runit=.5pt}
\begin{pspicture}(300,340.1574707)
{
\newrgbcolor{curcolor}{0 0 0}
\pscustom[linewidth=1,linecolor=curcolor]
{
\newpath
\moveto(85.03937,162.9921307)
\curveto(106.29921,205.5118107)(70.866142,198.4252007)(77.952756,219.6850407)
\curveto(85.03937,233.8582707)(115.25433,235.1953707)(120.47244,219.6850407)
\curveto(127.55906,198.4252007)(92.125987,205.5118107)(113.38583,162.9921307)
}
}
{
\newrgbcolor{curcolor}{0 0 0}
\pscustom[linewidth=1,linecolor=curcolor]
{
\newpath
\moveto(85.03937,219.6850407)
\curveto(99.212598,202.4827507)(99.212598,207.5597907)(113.38583,219.6850407)
}
}
{
\newrgbcolor{curcolor}{0 0 0}
\pscustom[linewidth=1,linecolor=curcolor]
{
\newpath
\moveto(77.952756,198.4252007)
\lineto(77.952756,198.4252007)
\closepath
}
}
{
\newrgbcolor{curcolor}{0 0 0}
\pscustom[linewidth=1,linecolor=curcolor]
{
\newpath
\moveto(92.125984,212.5984207)
\curveto(88.615451,212.8646307)(99.212598,226.7716507)(106.29921,212.5984207)
}
}
{
\newrgbcolor{curcolor}{0 0 0}
\pscustom[linewidth=0.88163066,linecolor=curcolor]
{
\newpath
\moveto(186.82892415,113.3858307)
\curveto(248.03149415,113.3858307)(248.03149415,132.51968179)(248.03149415,151.65355288)
\curveto(248.03149415,170.78740397)(199.06943415,177.16535433)(180.70866415,177.16535433)
\curveto(131.74659415,177.16535433)(113.38582415,164.40945361)(113.38582415,151.65355288)
\curveto(113.38582415,132.51968179)(125.62634415,113.3858307)(186.82892415,113.3858307)
\closepath
}
}
{
\newrgbcolor{curcolor}{0 0 0}
\pscustom[linewidth=1,linecolor=curcolor]
{
\newpath
\moveto(92.125984,184.2519707)
\curveto(42.519685,177.1653507)(21.259843,113.3858307)(99.212598,85.0393707)
\curveto(152.90686,65.5141907)(296.1415,67.1670707)(283.46457,155.9055107)
\curveto(276.37795,205.5118107)(134.64567,184.2519707)(106.29921,184.2519707)
}
}
{
\newrgbcolor{curcolor}{0 0 0}
\pscustom[linewidth=1,linecolor=curcolor]
{
\newpath
\moveto(162.99213,248.0314957)
\curveto(42.311435,248.0314957)(7.0866142,219.6850407)(7.0866142,155.9055107)
\curveto(7.0866142,106.2992107)(10.248199,56.6929107)(191.33858,56.6929107)
\curveto(248.0315,56.6929107)(311.81102,92.1259807)(311.81102,134.6456707)
\curveto(311.81102,219.6850407)(219.68504,248.0314957)(162.99213,248.0314957)
\closepath
}
}
{
\newrgbcolor{curcolor}{0 0 0}
\pscustom[linewidth=1,linecolor=curcolor]
{
\newpath
\moveto(-14.173228,297.6377957)
\curveto(-38.682825,171.0040507)(-35.433071,110.0908807)(-35.433071,77.9527607)
\lineto(-35.433071,42.5196807)
\lineto(347.24409,28.3464607)
\lineto(347.24409,28.3464607)
\lineto(347.24409,28.3464607)
\moveto(368.50394,276.3779527)
\lineto(368.50394,276.3779527)
\moveto(368.50394,276.3779527)
\curveto(347.24409,198.4252007)(339.17174,112.9185007)(347.24409,28.3464607)
}
}
{
\newrgbcolor{curcolor}{0 0 0}
\pscustom[linewidth=1,linecolor=curcolor]
{
\newpath
\moveto(368.50394,276.3779527)
\lineto(368.50394,276.3779527)
\closepath
}
}
{
\newrgbcolor{curcolor}{0 0 0}
\pscustom[linewidth=1,linecolor=curcolor]
{
\newpath
\moveto(368.50394,276.3779527)
\lineto(-14.173228,297.6377957)
\lineto(-14.173228,297.6377957)
}
}
{
\newrgbcolor{curcolor}{0.2 0.2 0.2}
\pscustom[linestyle=none,fillstyle=solid,fillcolor=darkgray]
{
\newpath
\moveto(176.25972,114.3556807)
\curveto(161.05859,114.9725507)(150.4687,116.6361107)(140.39156,119.9902007)
\curveto(124.16866,125.3898707)(116.04848,133.9813107)(114.34911,147.5440607)
\curveto(113.75015,152.3244407)(114.29041,155.4660207)(116.24012,158.5404207)
\curveto(121.94009,167.5283507)(136.46444,173.1916707)(159.82051,175.5332607)
\curveto(168.71488,176.4249807)(186.27889,176.5532307)(193.60772,175.7799707)
\curveto(217.08686,173.3027107)(234.53191,167.9491007)(242.43761,160.7948607)
\curveto(247.2089,156.4770907)(247.86148,154.1762507)(247.10947,144.3232907)
\curveto(246.47586,136.0218407)(245.17823,132.1209607)(241.6377,127.8744007)
\curveto(235.04023,119.9613007)(220.77246,115.3843207)(200.15557,114.5672807)
\curveto(196.42532,114.4194607)(191.57988,114.2278907)(189.38794,114.1415807)
\curveto(187.19599,114.0552807)(181.2883,114.1515807)(176.25972,114.3556807)
\closepath
}
}
{
\newrgbcolor{curcolor}{0.50196081 0.50196081 0.50196081}
\pscustom[linestyle=none,fillstyle=solid,fillcolor=gray]
{
\newpath
\moveto(144.14248,76.7981507)
\curveto(121.08077,79.0707507)(104.51432,83.0227407)(89.686411,89.7888807)
\curveto(67.250656,100.0265607)(53.350042,114.1676407)(49.427536,130.7442007)
\curveto(48.252973,135.7079207)(48.346245,143.8514407)(49.633791,148.7521907)
\curveto(53.762254,164.4662707)(67.483036,177.2736807)(85.299947,182.0441007)
\curveto(87.749765,182.7000307)(90.041856,183.2372307)(90.39348,183.2378907)
\curveto(91.786991,183.2398907)(89.453775,174.4362607)(86.187165,167.3657007)
\curveto(84.511382,163.7384707)(84.291793,162.6090207)(85.262369,162.6090207)
\curveto(85.70345,162.6090207)(88.782674,169.2077307)(90.07577,172.9240407)
\curveto(92.493352,179.8720707)(93.12627,187.6643607)(91.637169,192.1475007)
\curveto(90.433958,195.7699307)(89.05855,197.7905907)(84.088117,203.2380907)
\curveto(78.514292,209.3468907)(78.053592,210.2166607)(78.041849,214.6530107)
\curveto(78.033949,217.6564507)(78.221045,218.4901607)(79.324553,220.3673707)
\curveto(83.522367,227.5083907)(95.976932,231.6722007)(106.48029,229.4460907)
\curveto(115.35985,227.5641307)(120.46481,222.0272307)(120.46481,214.2783407)
\curveto(120.46481,210.5937307)(119.25419,208.4314507)(114.11106,202.9299007)
\curveto(111.68808,200.3380607)(109.18041,197.3212407)(108.53848,196.2258707)
\curveto(104.53003,189.3859707)(105.27392,179.2807207)(110.67554,167.1950407)
\curveto(112.93358,162.1428807)(112.83526,162.2937607)(113.46135,162.9198507)
\curveto(113.79245,163.2509507)(113.40575,164.6269807)(112.23053,167.2999807)
\curveto(111.28806,169.4435407)(110.10735,172.3770707)(109.60671,173.8189407)
\curveto(108.77746,176.2072507)(107.33659,182.2057507)(107.33659,183.2697207)
\curveto(107.33659,183.5107407)(108.54978,183.7090307)(110.03257,183.7103607)
\curveto(111.51535,183.7113607)(121.2736,184.3321407)(131.71756,185.0891107)
\curveto(166.17532,187.5865807)(168.69097,187.6938207)(192.43556,187.6773607)
\curveto(219.85183,187.6583607)(231.10647,186.7679107)(245.97182,183.4416507)
\curveto(263.66076,179.4835907)(275.64053,172.4027607)(280.31508,163.1425307)
\curveto(282.91004,158.0019307)(284.03414,148.8041207)(283.16636,139.8123107)
\curveto(282.3349,131.1969407)(280.27132,124.3037707)(276.51615,117.5979907)
\curveto(265.15679,97.3130707)(238.342,83.1829407)(201.34399,77.9858007)
\curveto(190.6295,76.4807307)(187.29326,76.2775607)(170.39891,76.1012807)
\curveto(155.9464,75.9504807)(151.56404,76.0667807)(144.14248,76.7981507)
\lineto(144.14248,76.7981507)
\closepath
\moveto(211.19015,114.0598307)
\curveto(218.71612,115.1140207)(222.47147,115.9339407)(227.56063,117.6340507)
\curveto(234.88305,120.0802107)(239.77356,123.1779007)(243.00706,127.4179407)
\curveto(247.08653,132.7672907)(248.46739,137.9631407)(248.46032,147.9372207)
\curveto(248.45632,153.7678207)(248.34623,154.6344707)(247.34234,156.7482107)
\curveto(246.24629,159.0559807)(243.10253,162.4210707)(240.17243,164.4228807)
\curveto(231.66681,170.2338107)(213.00477,175.2299807)(192.55528,177.1708607)
\curveto(184.85099,177.9020907)(164.19844,177.4996907)(155.62967,176.4514107)
\curveto(133.71684,173.7706307)(118.7544,167.1294207)(114.36999,158.1379207)
\curveto(112.86181,155.0449707)(112.53664,148.6296707)(113.63728,143.6824907)
\curveto(117.69586,125.4399607)(135.25093,116.1249907)(171.10221,113.1908007)
\curveto(172.26265,113.0958007)(180.49123,113.0751807)(189.38794,113.1449007)
\curveto(201.41168,113.2391007)(207.00798,113.4740007)(211.19015,114.0598107)
\lineto(211.19015,114.0597907)
\closepath
\moveto(103.7979,210.3753107)
\curveto(105.24051,211.4645007)(107.52358,213.3730907)(108.87138,214.6166207)
\curveto(110.21919,215.8601607)(111.95491,217.4173407)(112.72854,218.0770107)
\curveto(113.81916,219.0070007)(114.00996,219.4036607)(113.57791,219.8428207)
\curveto(113.1444,220.2834707)(112.26414,219.7255507)(109.6132,217.3299607)
\lineto(106.20572,214.2507007)
\lineto(105.14579,215.8105307)
\curveto(103.9869,217.5159907)(101.74398,219.0475007)(99.699582,219.5293007)
\curveto(97.077104,220.1473407)(93.407253,218.2525307)(91.395189,215.2415907)
\curveto(90.696139,214.1955007)(90.676612,214.2072907)(88.161313,217.1930607)
\curveto(86.593642,219.0539607)(85.381853,220.1014707)(84.976659,219.9459807)
\curveto(84.037279,219.5855107)(84.19899,219.3386907)(88.327855,214.8311707)
\curveto(93.332102,209.3679807)(95.547428,207.9595107)(98.746615,208.2070707)
\curveto(100.71475,208.3593707)(101.67205,208.7702907)(103.7979,210.3753107)
\lineto(103.7979,210.3753107)
\closepath
}
}
{
\newrgbcolor{curcolor}{0.80000001 0.80000001 0.80000001}
\pscustom[linestyle=none,fillstyle=solid,fillcolor=lightgray]
{
\newpath
\moveto(156.80183,58.0860507)
\curveto(75.142651,61.5619607)(30.454483,77.4830007)(15.42276,108.4551307)
\curveto(9.6171471,120.4173107)(8.1801264,129.2500907)(8.1761123,152.9972907)
\curveto(8.1738298,166.5009207)(8.3222633,170.3477807)(9.0141388,174.7158707)
\curveto(13.369962,202.2159407)(25.498042,218.7591707)(49.795196,230.3430007)
\curveto(57.667834,234.0963407)(63.767775,236.2382207)(73.931447,238.8179907)
\curveto(93.204229,243.7098757)(115.00335,246.1456567)(146.95567,246.9775427)
\curveto(171.65779,247.6206687)(185.64454,246.6250277)(203.45388,242.9557417)
\curveto(220.71823,239.3987407)(236.92204,233.7475207)(251.59575,226.1658507)
\curveto(271.38381,215.9416807)(287.0395,202.0541507)(297.1447,185.7612007)
\curveto(306.04668,171.4082607)(310.80827,154.2683707)(310.86655,136.3678107)
\curveto(310.90705,123.9258007)(307.86718,115.1805407)(299.65792,104.1223807)
\curveto(293.6341,96.0080907)(281.29961,85.5300007)(270.2359,79.1285007)
\curveto(250.60988,67.7728307)(229.97201,61.0284407)(206.5015,58.3003007)
\curveto(199.8449,57.5265507)(172.69562,57.4095207)(156.80183,58.0860507)
\lineto(156.80183,58.0860507)
\closepath
\moveto(195.95204,75.8089407)
\curveto(240.87006,81.3504507)(269.64568,97.1665607)(280.3416,122.1924707)
\curveto(284.02553,130.8120007)(285.53376,140.9063107)(284.62226,150.8422707)
\curveto(283.97073,157.9443807)(283.30753,160.6196707)(281.23412,164.5097907)
\curveto(279.08353,168.5447207)(275.80952,171.9870407)(271.27144,174.9846407)
\curveto(266.34534,178.2385407)(264.30465,179.2458607)(258.31107,181.3821307)
\curveto(249.1061,184.6630107)(238.67373,186.7290407)(223.61507,188.2533207)
\curveto(216.98086,188.9248507)(210.96646,189.0854007)(192.20113,189.0918707)
\curveto(168.55633,189.0998707)(164.11511,188.9175507)(134.76519,186.7319407)
\curveto(126.25529,186.0982307)(116.60253,185.4425707)(113.31462,185.2749207)
\lineto(107.33659,184.9701007)
\lineto(107.33659,186.8683907)
\curveto(107.33659,189.5568307)(108.95738,194.4684207)(110.58682,196.7178007)
\curveto(111.34848,197.7692407)(113.6247,200.3886407)(115.6451,202.5386907)
\curveto(121.23714,208.4895907)(121.64076,209.3082207)(121.61801,214.6530107)
\curveto(121.60301,218.1699807)(121.38966,219.5639607)(120.60395,221.2778207)
\curveto(117.95178,227.0629607)(111.23328,230.7910707)(102.40159,231.3783507)
\curveto(94.428707,231.9085207)(85.943693,229.2375007)(81.045167,224.6555007)
\curveto(77.496747,221.3363707)(76.673012,219.4618307)(76.649756,214.6530107)
\curveto(76.625286,209.5929307)(77.273895,208.3142807)(82.892781,202.3457407)
\curveto(87.63036,197.3133307)(89.029597,195.2993607)(90.156167,191.8912607)
\curveto(90.903045,189.6318107)(91.370761,185.3500107)(90.910083,184.9893807)
\curveto(90.790077,184.8954807)(88.436934,184.2719007)(85.680871,183.6037507)
\curveto(75.648397,181.1715707)(66.909188,176.2340507)(59.746814,168.9514007)
\curveto(54.701059,163.8209107)(51.57463,158.8275007)(49.394488,152.4170607)
\curveto(45.811575,141.8819307)(46.343996,131.9728107)(51.05241,121.5607807)
\curveto(54.733275,113.4210407)(63.695326,103.5769407)(73.109461,97.3328707)
\curveto(92.023117,84.7881007)(112.05538,78.5781907)(143.90805,75.3856107)
\curveto(156.10221,74.1633907)(184.48501,74.3942607)(195.95204,75.8089407)
\lineto(195.95204,75.8089407)
\closepath
}
}
\rput(180,150){\psframebox*[framearc=.6]{$S^1$}}
\rput(90,130){\psframebox*[framearc=.6]{$S^2$}}
\rput(180,220){\psframebox*[framearc=.6]{$S^3$}}
\rput(300,240){\psframebox*[framearc=.6]{$S$}}
\end{pspicture}

\bildnummer
\end{center}

For every $i\ge 1$ let $p^i_t(x,y)$ denote the heat kernel of $S^i$ (for Dirichlet boundary conditions).
We extend these heat kernels $p^i$ to $[0,T]\times S\times S$ by zero, i.e., for $x\in S\setminus S^i$ or $y\in S\setminus S^i$ one sets $p^i_t(x,y)=0$.

Then one gets monotone convergence to the heat kernel $p$ of $S$, see \cite[Thm.~3.6]{Dodziuk}.
This means that $p_t^i(x,y)\le p_t^{i+1}(x,y)$ for every $i\ge 1$ and $p_t^i(x,y) \to p_t(x,y)$ as $i\to\infty$ holds for all $x,y\in S$ and $t\in [0,T]$.

\begin{rem}
For any $i\ge 1$ and any $T>0$ one has 
\[
C_{x_0}\left([0,T];S^i \right) =C_{x_0}\left([0,T];S \right)\cap\bigcap_{s\in [0,T]\cap \mathbb{Q}}\left\{w:[0,T]\to S \mid w(s)\in S^i\right\},
\]
and hence $C_{x_0}\left([0,T];S^i \right) \subset C_{x_0}\left([0,T];S \right)$ is a measurable subset and similarly for $C_{x_0}^{y_0}\left([0,T];S^i \right) \subset C_{x_0}^{y_0}\left([0,T];S \right)$.
\end{rem}

For any $i\ge 1$ Corollary~\ref{Dirichletwienermeasure} and Remark~\ref{wienercylinders} give a measure $\W^i_{x_0}$ on $C_{x_0}\left([0,T];S \right)$ with support in $C_{x_0}\left([0,T];S^i \right)$ such that for any $n\in \N$, any $0<t_1<\ldots<t_n= T$ and any Borel sets $B_1,\ldots,B_n\subset S$
one has
\begin{align}
&\W_{x_0}^i\left(\left\{w\in C_{x_0}\left([0,T];S\right) \,\big|\,
  w(t_1)\in B_1,\ldots,w(t_n)\in B_n \right\} \right)  \nonumber\\
 &=\int_{S^n}\1_{B_1\times\ldots\times B_n}(x_1,\ldots,x_n)\,p^i_{t_n-t_{n-1}}(x_n,x_{n-1})\cdots
 p^i_{t_2-t_1}(x_2,x_1) p^i_{t_1}(x_1,x_0)d\mu(x_1)\cdots d\mu(x_n).\label{wienerfori}
\end{align}

\noindent
Next, we need an elementary result from measure theory (compare \cite[p.~30f]{Doob}):
\begin{lem}\label{VHS}
Let $(\Omega,\mathcal{A})$ be a measurable space.
For any $i\ge 1$ let $\lambda^i$ be a measure on $(\Omega,\mathcal{A})$.
Assume that $\lambda^i(A)\le \lambda^{i+1}(A)$ for any $A\in\mathcal{A}$ and any $i\ge 1$, and set
\[
\lambda(A)=\lim_{i\to\infty}\lambda^i(A)\in [0,\infty]\quad\mbox{ for any }A\in\mathcal{A}.
\]
Then $\lambda$ is measure on $(\Omega,\mathcal{A})$.
\end{lem}
\begin{proof}
Obviously we have $\lambda(\emptyset)= 0$, and we need to show the $\sigma$-additivity of $\lambda$.
Let $(A_k)_{k\ge 1}$ be a family of pairwise disjoint measurable subsets of $\Omega$.
For any $\ell\ge 1$ we have
\[
\sum_{k=1}^\ell\lambda(A_k)= \sum_{k=1}^\ell\lim_{i\to\infty}\lambda^i(A_k)
=\lim_{i\to\infty}\lambda^i(\bigcup_{k=1}^\ell A_k) =\lambda(\bigcup_{k=1}^\ell A_k) 
\le \lambda(\bigcup_{k=1}^\infty A_k).
\]
As $\ell\to\infty$ we get
\begin{equation}\label{eqn_first_half_of_measure}
\sum_{k=1}^\infty \lambda(A_k)\le \lambda(\bigcup_{k=1}^\infty A_k).
\end{equation}
Conversely, the monotonicity of the sequence of measures give for any $i\ge 1$
\[
\lambda^i(\bigcup_{k=1}^\infty A_k) =\sum_{k=1}^\infty \lambda^i(A_k)\le \sum_{k=1}^\infty \lambda(A_k).
\]
Letting $i\to\infty$ we get $\lambda(\bigcup\limits_{k=1}^\infty A_k)\le \sum\limits_{k=1}^\infty \lambda(A_k)$, which together with \eqref{eqn_first_half_of_measure} shows  the $\sigma$-additivity.
\end{proof}

The next lemma states that the monotonicity of measures can be verified on a generator of the $\sigma$-algebra (see also \cite[Satz~5.8]{Elstrodt}).

\begin{lem}\label{crit:monotony}
Let $(\Omega,\mathcal{A})$ be a measurable space.
Let $\mathcal{G}$ be a semiring\footnote{This means $\emptyset\in\mathcal{G}$ and for any $G,H\in\mathcal{G}$ one has $G\cap H\in\mathcal{G}$ and the set $G\cap H^c$ can be written as the union of finitely many disjoint sets $G_1,\ldots,G_m\in\mathcal{G}$.} on $\Omega$ which generates the $\sigma$-algebra $\mathcal{A}$.
Let $\lambda$ and $\mu$ be two finite measures on $\mathcal{A}$ such that $\lambda(G)\le \mu(G)$ for every $G\in \mathcal{G}$.
Then one has $\lambda(A)\le \mu(A)$ for every $A\in \mathcal{A}$.
\end{lem}
\begin{proof}
For any $G\in\mathcal{G}$ we set $\nu(G)=\mu(G)-\lambda(G)\ge 0$. 
This yields a $\sigma$-additive set function $\nu:\mathcal{G}\to[0,\infty)$ on the semiring $\mathcal{G}$ with $\nu(\emptyset)=0$. 
By Carath\'eodory's Extension Theorem (compare e.g.~\cite[Kor.~5.7]{Elstrodt}) one can extend $\nu$ uniquely to a finite measure $\nu:\mathcal{A}\to [0,\infty)$.
The $\sigma$-additive set functions $\mu$ and $\lambda+\nu$ coincide on $\mathcal{G}$ which generates $\mathcal{A}$ and is stable under $\cap$. 
Hence, we can conclude that $\mu(A)=\lambda(A)+\nu(A)$ for every $A\in\mathcal{A}$, and as $\nu(A)\ge 0$ we are done.
\end{proof}

Now we consider $\lambda^i=\W^i_{x_0}$ obtained as above.
The $\sigma$-algebra $\mathcal{C}$ of $C_{x_0}\left([0,T];S \right)$ is generated by the semiring 
\[
{\mathcal Z}= \Big\{ \bigcap_{i=1}^n \pi_{t_i}^{-1}(B_i) \cap C_{x_0}\left([0,T];S\right)\,\big|\, 0\le t_1<\ldots<t_n\le T \mbox{ and Borel sets }B_1,\ldots, B_n\subset S \Big\}.
\]
Since the kernels $p_t^i(x,y)$ increase monotonically in $i$, formula \eqref{wienerfori} shows that the measures $\lambda^i$ increase monotonically on $\mathcal{Z}$.
By Lemma~\ref{crit:monotony} these $\lambda^i=\W^i_{x_0}$ form a monotone sequence of measures on $C_{x_0}\left([0,T];S \right)$. 
For any measurable subset $A\in\mathcal{C}$ we set
\[
\W_{x_0}(A):=\lim_{i\to\infty}\W^i_{x_0}(A) \in[0,1].
\]
Then Lemma~\ref{VHS} says that $\W_{x_0}$ is a measure on $C_{x_0}\left([0,T];S \right)$.
Monotone convergence in (\ref{wienerfori}) yields that $\W_{x_0}$ is indeed the Wiener measure induced by the heat kernel of $S$.

For the conditional Wiener measure the discussion is exactly the same.
We summarize:

\begin{prop}\label{WienerlebtaufallenMgfen}
Let $S$ be a connected Riemannian manifold without boundary, let $x_0,y_0\in S$.
Then the heat kernel induces a Wiener measure $\W_{x_0}$ on $\left(C_{x_0}\left([0,T];S\right),\mathcal{C}\right)$ and a conditional Wiener measure $\W_{x_0}^{y_0}$ on $\left(C_{x_0}^{y_0}\left([0,T];S\right),\mathcal{C}\right)$.

In both cases, the set of H\"older continuous paths of any order $\theta\in(0,1/2)$ has full measure.\qed
\end{prop}

\begin{rem}
A direct construction of the (conditional) Wiener measure on arbitrary connected Riemannian manifolds using Corollary~\ref{wienermeasureexists} or Corollary~\ref{wienerbridgeexists} seems impossible because the relevant estimates \eqref{WienerChentsov} and \eqref{WienerChentsovBridge} need not hold unless one assumes suitable restrictions on the geometry.
\end{rem}

\begin{rem}
The bound $1/2$ on the H\"older exponent is sharp.
On Euclidean space, the set of H\"older continuous paths of order $\theta\ge 1/2$ is known to be a null set, see \cite[Thm.~VII]{PWZ} or \cite[Thm.~5.4]{Simon} for the case $\theta>1/2$.
In particular, differentiable paths form a null set.
\end{rem}

\subsection{Infinite paths on stochastically complete manifolds}

Up to now we have considered paths which are defined on compact intervals of the form $[0,T]$.
In contrast to the conditional Wiener measure, the Wiener measure is defined on spaces of paths where the initial point is fixed but the end point is not.
Therefore it is reasonable to consider $C_{x_0}([0,\infty);S)=\{\mbox{continuous paths }w:[0,\infty)\to S\mid w(0)=x_0\}$ and define Wiener measure on this space.
Again, $C_{x_0}([0,\infty);S)$ is equipped with the Borel $\sigma$-algebra $\mathcal{C}$ induced by the compact-open topology.

\begin{dfn}\label{Def_infinite_Wiener}
Let $p$ denote the heat kernel of $S$ and let $x_0\in S$. 
A measure $\W_{x_0}$ on $\left(C_{x_0}\left([0,\infty);S\right),\mathcal{C}\right)$ is called {\em Wiener measure} induced by the heat kernel if
\begin{align}\label{Wienerbed}
\W_{x_0}&\left(\left\{w\in C_{x_0}\left([0,\infty);S\right) \,\big|\,
  w(t_1)\in U_1,\ldots,w(t_n)\in U_n\right\} \right)  \\
 =\int_{S^n}\1_{U_1\times\ldots\times U_n}&(x_1,\ldots,x_n)\,p_{t_n-t_{n-1}}(x_n,x_{n-1})\cdots
 p_{t_2-t_1}(x_2,x_1) p_{t_1}(x_1,x_0)d\mu(x_1)\cdots d\mu(x_n), \nonumber
\end{align}
for any $n\in \N$, any $0<t_1<\ldots<t_n$ and any open subsets $U_1,\ldots,U_n\subset S$.
\end{dfn}

Again, the Wiener measure is unique if it exists.

\begin{prop}\label{inifite_Wiener_exists_stoch_comlete}
Let $S$ be a stochastically complete connected Riemannian manifold without boundary, let $x_0\in S$.
Then the Wiener measure $\W_{x_0}$ on $\left(C_{x_0}\left([0,\infty);S\right),\mathcal{C}\right)$, which is induced by the heat kernel, exists.

The set of locally H\"older continuous paths of any order $\theta\in(0,1/2)$ has full measure.
\end{prop}

\begin{proof}
For $T>0$ consider the extension map $\extT : C_{x_0}([0,T];S) \to C_{x_0}([0,\infty);S)$ which prolongs each path $w\in C_{x_0}([0,T];S)$ by its value at $T$, i.e.,
$$
\extT(w)(t) = \begin{cases}
              w(t) & \mbox{if }t\in [0,T];\\
              w(T) & \mbox{if }t\in [T,\infty) .
              \end{cases}
$$
This map is continuous with respect to the compact-open topology and hence measurable.
Then $\lambda_T := (\extT)_*(\W_{x_0})$ is a measure on $C_{x_0}([0,\infty);S)$.
For $0\le t_1 < \cdots < t_n$ we denote the canonical projection by $\pi_{\{t_1,\ldots,t_n\}}:S^{[0,\infty)}\to S^{\{t_1,\ldots,t_n\}}$, and as in Section~\ref{section_2} let $\mathcal{B}^{\{t_1,\ldots,t_n\}}$ be the product $\sigma$-algebra on $S^{\{t_1,\ldots,t_n\}}$.
We consider the ring\footnote{A system of sets $\mathcal{R}$ is called a {\em ring} if $\emptyset\in\mathcal{R}$ and for any $G,H\in\mathcal{R}$ one has $G\cup H\in\mathcal{R}$ and $G\setminus H \in\mathcal{R}$.} of cylinder sets
\begin{equation*}
{\mathcal{R}}= \Big\{  \pi_{\{t_1,\ldots,t_n\}}^{-1}(B) \cap C_{x_0}\left([0,T];S\right)\,\big|\, 0\le t_1<\ldots<t_n \mbox{ and  } B\in\mathcal{B}^{\{t_1,\ldots,t_n\}}  \Big\}.
\end{equation*}
We notice that $\sigma$-algebra $\mathcal{C}$ is generated by $\mathcal{R}$.
Given a cylinder set $Z = \{w\in C_{x_0}([0,\infty);S)\,\mid\, \left(w(t_1),\ldots,w(t_n)\right)\in B\}$ where $0\le t_1 < \cdots < t_n$ and  $B\in\mathcal{B}^{\{t_1,\ldots,t_n\}}$ and given $T_2 > T_1 > t_n$ we check, using stochastic completeness,
\begin{align*}
\lambda_{T_2}(Z)
&=
\W_{x_0}(\mathrm{ext}_{T_2}^{-1}(Z)) \\
&=
\W_{x_0}(\{w\in C_{x_0}([0,T_2];S)\mid \left(w(t_1),\ldots,w(t_n)\right)\in B\}) \\
&=
\int_{S^{n+1}}\1_{B\times S}(x_1,\ldots,x_n,x_{n+1})\,p_{T_2-t_{n}}(x_{n+1},x_{n})\cdots
p_{t_1}(x_1,x_0)d\mu(x_1)\cdots d\mu(x_{n+1}) \\
&=
\int_{S^n}\1_{B}(x_1,\ldots,x_n)\,p_{t_n-t_{n-1}}(x_{n},x_{n-1})\cdots
p_{t_1}(x_1,x_0)d\mu(x_1)\cdots d\mu(x_n) \\
&=
\int_{S^{n+1}}\1_{B\times S}(x_1,\ldots,x_n,x_{n+1})\,p_{T_1-t_{n}}(x_{n+1},x_{n})\cdots
p_{t_1}(x_1,x_0)d\mu(x_1)\cdots d\mu(x_{n+1}) \\
&=
\W_{x_0}(\{w\in C_{x_0}([0,T_1];S)\mid \left(w(t_1),\ldots,w(t_n)\right)\in B\}) \\
&=
\lambda_{T_1}(Z) .
\end{align*}
In particular, the limit $\lambda(Z) := \lim_{T\to \infty}\lambda_T(Z)$ exists for all cylinder sets $Z$.
We have defined a finitely additive function $\lambda$ on the ring of cylinder sets $\mathcal{R}$ with values in $[0,1]$.
Now, if $(Z_n)_n$ is a sequence of cylinder sets with $Z_{n}\subset Z_{n+1}$ for all $n$ and $Z_\infty=\bigcup_n Z_n \in\mathcal{R}$, then for sufficiently large $T$ we have $\lambda(Z_n)=\lambda_T(Z_n)$ for all $n$ and $\lambda(Z_\infty)=\lambda_T(Z_\infty)$, and therefore
$$
\lim_{n\to\infty}\lambda(Z_n) = \lim_{n\to\infty}\lambda_T(Z_n) = \lambda_T(Z_\infty)=\lambda(Z_\infty)
$$
because $\lambda_T$ is a measure.
By \cite[Satz 1.10]{Elstrodt} this shows that $\lambda$ is $\sigma$-additive on $\mathcal{R}$, and Carath\'eodory's Extension Theorem applies.
As $\mathcal{R}$ is stable under $\cap$ we can extend $\lambda$ in a unique manner to a measure on $\mathcal{C}$, the $\sigma$-algebra generated by the cylinder sets.
This is the wanted Wiener measure $\W_{x_0}$ on $\left(C_{x_0}\left([0,\infty);S\right),\mathcal{C}\right)$.

By construction, it is induced by the heat kernel.
Fix $\theta\in (0,1/2)$.
It remains to show that locally H\"older continuous paths of order $\theta$ have full measure for $\W_{x_0}$.
Let $\mathcal{N}$ be the complement of the set of locally H\"older continuous paths of order $\theta$ in $C_{x_0}([0,\infty);S)$.
We have to show that $\mathcal{N}$ is a null set.
Now $\mathcal{N}=\bigcup_{k=1}^\infty \mathcal{N}_k$ where 
$$
\mathcal{N}_k = 
\left\{w\in C_{x_0}([0,\infty);S) \mid w|_{[0,k]}\mbox{ is not H\"older continuous of order }\theta\right\}.
$$
Hence it suffices to show that each $\mathcal{N}_k$ is a null set.
Let $\eps>0$.
Since $\widetilde{\mathcal{N}_k}:=\{w|_{[0,k]} \mid w\in \mathcal{N}_k\}$ is a null set in $C_{x_0}([0,k];S)$, we can cover  $\widetilde{\mathcal{N}_k}$ by countably many cylinder sets $\widetilde Z_n$ in $C_{x_0}([0,k];S)$ such that $\sum_n \W_{x_0}(\widetilde Z_n) < \eps$.
Now $Z_n := \{w\in C_{x_0}([0,\infty);S) \mid w|_{[0,k]}\in \widetilde Z_n\}$ is a cylinder set in $C_{x_0}([0,\infty);S)$ with $\W_{x_0}(Z_n) = \lambda_k(Z_n) = \W_{x_0}(\widetilde Z_n)$.
Hence we have 
$$
\mathcal{N}_k \subset \bigcup_{n=1}^\infty Z_n
\quad \mbox{and}\quad
\sum_n \W_{x_0}(Z_n) = \sum_n \W_{x_0}(\widetilde Z_n) < \eps.
$$
Thus $\mathcal{N}_k$ is a null set.
\end{proof}
\begin{rem}
Let $S$ be a stochastically complete connected Riemannian manifold without boundary, let $x_0\in S$, let $T>0$.
The heat kernel induces a Wiener measure $\W_{x_0}$ on $C_{x_0}\left([0,\infty);S\right)$ and another one $\W_{x_0}^T$ on $C_{x_0}\left([0,T];S\right)$.
Let $\mathrm{rest}_T:C_{x_0}\left([0,\infty);S\right)\to C_{x_0}\left([0,T];S\right)$ denote the restriction map $w\mapsto w|_{[0,T]}$.
Then the two Wiener measures are related by 
\[
\W_{x_0}^T=(\mathrm{rest}_T)_*\W_{x_0}.
\]
This equality can be easily be verified for cylinder sets in $C_{x_0}\left([0,T];S\right)$, which form a generator that is stable under $\cap$, hence it also holds on the whole $\sigma$-algebra $\mathcal{C}$ on $C_{x_0}\left([0,T];S\right)$.
\end{rem}

\begin{rem}

For a stochastically incomplete manifold $S$, this construction will in general yield the zero measure on $C_{x_0}\left([0,\infty);S\right)$.
One can rectify the situation by replacing $S$ by its $1$-point compactification $\widehat S:= S \cup \{\infty\}$.
The Riemannian volume measure of $S$ is extended to a measure $\widehat\mu$ on the Borel $\sigma$-algebra $\mathcal{\widehat B}$ of $\widehat S$ by giving $\{\infty\}$ measure $1$, i.e.~$\widehat\mu=\mu+\delta_\infty$.
The heat kernel $p$ of $S$ is extended to a transition function $\ph$ on $\widehat S$ by
$$
\ph_t(x,y) =
\begin{cases}
p_t(x,y) & \mbox{ if }x,y\in S,\\
1-\int_Sp_t(w,y)d\mu(w) & \mbox{ if }y\in S \mbox{ and }x=\infty,\\
0 & \mbox{ if }x\in S \mbox{ and }y=\infty,\\
1 & \mbox{ if }x=y=\infty.
\end{cases}
$$
The idea is that now paths can leave $S$ and go to infinity in finite time.
The probability of having left $S$ after time $t$ is complementary to that of staying in $S$.
Once arrived at $\infty$, the probability of returning to $S$ is zero.

One checks easily that $\ph$ defines a substochastic transition function on $(\widehat S,\mathcal{\widehat B}, \widehat\mu)$ as in Definition~\ref{def:transfunc}.
The total measure $\int_{\widehat S}\ph_t(x,y)d\widehat\mu(x)$ is now equal to $1$ as for stochastically complete manifolds.

Again,  one equips $C_{x_0}([0,\infty);\widehat S)$ with the Borel $\sigma$-algebra $\mathcal{\widehat C}$ induced by the compact-open topology.
A limiting procedure can be used to show existence and uniqueness of Wiener measure on $(C_{x_0}([0,\infty);\widehat S\,),\mathcal{\widehat C})$, i.e., of a ${\widehat \W}_{x_0}$ such that

\begin{align}\label{Wienerbed_incomplete}
{\widehat \W}_{x_0}&\left(\left\{w\in C_{x_0}([0,\infty);\widehat S\,) \big|\,
  w(t_1)\in U_1,\ldots,w(t_n)\in U_n\right\} \right)  \\
 =\int_{{\widehat S}^n}\1_{U_1\times\ldots\times U_n}&(x_1,\ldots,x_n)\,\ph_{t_n-t_{n-1}}(x_n,x_{n-1})\cdots
 \ph_{t_2-t_1}(x_2,x_1) \ph_{t_1}(x_1,x_0)d\widehat\mu(x_1)\cdots d\widehat\mu(x_n), \nonumber
\end{align}
for any $n\in \N$, any $0<t_1<\ldots<t_n$ and any open subsets $U_1,\ldots,U_n\subset \widehat S$.
We will not use this in the sequel.
\end{rem}

\section{Coverings}
Let $\pi:\tilde S \to S=\tilde S/\Gamma$ be a normal Riemannian covering where $\Gamma$ denotes the group of deck transformations acting by isometries on $\tilde S$.
By $\mu$ we denote the Riemannian volume measures both on $S$ and on $\tilde S$.
In order to compare the Wiener measures on $S$ and on $\tilde S$ we need to know how the heat kernels are related.
The following proposition is well known but a reference seems to be lacking.

\begin{prop}\label{prop:HeatKernelCovering}
Let $\tilde p$ be the heat kernel of $\tilde S$ and $p$ the heat kernel of $S$.
Suppose that $S$ is stochastically complete.

Then $\tilde S$ is also stochastically complete and we have for all $\tilde x,\tilde y\in\tilde S$ and all $t>0$
\begin{equation}\label{eq:HeatKernelCovering}
p_t(x,y) = \sum_{\gamma\in\Gamma} \tilde p_t(\tilde x,\gamma\tilde y)
\end{equation}
where $x=\pi(\tilde x)$ and $y=\pi(\tilde y)$.
The series in \eqref{eq:HeatKernelCovering} converges in $C^\infty$.
\end{prop}
Here $C^\infty$-convergence means the following:
Let $I\subset (0,\infty)$ be a relatively compact interval and let $U,V\subset S$ be relatively compact open subset over which the covering is trivial, i.e., $\pi^{-1}(U) = \bigsqcup_{\gamma\in\Gamma}\gamma \tilde U$ for some open subset $\tilde U \subset \tilde S$ for which $\pi|_{\tilde U}:\tilde U \to U$ is a diffeomorphism and similarly for $V$.
Then the series 
$$
(t,x,y) \mapsto \sum_{\gamma\in\Gamma} \tilde p_t((\pi|_{\tilde U})^{-1} (x),\gamma (\pi|_{\tilde V})^{-1} (y))
$$
converges together with all derivatives uniformly to $p|_{I\times U\times V}$.

For the proof we need the following auxiliary lemma.

\begin{lem}\label{lem:LpKonv}
Let $(\Omega,\mathcal{A},\mu)$ be a measure space and let $f_n:\Omega \to [0,\infty]$ be a monotonically increasing sequence of nonnegative measurable functions converging almost everywhere pointwise to $f:\Omega \to [0,\infty]$.
Let $1\leq r<\infty$.

If $\|f\|_{L^r}<\infty$, then $f_n\in L^r(\Omega,\mu)$ for all $n$ and $f_n \to f$ in $L^r(\Omega,\mu)$.
\end{lem}

\begin{proof}
Since $0\le f_n \le f$ almost everywhere we have $\|f_n\|_{L^r}\le\|f\|_{L^r}<\infty$ and therefore $f_n\in L^r(\Omega,\mu)$.
Moreover, $f-f_n\ge 0$ converges  almost everywhere pointwise monotonically to $0$.
Monotone convergence implies
\begin{equation*}
\|f-f_n\|_{L^r}^r = \int_\Omega (f-f_n)^rd\mu \to 0.
\qedhere
\end{equation*}
\end{proof}

\begin{proof}[Proof of Proposition~\ref{prop:HeatKernelCovering}]
Since $\tilde p>0$, the right hand side of \eqref{eq:HeatKernelCovering} defines a measurable function
$$
q : \R\times S \times S \to (0,\infty], \quad
q_t(x,y) := \sum_{\gamma\in\Gamma} \tilde p_t(\tilde x,\gamma\tilde y).
$$
By construction, $q_t(x,y)$ does not depend on the choice of lift $\tilde y$ of $y$.
Because of 
$$
\sum_{\gamma\in\Gamma} \tilde p_t(\tilde x,\gamma\tilde y)
=
\sum_{\gamma\in\Gamma} \tilde p_t(\gamma^{-1}\tilde x,\tilde y)
$$
it does not depend on the choice of lift of $x$ either.

We fix $x\in S$.
From
\begin{equation}\label{eq:kleinereins}
\int_S q_t(x,y)\,d\mu(y) = \int_{\tilde S}\tilde p_t(\tilde x,\tilde y)\,d\mu(\tilde y) \le 1
\end{equation}
we see that $(t,y)\mapsto q_t(x,y)$ is an $L^1_\mathrm{loc}$-function.
In particular, it can be considered as a distribution on $(0,\infty)\times S$.

Let $\phi\in C^\infty_c((0,\infty)\times S)$ be a test function.
We assume that the support of $\phi$ is contained in $(0,\infty)\times U$ where $U\subset S$ is an open subset over which the covering $\pi$ is trivial.
This assumption creates no loss of generality because every test function can be written as a finite sum of such test functions.
Put $\mathcal{S}_+ := \{(t,y)\in(0,\infty)\times S \mid \big(-\frac{\partial}{\partial t}-\Delta\big)\phi(t,y)\ge0\}$ and similarly $\mathcal{S}_- := \{(t,y)\in(0,\infty)\times S \mid \big(-\frac{\partial}{\partial t}-\Delta\big)\phi(t,y)\le0\}$.
Writing $\tilde x=\big(\pi|_{\tilde U}\big)^{-1}(x)$ and similarly for $y$, we get by monotone convergence
\begin{align*}
\int_{\mathcal{S}_+}q_t(x,y)\Big(-\frac{\partial}{\partial t}-\Delta\Big)&\phi(t,y)\,d\mu(y)\,dt \\
&=
\int_{\mathcal{S}_+\cap ((0,\infty)\times U)}q_t(x,y)\Big(-\frac{\partial}{\partial t}-\Delta\Big)\phi(t,y)\,d\mu(y)\,dt \\
&=
\int_{\mathcal{S}_+\cap ((0,\infty)\times U)}\sum_{\gamma\in\Gamma} \tilde p_t(\tilde x,\gamma\tilde y)\Big(-\frac{\partial}{\partial t}-\Delta\Big)\phi(t,y)\,d\mu(y)\,dt \\
&=
\sum_{\gamma\in\Gamma} \int_{\mathcal{S}_+\cap ((0,\infty)\times U)}\tilde p_t(\tilde x,\gamma\tilde y)\Big(-\frac{\partial}{\partial t}-\Delta\Big)\phi(t,y)\,d\mu(y)\,dt.
\end{align*}
Since $(t,y) \mapsto q_t(x,y)$ is $L^1_\mathrm{loc}$ and $\phi$ has compact support, the integral over $\mathcal{S}_+$ is finite. 
This shows in particular that the series over $\gamma$ converges absolutely.
We obtain a similar expression for the integral over $\mathcal{S}_-$.
Adding the two integrals yields
\begin{align*}
\int_0^\infty\int_S q_t(x,y)\Big(-\frac{\partial}{\partial t}-\Delta\Big)&\phi(t,y)\,d\mu(y)\,dt \\
&=
\sum_{\gamma\in\Gamma} \int_0^\infty\int_U\tilde p_t(\tilde x,\gamma\tilde y)\Big(-\frac{\partial}{\partial t}-\Delta\Big)\phi(t,y)\,d\mu(y)\,dt \\
&=
\sum_{\gamma\in\Gamma} \int_0^\infty\int_U\Big(\frac{\partial}{\partial t}-\Delta_{\tilde y}\Big)\tilde p_t(\tilde x,\gamma\tilde y)\cdot\phi(t,y)\,d\mu(y)\,dt 
\,=\,
0.
\end{align*}
Hence $(t,y)\mapsto q_t(x,y)$ satisfies the heat equation in the weak sense and, by parabolic regularity \cite[Thm.~7.4~(i)]{Gri09}, it coincides almost everywhere with a smooth function $(t,y)\mapsto \hat q_t(x,y)$ solving the heat equation in the classical sense.

For any test function $\phi\in C^\infty_c(S)$ the function $\tilde\phi := \phi\circ\pi\in C^\infty(\tilde S)$ will no longer have compact support in general, but it is bounded.
By \cite[Lem.~9.2]{Gri09} we have
$$
\lim_{t\searrow 0}\int_S \hat q_t(x,y)\phi(y)\,d\mu(y)
=
\lim_{t\searrow 0}\int_{\tilde S} \tilde p_t(\tilde x,\tilde y)\tilde\phi(\tilde y)\,d\mu(\tilde y)
=
\tilde\phi(\tilde x)
=
\phi(x).
$$
This together with \eqref{eq:kleinereins} shows that $(t,y)\mapsto\hat q_t(x,y)$ is a regular fundamental solution at $x$ in the terminology of \cite[Sec.~9]{Gri09}.
Since $S$ is stochastically complete, \cite[Cor.~9.6]{Gri09} implies that $\hat q_t(x,y)=p_t(x,y)$.

We know that $q_t(x,y)=p_t(x,y)$ for almost all $(t,y)\in (0,\infty)\times S$.
Next we show $q=p$ everywhere on $(0,\infty)\times S\times S$.
Since the function $(t,y)\mapsto p_t(x,y)$ is smooth, it is in $L^2_\mathrm{loc}((0,\infty)\times S)$.
Lemma~\ref{lem:LpKonv} implies that the series in \eqref{eq:HeatKernelCovering} converges in $L^2_\mathrm{loc}$ to  $(t,y)\mapsto p_t(x,y)$.
Hence \cite[Thm.~7.4~(ii)]{Gri09} applies and shows that the series converges locally uniformly.
In particular, $q=p$ everywhere.

Because of the symmetry of the heat kernel, it solves the heat equation also for the $x$-variable,
$$
\frac{\partial}{\partial t}p_t(x,y) = \Delta_x p_t(x,y).
$$
Thus $(t,x,y)\mapsto p_{2t}(x,y)$ solves the heat equation on $S\times S$,
$$
\frac{\partial}{\partial t}p_{2t}(x,y) = 
\Delta_x p_{2t}(x,y) + \Delta_y p_{2t}(x,y).
$$
From
\begin{align*}
\|p_{t}\|_{L^2(S\times S)}^2
&=
\int_{S\times S}p_{t}(x,y)^2 d\mu(x)d\mu(y) 
=
\int_{S}p_{2t}(x,x) d\mu(x) 
=
1
\end{align*}
we see that $p$ is in $L^2_\mathrm{loc}((0,\infty)\times S\times S)$.
Again by Lemma~\ref{lem:LpKonv} the series in \eqref{eq:HeatKernelCovering} converges in $L^2_\mathrm{loc}((0,\infty)\times S\times S)$ to $p$ and by \cite[Thm.~7.4~(ii)]{Gri09} we have $C^\infty$-convergence.

Finally, stochastic completeness of $S$ implies stochastic completeness of $\tilde S$ because
\begin{equation*}
1 = \int_S p_t(x,y)\,d\mu(y) = \int_{\tilde S}\tilde p_t(\tilde x,\tilde y)\,d\mu(\tilde y). \qedhere
\end{equation*}
\end{proof}

\begin{thm}
Let $S$ be a stochastically complete Riemannian manifold and let $\pi:\tilde S \to S=\tilde S/\Gamma$ be a normal Riemannian covering where $\Gamma$ denotes the group of deck transformations.
Let $\tilde x_0, \tilde y_0\in\tilde S$ and $x_0=\pi(\tilde x_0)$ and $y_0=\pi(\tilde y_0)$.
Let $I=[0,T]$ with $T>0$ or $I=[0,\infty)$.

Then 
$$
\pi_*: C_{\tilde x_0}\left(I;\tilde S\right) \to C_{x_0}\left(I;S\right), \quad\pi_*(w) = \pi\circ w,
$$ 
is a homeomorphism which preserves the Wiener measures, 
$$
\pi_*\W_{\tilde x_0}=\W_{x_0}.
$$
Similarly, 
$$
\pi_*: \bigsqcup_{\gamma\in\Gamma}C_{\tilde x_0}^{\gamma\tilde y_0}\left([0,T];\tilde S\right)\to C_{x_0}^{y_0}\left([0,T];S\right)
$$ 
is a homeomorphism such that 
$$
\W_{x_0}^{y_0}=\sum_{\gamma\in\Gamma}\pi_*\W_{\tilde x_0}^{\gamma\tilde y_0}.
$$
\end{thm}
The theorem makes the plausible statement that the probability (density) of finding a path emanating from $\tilde x_0$ after time $t$ at one of the points in $\pi^{-1}(y_0)$ is the same as that of finding the projected path at $y_0$.

\begin{center}
\psset{xunit=.5pt,yunit=.5pt,runit=.5pt}
\begin{pspicture}(-250,-160)(250,200){
\psline(-240,180)(100,180)
\psline(-100,60)(240,60)
\psline(-240,180)(-100,60)
\psline(100,180)(240,60)
\psline[arrows=->](0,30)(0,-40)
\psellipse(170,-100)(20,40)
\pscustom[linecolor=white,fillstyle=solid,fillcolor=white]{\pspolygon(-170,-60)(170,-60)(170,-140)(-170,-140)(-170,-60)}
\psline(-170,-60)(170,-60)
\psline(-170,-140)(170,-140)
\psellipse[linewidth=.2,linestyle=dashed](170,-100)(20,40)
\psellipse(-170,-100)(20,40)
}

\psline[linecolor=blue](-100,120)(-104,118)(-102,117)(-103,114)(-100.5,115)(-99,115.5)(-100,112.5)
(-99,113)(-98,115)(-97,111)(-96,114)(-95.5,111)(-94.5,108.5)
(-92,110)(-92.5,109)(-91.5,108.5)(-90,107.5)(-91,105)(-86,107)(-87,102.5)(-86.5,105.5)(-85.5,102.5)
(-85,100)(-84,102)(-83,105)(-82,104)(-81,106)(-80,110)(-79,109)(-78,112)(-77,109)(-76,110)(-75,109)(-74,106)(-73,105)(-72,109)(-71,107)
(-70,110)(-69,112)(-68,109)(-67,112)(-66,108)(-65,112)(-65,107)(-63.5,108)(-61,103)(-57.5,103)
(-60,100)(-59,98)(-57,101)(-56,95)(-55,99)(-56.5,94)(-55,95)(-60,94)(-56,93)(-53,92)(-55.5,90.5)(-50,90)

\psline[linecolor=blue](-110,-110)(-114,-112)(-112,-113)(-115,-116)(-110,-115.5)(-108,-116.5)(-109,-119)
(-108,-118)(-107,-116)(-106,-119)(-104.5,-117)(-104,-119)(-103,-121.5)
(-102,-120)(-102.5,-121)(-101.5,-121.5)(-100,-122.5)(-101,-125)(-96,-122)(-97,-127.5)(-96.5,-124.5)(-95.5,-127.5)
(-95,-130)(-94,-128)(-93,-125)(-92,-126)(-91,-124)(-90,-120)(-89,-121)(-88,-118)(-87,-121)(-86,-120)(-85,-121)(-84,-124)(-83,-125)(-82,-121)(-81,-123)(-80,-120)(-79.5,-118)(-79,-121)(-77,-118)(-76,-122)(-75.5,-118)(-75.5,-123)(-74.5,-122)(-74,-127)(-73.5,-127)
(-74,-130)(-74.5,-132)(-73,-129)(-74,-135)(-72,-131)(-73,-136)(-72,-135)(-70,-135)(-73.5,-133)(-71,-132)(-72.5,-136)(-70,-140)

\psline[linecolor=blue](-50,90)(-49,88)(-48,84)(-47,92)(-46,93)(-45,91)(-44,92)(-43,89)(-42,86)(-41,89)(-40,90)
(-39,91)(-38,88)(-37,87)(-36,86)(-35,84)(-34,89)(-33,91)(-32,92)(-31,89)(-30,92)
(-29,94)(-28,95)(-27,91)(-26,96)(-25,93)(-24,97)(-23,105)(-22,99)(-21,102)(-20,108)
(-19,104)(-18,105)(-17,102)(-16,101)(-15,103)(-14,100)(-13,104)(-12,102)(-11,101)(-10,105)
(-9,102)(-8,99)(-7,95)(-6,91)(-5,94)(-4,92)(-3,88)(-2,87)(-1,85)(0,80)
(1,84)(2,85)(3,88)(4,82)(5,87)(6,92)(7,89)(8,91)(9,87)(10,90)
(11,91)(12,87)(13,88)(14,88)(15,85)(16,87)(17,83)(18,86)(19,82)(20,80)
(21,78)(22,81)(23,75)(24,78)(25,74)(26,79)(27,81)(28,86)(29,85)(30,90)
(31,87)(32,85)(33,86)(34,79)(35,81)(36,82)(37,81)(38,79)(39,78)(40,75)
(41,79)(42,82)(43,87)(44,86)(45,84)(46,82)(47,78)(48,81)(49,79)(50,75)
(51,77)(52,76)(53,81)(54,80)(55,85)(56,86)(57,84)(58,85)(59,87)(60,90)
(61,89)(62,88)(63,92)(64,97)(65,99)(66,100)(67,98)(68,101)(69,98)(70,100)
(71,101)(72,97)(73,99)(74,94)(75,97)(76,92)(77,91)(78,89)(79,88)(80,90)
(81,91)(82,86)(83,88)(84,83)(85,82)(86,79)(87,81)(88,77)(89,78)(90,75)
(91,78)(92,74)(93,77)(94,73)(95,71)(96,73)(97,69)(98,71)(99,67)(100,70)

\psline[linecolor=blue,linewidth=.2,linestyle=dotted](-70,-140)(-40,-120)(10,-75)(30,-60)

\psline[linecolor=blue](100,70)(97.5,70.5)(102.5,73.5)(100,74)(99,75.5)(101,77)(101.5,78)(103.5,80.5)(104.5,80.5)(105.5,82.5)(104,83)(103,83)(105,85)(107,84.5)(106,81)(108.5,81)(109,79)(110.5,78)(112,77)(111.5,74.5)(112,72)(113.5,71)
(115,70)(117,72.1)(116,69.7)(119,73.3)(118.5,71.65)(121,74.5)(121,73.6)(121,72.7)(122.5,74.05)(123,74)(125,76.2)(127,78.5)(127,77.7)(129,80)(128,77.7)(130,80)
\psline[linecolor=blue](30,-60)(32.5,-63.4)(32.5,-60)(33,-62.7)(34,-64.6)(36,-63.5)(36.5,-65.4)(39.5,-62.8)(41.5,-61.7)(41.5,-65.1)(43,-64)
(44.2,-64.85)(40.9,-66.7)(42.85,-67.05)(41.8,-68.4)(42.25,-69.25)(42.7,-69.5)(40.9,-71.7)(39.1,-73.3)(39.55,-74.15)
(40,-75)(42.1,-75)(39.7,-78)(43.3,-77)(41.65,-79.5)(44.5,-79)(43.6,-81)(42.7,-83)(44.05,-82.5)(44,-85)(46.2,-85)(48.5,-85)(47.7,-87)(50,-87)(47.7,-90)(50,-90)

\psline[linecolor=brown](-100,120)(-99,126)(-98,118)(-97,122)(-95,119)(-93,128)(-92,126)(-91,132)(-90,135)(-89,132)(-88,127)(-87,135)(-86,132)(-85,130)(-84,128)(-82,145)(-83,152)(-81,157)(-80,160)(-79,165)(-78,163)(-77,157)(-76,154)(-75,158)(-74,160)(-73,158)(-72,157)(-71,160)(-69,162)(-68,159)(-67,161)(-66,160)(-65,157)(-64,155)(-63,150)(-62,153)(-61,151)(-60,150)(-59,152)(-58,156)(-57,160)(-56,152)(-55,155)(-54,152)(-53,160)(-52,162)(-51,161)(-50,160)(-49,157)(-48,155)(-47,151)(-46,152)(-45,150)(-44,146)(-43,143)(-42,144)(-41,140)(-39,135)(-38,132)(-37,135)(-36,131)(-35,129)(-34,125)(-33,130)(-32,125)(-31,130)(-30,125)(-29,128)(-28,126)(-27,130)(-26,128)(-25,133)(-24,134)(-23,129)(-22,130)(-21,132)(-20,130)(-19,133)(-18,132)(-17,137)(-16,139)(-15,140)(-14,139)(-13,138)(-12,135)(-11,132)(-9,130)(-8,128)(-7,131)(-6,132)(-5,135)(-4,133)(-3,136)(-2,131)(-1,133)(0,135)(1,133)(2,130)(3,125)(4,122)(5,118)(6,116)(7,115)(8,118)(9,113)(10,110)(11,114)(12,111)(13,112)(14,115)(15,114)(16,116)(17,118)(18,120)(19,119)(20,120)(21,123)(22,125)(23,128)(24,123)(25,124)(26,123)(27,127)(28,127)(29,126)(31,125)(32,121)(33,120)(34,126)(35,122)(36,121)(37,120)(38,121)(39,118)(40,115)(41,118)(42,127)(43,130)(44,132)(45,134)(46,133)(47,134)(48,130)(49,126)(50,130)(51,132)(52,137)(53,136)(54,140)(55,138)(56,142)(57,138)(58,138)(59,142)(60,140)
\psline[linecolor=brown](-110,-110)(-109,-105)(-108,-113)(-107,-117)(-105,-115)(-103,-108)(-102,-110)(-101,-106)(-100,-100)(-99,-103)(-98,-108)(-97,-100)(-96,-103)(-95,-104)(-94,-106)(-92,-88)(-93,-80)(-91,-75)(-90,-72)(-89,-67)(-88,-69)(-87,-73)(-86,-78)(-85,-74)(-84,-72)(-83,-74)(-82,-75)(-81,-72)(-79,-70)(-78,-73)(-77,-71)(-76,-72)(-75,-74)(-74,-75)(-73,-80)(-72,-77)(-71,-79)(-70,-80)(-69,-78)(-68,-76)(-67,-75)(-66,-80)(-65,-78)(-64,-80)(-63,-75)(-62,-73)(-61,-74)(-60,-75)(-59,-78)(-58,-80)(-57,-84)(-56,-83)(-55,-85)(-54,-89)(-53,-92)(-52,-91)(-51,-95)(-49,-97)(-48,-99)(-47,-97)(-46,-100)(-45,-101)(-44,-105)(-43,-100)(-42,-105)(-41,-103)(-40,-105)(-39,-102)(-38,-104)(-37,-100)(-36,-102)(-35,-97)(-34,-96)(-33,-101)(-32,-100)(-31,-98)(-30,-100)(-29,-97)(-28,-98)(-27,-98)(-26,-95)(-25,-94)(-24,-94)(-23,-96)(-22,-95)(-21,-98)(-19,-100)(-18,-112)(-17,-109)(-16,-108)(-15,-105)(-14,-107)(-13,-104)(-12,-109)(-11,-107)(-10,-105)(-9,-106)(-8,-108)(-7,-122)(-6,-115)(-5,-118)(-4,-120)(-3,-121)(-2,-118)(-1,-122)(0,-125)(1,-121)(2,-123)(3,-124)(4,-117)(5,-117)(6,-115)(7,-112)(8,-110)(9,-111)(10,-110)(11,-107)(12,-105)(13,-102)(14,-107)(15,-106)(16,-107)(17,-103)(18,-103)(19,-104)(21,-105)(22,-109)(23,-110)(24,-104)(25,-108)(26,-109)(27,-110)(28,-109)(29,-112)(30,-115)(31,-112)(32,-103)(33,-100)(34,-98)(35,-96)(36,-97)(37,-96)(38,-100)(39,-104)(40,-100)(41,-98)(42,-93)(43,-94)(44,-90)(45,-91)(46,-88)(47,-92)(48,-92)(49,-88)(50,-90)

\psdot[dotstyle=square*,dotsize=8](-100,120)
\psdot[dotstyle=triangle*,dotsize=8](130,80)
\psdot[dotstyle=triangle*,dotsize=8](60,140)
\psdot[dotstyle=square*,dotsize=8](-110,-110)
\psdot[dotstyle=triangle*,dotsize=8](50,-90)
\rput(14,-10){$\pi$}
\rput(220,-100){$S$}
\rput(220,140){$\tilde S$}
\rput(-117,121){$\tilde x_0$}
\rput(149,80){$\gamma\tilde y_0 $}
\rput(77,140){$\tilde y_0 $}
\rput(-125,-107){$x_0$}
\rput(65,-90){$y_0$}
\end{pspicture}

\bildnummer
\end{center}

\begin{proof}
The inverse of $\pi_*$ is given by lifting the continuous paths.
Formula \eqref{eq:HeatKernelCovering} in Proposition~\ref{prop:HeatKernelCovering} shows that the formulas for the measures hold for all cylinder sets in $C_{x_0}\left(I;S\right)$ and in $C_{x_0}^{y_0}\left([0,T];S\right)$, respectively.
Uniqueness of the Wiener measure concludes the proof.
\end{proof}

\section{Examples}
To apply the theory developed so far, we compute the expectation value for the distance of a random path from its initial point after time $t$.
Let $S$ be a geodesically and stochastically complete Riemannian manifold.
Let $p$ be its heat kernel and $\rho$ its distance function.
Then the expectation value for the distance of a random path emanating from $x_0\in S$ to $x_0$ after time $t$ is given by 
\begin{align}
\E[\rho(X_0,X_t)]
&=
\int_{C_{x_0}([0,t];S)} \rho(x_0,w(t))\, d\W_{x_0}(w) \nonumber \\
&=
\int_S\rho(x_0,x) p_t(x,x_0)\,dx . \label{eq:expvalue}
\end{align}

Explicit formulas for the heat kernel are available only for very few manifolds.
The most prominent example is Euclidean space where the heat kernel has been known for a long time.
Euclidean space is geodesically and stochastically complete and its heat kernel is given by 
$$
p_t(x,y) = (4\pi t)^{-n/2}\cdot\exp\left(-\frac{\|x-y\|^2}{4t}\right) .
$$
By homogeneity of $\R^n$ we may assume that the initial point of our random path is the origin.
Then \eqref{eq:expvalue} gives
\begin{align*}
\E[\rho(X_0,X_t)]
&=
\int_{\R^n} \|x\| \cdot (4\pi t)^{-n/2}\cdot\exp\left(-\frac{\|x\|^2}{4t}\right) dx \\
&=
(4\pi t)^{-n/2}\cdot \mathrm{vol}(S^{n-1})\cdot \int_0^\infty r \cdot \exp\left(-\frac{r^2}{4t}\right)\cdot r^{n-1} dr \\
&=
(4\pi t)^{-n/2}\cdot \mathrm{vol}(S^{n-1})\cdot \int_0^\infty \left(\sqrt{4t}s\right)^n \cdot \exp\left(-s^2\right) \sqrt{4t}\,ds \\
&=
\frac{\sqrt{4t}\cdot\mathrm{vol}(S^{n-1})}{\pi^{n/2}}\cdot \int_0^\infty s^n \cdot \exp\left(-s^2\right) \,ds \\
&=
2\frac{\Gamma\left(\frac{n+1}{2}\right)}{\Gamma\left(\frac{n}{2}\right)}\sqrt{t} .
\end{align*}
The expectation value is in all dimensions proportional to $\sqrt{t}$.
The coefficient grows with the dimension\footnote{This can be seen using the Cauchy-Schwarz inequality: $\Gamma\left(\tfrac{n+1}{2}\right)^2=\left(\int_0^\infty t^{(n-1)/2}e^{-t}dt\right)^2=\left(\int_0^\infty t^{n/4}t^{(n-2)/4}e^{-t}dt\right)^2<\int_0^\infty t^{n/2}e^{-t}dt\cdot\int_0^\infty t^{(n-2)/2}e^{-t}dt=\Gamma\left(\tfrac{n+2}{2}\right)\Gamma\left(\tfrac{n}{2}\right)$.}.
This is plausible because in higher dimensions the random path has ``more space'' to depart from the initial point.

Now we pass from Euclidean space to hyperbolic space $H^n$.
We restrict ourselves to the $3$-dimensional case, $n=3$, where the heat kernel is given by
$$
p_t(x,y) = e^{-t}\cdot (4\pi t)^{-3/2}\cdot \frac{\rho(x,y)}{\sinh(\rho(x,y))}\cdot e^{-\frac{\rho(x,y)^2}{4t}} ,
$$
see \cite[p.~396]{DGM}.
We fix $x_0\in H^3$ and compute, using \eqref{eq:expvalue},
\begin{align*}
\E[\rho(X_0,X_t)]
&=
\int_{H^3} \rho(x_0,x)  p_t(x,x_0)\,dx \\
&=
e^{-t}\cdot (4\pi t)^{-3/2}\cdot\int_{H^3}  \frac{\rho(x_0,x)^2}{\sinh(\rho(x_0,x))}\cdot e^{-\frac{\rho(x_0,x)^2}{4t}}\,dx \\
&=
e^{-t}\cdot (4\pi t)^{-3/2}\cdot\mathrm{vol}(S^2)\cdot\int_0^\infty  r^2\cdot e^{-\frac{r^2}{4t}}\cdot\sinh(r)\,dr \\
&=
e^{-t}\cdot (4\pi t)^{-3/2}\cdot 4\pi \cdot \left(4t^2+2t^{3/2}\cdot \sqrt{\pi}\cdot e^t\cdot\mathrm{erf}(\sqrt{t})(1+2t)\right) \\
&=
e^{-t}\cdot 2\pi^{-1/2}\sqrt{t} + \mathrm{erf}(\sqrt{t})(1+2t).
\end{align*}
Here $\mathrm{erf}(x) = 2\pi^{-1/2}\int_0^x e^{-s^2}ds$ is the so-called \emph{error function}.

\begin{center}
\begin{pspicture}(0,-0.5)(7,8)
\psset{yunit=.5}
\psgrid[subgriddiv=1,griddots=10](0,0)(7,15)

\psplot[plotpoints=200,linecolor=black,linewidth=1.5pt]{0}{7}{2.718281828 x neg exp x sqrt mul 1 2 x mul add 0 x sqrt /t {2.718281828 t t mul neg exp} .1 SIMPSON mul add 1.128379167 mul}
\rput(7.5,15){ $H^3$}

\psplot[plotpoints=200,linecolor=blue,linewidth=1.5pt]{0}{7}{x sqrt 2.256758334 mul}
\rput(7.5,6){\blue$\R^3$}

\psline[linewidth=1pt]{->}(-0.5,0)(7.5,0)
\psline[linewidth=1pt]{->}(0,-0.5)(0,15.5)
\rput(8,0){$t$}
\end{pspicture}

\bildnummer

\end{center}

The plot shows the expectation values of $\R^3$ and $H^3$ as functions of time $t$.
While the asymptotic behavior of both functions is the same as $t\searrow 0$, it grows much faster for the hyperbolic space as $t\to\infty$.
Hence a random path in hyperbolic $3$-space departs faster from its initial point than a random path in Euclidean space.

For a nice discussion of the heat kernel on hyperbolic space in general dimensions see \cite{GN}.

\section{The Feynman-Kac Formula}
Let $S$ denote a connected Riemannian manifold without boundary, and let $t>0$ and $x_0,y_0\in S$.
In Proposition~\ref{WienerlebtaufallenMgfen} we have established existence
of the Wiener measure $\W_{x_0}$  on $(C_{x_0}([0,t];S),\mathcal{C})$ and the conditional Wiener measure $\W_{x_0}^{y_0}$  on $(C_{x_0}^{y_0}([0,t];S),\mathcal{C})$ which are induced by the heat kernel.
As before, let $\Delta$ denote (the Friedrichs extension of) the Laplace-Beltrami operator of $S$.
By Remark~\ref{wienercylinders} one can represent the heat semigroup $(e^{t\Delta})_{t\ge 0}$ applied to some
$g\in L^2(S,d\mu)$ as follows:
\begin{equation}\label{Feynman_Kac_triv_Potential}
\big(e^{t\Delta} g \big)(x_0)=\int_S g(x_1) p_t(x_1,x_0)\,d\mu(x_1)
= \int_{C_{x_0}\left([0,t];S\right)} g(w(t))\,d\W_{x_0}(w).
\end{equation}
Let $V\in L^\infty(S,d\mu)$ be a real valued function.
Then multiplication by $V$ defines a bounded operator $V:L^2(S,d\mu)\to L^2(S,d\mu)$.
The {\it Schr\"odinger operator} $H=\Delta-V$ is then a self\-adjoint operator in $L^2(S,d\mu)$ which is bounded from above.
Therefore $H$ generates a semigroup $(e^{t H})_{t\ge 0}$.
Next we want to generalize \eqref{Feynman_Kac_triv_Potential} to $H$.
For this we recall the classical Trotter product formula in the formulation of \cite[Thm.~X.51]{ReedSimon}:
\begin{thm}[Trotter product formula]\label{Trotter}
Let $A$ and $B$ be the generators of contraction semigroups on a Banach space $X$.
Denote their domains by $D(A)$ and $D(B)$, respectively.
Set $D=D(A)\cap D(B)$ and 
suppose that the closure of $(A+B)\big|_D$ generates a contraction semigroup on $X$.
Then, for all $\varphi\in X$,
\begin{equation*}
e^{t\,\overline{(A+B)}}\varphi = \lim_{n\to\infty}(e^{\frac{t}{n}A} e^{\frac{t}{n}B})^n\varphi.
\end{equation*}
\end{thm}
From this abstract product formula we will deduce the following Feynman-Kac formula.
We want to stress that neither geodesic completeness nor stochastic completeness of the Riemannian manifold $S$ are required.

\begin{thm}[Feynman-Kac formula]\label{Feynman-Kac}
Let $S$ be a connected Riemannian manifold.
Let $V\in L^\infty(S,d\mu)$ be a real valued function.

Then the semigroup generated by the Schr\"odinger operator $H=\Delta-V$ is given by
\begin{equation}\label{feynkacformula}
 \left(e^{tH} g \right)(x_0)=
\int_{C_{x_0}\left([0,t];S\right)} g(w(t))\cdot\exp\left(-\int_0^t V(w(s))ds \right)\,d\W_{x_0}(w)
\end{equation}
for any $g\in L^2(S,d\mu))$ and any $t>0$.
\end{thm}
The proof we give here follows the one of \cite[Theorem X.68]{ReedSimon}.

\begin{proof}[Proof of Theorem~\ref{Feynman-Kac}]
First we assume that $V:S\to\R$ is continuous and bounded.
Then we find a constant $\gamma>0$ with $|V|\le\gamma$, and the multiplication operator $(V+\gamma)$ is bounded selfadjoint operator on $L^2(S,d\mu)$ which has nonnegative spectrum and is therefore generator of a contraction semigroup on $L^2(S,d\mu)$.
We get that the operator $H_\gamma=\Delta-(V+\gamma)$ is an essentially selfadjoint operator in $L^2(S,d\mu)$ with domain $D(H_\gamma)=D(\Delta)\cap D(V+\gamma) = D(\Delta)$.
Since the spectrum of $H_\gamma$ is nonpositive, $H_\gamma$ generates a contraction semigroup on $L^2(S,d\mu)$.
Hence we can apply the Trotter product formula and get, for any $g\in L^2(S,d\mu)$,
\begin{align}
e^{tH}g&= e^{t\gamma} e^{tH_\gamma}g \nonumber\\
&= e^{t\gamma}\lim_{n\to\infty} (e^{\frac{t}{n}\Delta} e^{-\frac{t}{n}(V+\gamma)})^n g\nonumber\\
&= \lim_{n\to\infty} (e^{\frac{t}{n}\Delta} e^{-\frac{t}{n}V})^ng, \nonumber
\end{align}
where the limit is understood in $ L^2(S,d\mu)$.

For the moment let us fix a $g\in L^2(S,d\mu)$.
We note that $|g|\in L^2(S,d\mu)$ and therefore $e^{t\Delta}|g|\in L^2(S,d\mu)$ as well.
As the Riemannian volume measure $\mu$ is $\sigma$-finite, convergence of a sequence in $L^2(S,d\mu)$ implies that a suitable subsequence converges pointwise $\mu$-almost everywhere (see Corallary~\ref{cor1_appendixC}).
Therefore we can find a null set $\mathcal{N}_0\subset S$ and a sequence of positive integers $(n_k)_{k\ge 1}$ with $n_k\to\infty$ such that for all $x_0\in S\setminus \mathcal{N}_0$ one has
\begin{align}
e^{t\Delta}|g|(x_0)&<\infty, \label{gnormendlich}\\
e^{tH}g(x_0)&= \lim_{k\to\infty} (e^{\frac{t}{n_k}\Delta} e^{-\frac{t}{n_k}V})^{n_k}g(x_0). \label{Trotterpointwise}
\end{align}
For any $n\ge 1$ there is a null set $\mathcal{N}_n\subset S$ such that for any $x_0\in S\setminus \mathcal{N}_n$ we have 
\begin{align*}
&(e^{\frac{t}{n}\Delta} e^{-\frac{t}{n}V})^{n}g(x_0) \\
&= \int_S\cdots\int_S e^{-\frac{t}{n}V(x_n)}\cdots e^{-\frac{t}{n}V(x_1)}g(x_n)p_{\frac{t}{n}}(x_n,x_{n-1})\cdots p_{\frac{t}{n}}(x_2,x_1)p_{\frac{t}{n}}(x_1,x_0) d\mu(x_n)\cdots d\mu(x_1)\\
&= \int_S\cdots\int_S \exp\big(-\tfrac{t}{n}\sum_{j=1}^n V(x_j)\big)g(x_n)p_{\frac{t}{n}}(x_n,x_{n-1})\cdots p_{\frac{t}{n}}(x_2,x_1)p_{\frac{t}{n}}(x_1,x_0) d\mu(x_n)\cdots d\mu(x_1)\\
&=\int_{C_{x_0}([0,t];S)} \exp\big(-\tfrac{t}{n}\sum_{j=1}^n V\big(w(\tfrac{j}{n}\,t)\big)\big) g(w(t))\, d\W_{x_0}(w)
\end{align*}
where we have used Remark~\ref{wienercylinders} to get the last equality.
For all $x_0\in S$ and  $w\in C_{x_0}([0,t];S)$ the function $V\circ w:[0,t]\to\R$ is continuous and the Riemann sum converges to the integral:
\begin{equation*} 
\tfrac{t}{n}\sum_{j=1}^n V\big(w(\tfrac{j}{n}\,t)\big) \xrightarrow{n\to\infty} \int_0^t V(w(s))\,ds.
\end{equation*}
As one has
$\big| \exp\big(-\tfrac{t}{n}\sum_{j=1}^n V\big(w(\tfrac{j}{n}\,t)\big)\big) g(w(t))\big| \le e^{t\gamma}\, \big|g(w(t)) \big|$
and, by \eqref{Feynman_Kac_triv_Potential} and \eqref{gnormendlich}, 
\begin{equation}\label{greatdominator}
\int_{C_{x_0}([0,t];S)}e^{t\gamma}\, \big|g(w(t)) \big|\,d\W_{x_0}(w)\le e^{t\gamma}\,e^{t\Delta}|g|(x_0)<\infty
\end{equation}
for all $x_0\in S\setminus \mathcal{N}_0$, one can apply dominated convergence.
Using \eqref{Trotterpointwise} we get for any $x_0$ in the complement of the null set $\mathcal{N}:=\bigcup_{n\ge0} \mathcal{N}_n$ that
\begin{equation}\label{FK_pf}
\left(e^{tH} g \right)(x_0)=
\int_{C_{x_0}\left([0,t];S\right)} g(w(t))\cdot\exp\left(-\int_0^t V(w(s))ds \right)\,d\W_{x_0}(w).
\end{equation}

This proves the Feynman-Kac formula for bounded and continuous potentials $V$.
Now we pass to the general situation and drop the assumption of continuity of $V$.
Let $V:S\to[-\infty,\infty]$ be in $L^\infty(S,d\mu)$.
We choose a sequence of continuous functions $V_n:S\to\R$ with $|V_n(x)|\le \|V\|_{L^\infty(S,d\mu)}$ so that $V_n(x)\rightarrow V(x)$ pointwise for $\mu$-almost all $x\in S$. 
(We will establish the existence of such a sequence $(V_n)_{n\ge 1}$ in Lemma~\ref{lemma_Linfty_contin_approx} in Appendix~\ref{appendix_b}.)
All induced Schr\"odinger operators $H_n=\Delta-V_n$ are essentially selfadjoint and have the same domain $D(H_n)=D(H)=D(\Delta)$.
For every $f\in  D(\Delta)\subset L^2(S,d\mu)$ we have 
$$
\|Vf-V_nf\|_{L^2}^2 = \int_S |V(x)-V_n(x)|^2|f(x)|^2 d\mu(x) \xrightarrow{n\to\infty} 0
$$
by dominated convergence.
Thus
\[
H_n f=\Delta f - V_nf
\xrightarrow{n\to\infty}
\Delta f - V f =H f
\]
in $L^2(S,\mu)$.
Hence we obtain (e.g.~by means of criterion (i) of \cite[Thm.~9.16]{Weidmann}) that $H_n$ converges to $H$ in the strong resolvent sense.
From that and from the fact that all the Schr\"odinger operators $H_n$ and $H$ are bounded from above by $\|V\|_{L^\infty(S,d\mu)}$ we can conclude that the generated semigroups converge in the strong sense \cite[Thm.~9.18(b)]{Weidmann}), i.e.~for every $g\in L^2(S,d\mu)$ we have
\[
e^{tH_n} g\xrightarrow{n\to\infty}e^{tH} g\quad\mbox{ in }L^2(S,d\mu).
\]
Hence we get (after possibly passing to a subsequence) that, for $\mu$-almost all $x_0$,
\begin{equation}\label{pointwiseagain}
\left(e^{tH_n} g\right)(x_0)\xrightarrow{n\to\infty}\left(e^{tH} g\right)(x_0).
\end{equation}
Since  the set $B=\{x\in S\mid V_n(x) \mbox{ does not converge to } V(x)\}$ has measure zero, Lemma~\ref{nullsetlemma} applies and for  $\W_{x_0}$-almost all paths $w\in{C_{x_0}\left([0,t];S\right)}$ we have pointwise $V_n(w(s))\to V(w(s))$ for almost every $s\in[0,t]$.
Furthermore, the potentials $V_n$ and $V$ are all bounded by $\|V\|_{L^\infty(S,d\mu)}$ and for $\W_{x_0}$-almost all paths $w\in{C_{x_0}\left([0,t];S\right)}$ we have the convergence
$\int_0^tV_n(w(s))ds\to \int_0^t V(w(s))ds$.
Hence dominated convergence implies that
\begin{equation}\label{W-Konvergenz_auf_C}
g(w(t))\cdot\exp\left(-\int_0^t V_n(w(s))ds \right)
\xrightarrow{n\to\infty}
g(w(t))\cdot\exp\left(-\int_0^t V(w(s))ds \right).
\end{equation}
for $\W_{x_0}$-almost all $w\in{C_{x_0}\left([0,t];S\right)}$.
We notice that all functions on $C_{x_0}\left([0,t];S\right)$ defined by the left hand side and the right hand side in \eqref{W-Konvergenz_auf_C} are dominated by $e^{t \|V\|_{L^\infty(S,d\mu)}}|g(w(t))|$ which is $\W_{x_0}$-integrable by \eqref{greatdominator}.
Therefore we can apply dominated convergence once more and get, for $\mu$-almost every $x_0$, 
\begin{align*}
\int_{C_{x_0}\left([0,t];S\right)} &g(w(t))\cdot\exp\left(-\int_0^t V(w(s))ds \right)\,d\W_{x_0}(w)\\
\stackrel{\phantom{\eqref{FK_pf}}}{=} 
&\lim_{n\to\infty}\int_{C_{x_0}\left([0,t];S\right)} g(w(t))\cdot\exp\left(-\int_0^t V_n(w(s))ds \right)\,d\W_{x_0}(w)\\
\stackrel{\eqref{FK_pf}}{=} 
&\lim_{n\to\infty}\left(e^{tH_n} g \right)(x_0)\\
\stackrel{\eqref{pointwiseagain}}{=} 
&\left(e^{tH} g \right)(x_0). \qedhere
\end{align*}
\end{proof}

\begin{rem}
The Feynman-Kac formula \eqref{feynkacformula} holds for more general potentials than just for $V\in L^\infty(S,d\mu)$, which would not cover all physically relevant examples.
E.g.~for $S=\R^3$ it is known (see \cite[Theorem~X.68]{ReedSimon}) that \eqref{feynkacformula} is even true for $V\in L^2(\R^3)+L^\infty(\R^3)$.
\end{rem}

\begin{rem}
Under the assumptions of Theorem~\ref{Feynman-Kac} we can apply Lemma~\ref{disintegration} in order to rewrite \eqref{feynkacformula} as
\[
\left(e^{tH} g \right)(x_0)=
\int_S\int_{C_{x_0}^{y_0}\left([0,t];S\right)} g(y_0)\cdot\exp\left(-\int_0^t V(w(s))ds \right)\,d\W_{x_0}^{y_0}(w) d\mu(y_0)
\]
and we conclude that the integral kernel of of $e^{tH}$ is given by
\begin{equation}
q_t(x_0,y_0)=\int_{C_{x_0}^{y_0}\left([0,t];S\right)}\exp\left(-\int_0^t V(w(s))ds \right)\,d\W_{x_0}^{y_0}(w).
\label{eq:Schroedingerkern}
\end{equation}
\end{rem}

\begin{cor}
Let $S$ be a connected Riemannian manifold.
Let $V_1,V_2\in L^\infty(S,d\mu)$ be real valued functions and let $q_t^1$ and $q_t^2$ be the integral kernels of $e^{t(\Delta-V_1)}$ and $e^{t(\Delta-V_2)}$ respectively.
If $V_1 \leq V_2$, then
$$
q_t^1(x,y) \geq q_t^2(x,y) 
$$
for all $t>0$ and all $x,y\in M$.
\end{cor}

\begin{proof}
Clear from \eqref{eq:Schroedingerkern}.
\end{proof}

As another consequence we get that formula \eqref{eq:HeatKernelCovering} in Proposition~\ref{prop:HeatKernelCovering} also holds for Schr\"odinger operators.

\begin{cor}
Let $S$ be a stochastically complete connected Riemannian manifold and let $\pi:\tilde S \to S=\tilde S/\Gamma$ be a normal Riemannian covering where $\Gamma$ denotes the group of deck transformations..
Let $V\in L^\infty(S,d\mu)$ be a real valued function.
Put $\tilde V:=V\circ\pi \in L^\infty(\tilde S,d\mu)$.
Let $\tilde q$ be the integral kernel of $\tilde H = \Delta -\tilde V$ and $q$ the integral kernel of $H=\Delta -V$.

Then we have for all $\tilde x,\tilde y\in\tilde S$ and all $t>0$
\begin{equation}
q_t(x,y) = \sum_{\gamma\in\Gamma} \tilde q_t(\tilde x,\gamma\tilde y)
\end{equation}
where $x=\pi(\tilde x)$ and $y=\pi(\tilde y)$.
\end{cor}

\begin{proof}
Using Proposition~\ref{prop:HeatKernelCovering} and the Feynman-Kac formula, we see
\begin{align*}
q_t(x,y)
&=
\int_{C_{x}^{y}\left([0,t];S\right)}\exp\left(-\int_0^t V(w(s))ds \right)\,d\W_{x}^{y}(w) \\
&=
\sum_{\gamma\in\Gamma}\int_{C_{x}^{y}\left([0,t];S\right)}\exp\left(-\int_0^t V(w(s))ds \right)\,d\pi_*\W_{\tilde x}^{\gamma\tilde y}(w) \\
&=
\sum_{\gamma\in\Gamma}\int_{C_{\tilde x}^{\gamma\tilde y}\left([0,t];\tilde S\right)}\exp\left(-\int_0^t V(\pi\circ w(s))ds \right)\,d\W_{\tilde x}^{\gamma\tilde y}(w) \\
&=
\sum_{\gamma\in\Gamma}\int_{C_{\tilde x}^{\gamma\tilde y}\left([0,t];\tilde S\right)}\exp\left(-\int_0^t \tilde V(w(s))ds \right)\,d\W_{\tilde x}^{\gamma\tilde y}(w) \\
&=
\sum_{\gamma\in\Gamma}\tilde q_t(\tilde x,\gamma\tilde y).  \qedhere
\end{align*}
\end{proof}

\appendix
\section{Some Measure Theory}\label{appendix_b}
\noindent
In this appendix we collect some elementary facts from measure theory.
We start with two basic statements from measure theory (compare \cite[Lemma~4.1]{Kal} and \cite[Lemma~1.20]{Kal}).
\begin{lem}
Let $(\Omega,\mathcal{A},\mu)$ be a measure  space.
Then the following holds:
\begin{enumerate}
\item[a)] ({\it Markov's Inequality.}) 
For any $a>0$, any $\varepsilon>0$ and any measurable function $f:\Omega\to [0,\infty)$ one has
\[
\mu\left(\left\{\omega\in\Omega\,\mid\, f(\omega)\ge \varepsilon \right\} \right)\le \frac{1}{\varepsilon^a}\,
\int_\Omega f^a(\omega)\,d\mu(\omega) .
\]
\item[b)] ({\it Fatou's Lemma.}) 
Let $\mu$ be a finite measure, i.e.\ $\mu(\Omega)<\infty$, then
for any sequence of measurable sets $(A_n)_{n\ge 1}$ one has
\[
\limsup_n \,\mu(A_n)\le \mu\Big(\bigcap_{k\ge 1}\bigcup_{n\ge k} A_n\Big) .
\]
\end{enumerate}
\end{lem}
\begin{proof}
a) 
Since $f\ge 0$ we have
\[
\int_\Omega f^a(\omega)\,d\mu(\omega)
\ge \int_{\left\{f\ge\varepsilon\right\}} f^a(\omega)\, d\mu(\omega) \ge \varepsilon^a\cdot \mu\left(\left\{f\ge\varepsilon\right\}\right).
\]
Markov's inequality follows.

b) 
First we note that $\mu(A_n)\le \mu(\bigcup_{m\ge k} A_m)$ whenever $n\ge k$.
Furthermore, for any nested sequence $(B_k)_{k\ge 1}$ in $\mathcal{A}$ with $B_k\supset B_{k+1}$ for  all $k$ we have $\mu(\bigcap_{k\ge 1}B_k)=\inf_{k\ge 1}\mu(B_k)$.\footnote{It is an elementary property of any finite measure $\mu$ that $\mu(\bigcap_{k\ge 1}B_k)=\inf_{k\ge 1}\mu(B_k)$ for any nested sequence $(B_k)_{k\ge 1}$ in the $\sigma$-algebra, compare e.g.~\cite[Satz~3.2]{Bauer_masstheorie}}
Applying this to $B_k= \bigcup_{m\ge k} A_m$ yields
\begin{equation*}
\limsup_n \,\mu(A_n)
=
\inf_{k\ge 1}\,\sup_{n\ge k}\, \mu(A_n)
\le 
\inf_{k\ge 1}\, \mu\Big(\bigcup_{m\ge k} A_m\Big)
= 
\mu\Big(\bigcap_{k\ge 1}\bigcup_{m\ge k} A_m\Big). \qedhere
\end{equation*}
\end{proof}

\begin{dfn}
Let $(\Omega,\mathcal{A},\mu)$ be a measure space.
Let $f,f_n:\Omega\to[-\infty,\infty]$ be measurable functions, $n\ge 1$.
The sequence $(f_n)_n$ {\it converges stochastically} to $f$ if, for any $\varepsilon>0$,
\begin{equation*}
\mu\left(\{\omega\in \Omega\,\big|\, |f_n(\omega)-f(\omega)|>\varepsilon \} \right)\xrightarrow{n\to\infty} 0.
\end{equation*}
\end{dfn}

If $\mu$ is  finite, almost sure convergence of measurable functions implies stochastic convergence.
More precisely, this means:
\begin{lem}\label{fastsicherstochastisch}
Let $(\Omega,\mathcal{A},\mu)$ be a measure space with $\mu(\Omega)<\infty$.
Consider measurable functions $f,f_n:\Omega\to[-\infty,\infty]$, $n\ge 1$.
Assume that there is a null set $\mathcal{N}\in\mathcal{A}$ such that for any $\omega\not\in\mathcal{N}$ one has $f_n(\omega)\to 0$ as $n\to\infty$.
Then $(f_n)_n$ converges stochastically to $f$.
\end{lem}

\begin{proof} 
We fix $\varepsilon>0$ and we set $A_n:=\{\omega\in \Omega\,\big|\, |f_n(\omega)-f(\omega)|>\varepsilon \}$.
Then the convergence $f_n(\omega)\to f(\omega)$ for any $\omega\not\in\mathcal{N}$ can be reformulated as 
$\bigcap\limits_{k\ge 1}\bigcup\limits_{n\ge k} A_n \subset \mathcal{N}$.
This implies $\mu(\bigcap\limits_{k\ge 1}\bigcup\limits_{n\ge k} A_n)=0$ and Fatou's Lemma yields the claim.
\end{proof}

\begin{lem}\label{appC_Lem1}
Let $(\Omega,\mathcal{A},\mu)$ be a measure space with $\mu(\Omega)<\infty$.
Let $(f_n)_n$ be a sequence of measurable functions converging stochastically to a measurable function $f$.
Then there is a subsequence of $(f_n)_n$ that converges to $f$ pointwise for $\mu$-almost every $\omega\in \Omega$.
\end{lem}
\begin{proof}
By a diagonal argument we can find a subsequence, again denoted by $(f_n)_n$, such that for any $N\ge 1$ we have
\[\sum_{n\ge 1} \mu (\{\omega\,\big|\, |f_n(\omega)-f(\omega)|>\tfrac1N \}) <\infty.\]
This implies
\begin{align*}
\mu(\{\omega\,\big|\,\lim_{n\to\infty}f_n(\omega)\ne f(\omega) \} ) &=\mu\left(\bigcap_{N=1}^\infty \bigcap_{\ell=1}^\infty
\bigcup_{n\ge\ell}\{\omega\,\big|\,|f_n(\omega)-f(\omega)|>\tfrac1N \} \right) \\
&\le \lim_{N\to\infty} \lim_{\ell\to\infty} \sum_{n\ge\ell} \mu(\{x\,\big|\,|f_n(\omega)-f(\omega)|>\tfrac1N \} )\quad =\quad 0,
\end{align*}
which gives the pointwise convergence $f_n(\omega)\to f(\omega)$ for $\mu$-almost all $\omega\in \Omega$.
\end{proof}

\begin{cor}\label{Cor_locally_stoch_conv_yields_subseq}
Let $(\Omega,\mathcal{A},\mu)$ be a measure space where $\mu$ is $\sigma$-finite.
Let $\Omega=\bigsqcup_{i\ge1} E_i$ be a decomposition with $E_i\in\mathcal{A}$ and $\mu(E_i)<\infty$, $i\ge1$.
Furthermore, let $f,f_n:\Omega\to [-\infty,\infty]$, $n\ge 1$ be measurable functions.
Assume that for any $i$ the sequence $(f_n)_n$ converges stochastically to $f$ on $E_i$, i.e.\ for any $i\ge 1$ and $\varepsilon>0$,
\begin{equation*}
\mu\left(E_i\cap \{\omega\in \Omega\,\big|\, |f_n(\omega)-f(\omega)|>\varepsilon \} \right)\xrightarrow{n\to\infty} 0.
\end{equation*}
Then there is a subsequence of $(f_n)_n$ that converges to $f$ pointwise for $\mu$-almost every $\omega\in \Omega$.
\end{cor}
\begin{proof}
By Lemma~\ref{appC_Lem1}, on each $E_i$ it is possible to pass over to a subsequence that converges $\mu$-almost everywhere on $E_i$.
Applying a diagonal argument concludes the proof.
\end{proof}

\begin{cor}\label{cor1_appendixC}
Let $(\Omega,\mathcal{A},\mu)$ be a measure space where $\mu$ is $\sigma$-finite.
Let $(f_n)_n$ be a sequence in $L^2(\Omega,d\mu)$ converging to $f$ in the $L^2$-sense.
Then there is a subsequence of $(f_n)_n$ that converges to $f$ pointwise for $\mu$-almost $\omega\in \Omega$.
\end{cor}
\begin{proof}
By Markov's inequality we have, for any $\varepsilon>0$,
\begin{equation*}
\mu\left(\{\omega\,\big|\, |f_n(\omega)-f(\omega)|>\varepsilon \right)\,\le\, \tfrac{1}{\varepsilon^2}\int_S |f_n-f|^2\,d\mu\,=\, 
\tfrac{1}{\varepsilon^2}\left\|f_n-f \right\|_{L^2(\Omega,d\mu)}.
\end{equation*}
This shows that $L^2$-convergence implies stochastic convergence of $f_n\to f$, in particular $(f_n)_n$ converges stochastically to $f$ on every measurable set of finite measure, and therefore Corollary~\ref{Cor_locally_stoch_conv_yields_subseq} applies.
\end{proof}

The following Lemma is a technical approximation result we will apply later.

\begin{lem}\label{Lemma:ultimate_approximation}
Let $(\Omega,\mathcal{A},\mu)$ be a measure space where $\mu$ is $\sigma$-finite.
Let $f_n, f\in L^\infty(\Omega,d\mu)$, $n\ge 1$ with $f_n\to f$ in $L^\infty(\Omega,d\mu)$.
Furthermore, assume that for any $n\ge1$ there is a sequence of measurable functions $f_n^k:\Omega\to[-\infty,\infty]$ such that the sequence $(f_n^k)_k$ converges to $f_n$ pointwise for $\mu$-almost every $\omega\in\Omega$.\\
Then there are two sequences of integers $(k_\ell)_{\ell\ge1}$ and $(n_\ell)_{\ell\ge1}$ such that 
the  sequence $(g_\ell)_\ell$,
obtained by setting $g_\ell= f_{n_\ell}^{k_\ell}$,
converges to $f$ pointwise for $\mu$-almost every $\omega\in\Omega$.
\end{lem}

\begin{proof}
Since $\mu$ is $\sigma$-finite, we can find a decomposition $\Omega=\bigsqcup_{i\ge1} E_i$ with $E_i\in\mathcal{A}$ and $\mu(E_i)<\infty$, $i\ge1$.
Applying Lemma~\ref{fastsicherstochastisch} we obtain that, for any $i\ge1$ and any $\varepsilon>0$,
\begin{equation}\label{eqn_conv_11}
\mu\left(E_i\cap \{\omega\in \Omega\,\big|\, |f_n(\omega)-f_n^k(\omega)|>\varepsilon \} \right)\xrightarrow{k\to\infty} 0.
\end{equation}
For any $\varepsilon>0$, the triangle inequality implies
\begin{equation*}
\{\omega\,\big|\, |f(\omega)-f_n^k(\omega)|>2\varepsilon \} \subset
\{\omega\,\big|\, |f(\omega)-f_n(\omega)|>\varepsilon \} \cup
\{\omega\,\big|\, |f_n(\omega)-f_n^k(\omega)|>\varepsilon \},
\end{equation*}
and, as $f_n\to f$ in $L^\infty(\Omega,d\mu)$, there is an $N\ge1$ with
$\mu( \{\omega\,\big|\, |f(\omega)-f_n(\omega)|>\varepsilon \})=0$ for all $n\ge N$.
Combining this with \eqref{eqn_conv_11} we obtain that for every $\varepsilon>0$ there is an $N\ge1$ such that, for any $n\ge N$ and any $i\ge1$,
\begin{equation}\label{eqn_conv_22}
\mu\left(E_i\cap \{\omega\in \Omega\,\big|\, |f(\omega)-f_n^k(\omega)|>\varepsilon \} \right)\xrightarrow{k\to\infty} 0.
\end{equation}
For any $\ell\ge1$ we let $2^{-\ell}$ play the role of $ \varepsilon$ in \eqref{eqn_conv_22} and we choose $n_\ell$ and $k_\ell$ so large that
\begin{equation}\label{eqn_conv_33}
\mu\left(E_i\cap \{\omega\in \Omega\,\big|\, |f(\omega)-f_{n_\ell}^{k_\ell}(\omega)|>2^{-\ell} \}\right) < 2^{-\ell}
\end{equation}
for all $i=1,\ldots,\ell$. 
(Such $n_\ell$ and $k_\ell$ exist because we only impose finitely many conditions for their choice.)
From \eqref{eqn_conv_33} we can conclude that the sequence $(g_\ell)_\ell$,
obtained by setting $g_\ell= f_{n_\ell}^{k_\ell}$ converges stochastically to $f$ on any $E_i$.
Then Corollary~\ref{Cor_locally_stoch_conv_yields_subseq} applies and a subsequence of  $(g_\ell)_\ell$ converges to $f$ pointwise for $\mu$-almost every $\omega\in\Omega$.
\end{proof}

For the remainder of this appendix, let $(S,\mathcal{B})$ be a topological space equipped with its Borel $\sigma$-algebra.
\begin{dfn}
A measure $\mu$ on $\mathcal{B}$ is called {\em locally finite} if any $x\in S$ possesses an open neighborhood $U$ with $\mu(U)<\infty$.
\end{dfn}

\begin{exm}
If $S$ is a Riemannian manifold, then the volume measure is locally finite.
\end{exm}

For locally finite measures one has the following classical result (see \cite[Kap.~VIII, Satz 1.16]{Elstrodt} for a proof):
\begin{thm}[Ulam]\label{Ulam}
Let $(S,\rho,\mu)$ a metric measure space where the measure $\mu$ is locally finite.
Then $\mu$ is a regular measure in the sense that for any Borel set $A\subset S$ we have
\begin{eqnarray}
\mu(A)&=&\sup\left\{\mu(K)\mid K\subset S\mbox{ compact with }K\subset A \right\}\nonumber\\
&=&\inf\left\{\mu(U)\mid U\subset S\mbox{ open with } A\subset U\right\}. \label{dfn_regular_measure}
\end{eqnarray}
\end{thm}

From that we will deduce the following approximation result that we have used in the proof of Theorem~\ref{Feynman-Kac}.

\begin{lem}\label{lemma_Linfty_contin_approx}
Let $S$ be a differentiable connected manifold,
let $\mu$ be a locally finite measure on the Borel $\sigma$-algebra of $S$.
Let $V:S\to[-\infty,\infty]$ be in $L^\infty(S,d\mu)$.
Then there exists a sequence of smooth functions $V_n:S\to\R$ with $|V_n(x)|\le \|V\|_{L^\infty(S,d\mu)}$ 
and $V_n(x)\rightarrow V(x)$ 
pointwise for $\mu$-almost all $x\in S$.
\end{lem}

\begin{proof}
As in Example~\ref{exm_complete_mm} we can find a metric $\rho$ on $S$ such that $(S,\rho,\mu)$ forms a metric measure space, and Ulam's Theorem holds for the measure $\mu$.

In the first step we prove the claim for $V=\1_A$ where $A\subset S$ is a Borel set with $\mu(A)<\infty$.
We fix an open subset $W\subset S$ with $A\subset W$.
Using \eqref{dfn_regular_measure} we can find a sequence of compact sets $K_n\subset A$, $n\ge 1$, with $K_n\subset K_{n+1}$ and $\mu(K_n)\to \mu(A)$, as well as a sequence of open sets $U_n\subset W$, $n\ge 1$, with $A\subset U_n$, $U_{n+1}\subset U_n$ and $\mu(U_n)\to \mu(A)$.
For any $n\ge 1$ we choose a smooth cut-off function $V_n:S\to [0,1]$ being $1$ on $K_n$ and with support in $U_n\subset W$.
In particular, we have $|V_n(x)|\le \|V\|_{L^\infty(S,d\mu)}=1$.
The sequence $V_n$ converges pointwise to $0$ on $\bigcap\limits_{n\ge1} U_n$, and to $1$ on $\bigcup\limits_{n\ge1} K_n$.
By \eqref{dfn_regular_measure},
\[
\mu(\bigcap_{n\ge1} U_n\setminus \bigcup_{n\ge1} K_n ) =\mu(\bigcap_{n\ge1} U_n\setminus A )+\mu(A\setminus \bigcup_{n\ge1} K_n )=0,
\]
hence we have pointwise convergence $V_n(x)\to V(x)$ for $\mu$-almost all $x\in S$.
This proves the lemma for $V=\1_A$.

In the second step we verify the claim for $V=\1_A$ where $A\subset S$ is an arbitrary Borel set.
We choose a locally finite cover of $S$ by countably many open sets $W_i\subset S$, $i\ge 1$ with $\mu(W_i)<\infty$ for any $i$.
For each $i$ we can apply the above argument to $A_i=A\cap W_i$ and get a $\mu$-almost everywhere pointwise convergent sequence of smooth functions
$V^i_n(x)\to \1_{A_i}(x)$ such that $0\le V^i_n\le 1$ and each function $V^i_n$ has support in $W_i$.
Then we take a partition of unity $(\chi_i)_i$ subordinate to the cover $(W_i)_i$.
We note that $V=\sum_i \chi_i \1_{A_i}$ and
 set $V_n=\sum_i \chi_i V^i_n$, which yields the Lemma for arbitrary characteristic functions.

In the third step we consider a step function $V=\sum_{k=1}^\ell \alpha_k\1_{A_k}$ where $A_k$ are Borel sets and $\alpha_k\in\R$.
For any $k$ we find a sequence of smooth functions $V_n^k:S\to\R$ with $|V_n^k(x)|\le 1$ such that for $\mu$-almost all $x\in S$ one has
$V_n^k(x)\rightarrow \1_{A_k}(x)$ as $n\to\infty$.
By setting $V_n=\sum_{k=1}^\ell \alpha_k V_n^k$ we obtain a sequence as required in the Lemma.

Finally, we consider the general case.
Let $V:S\to[-\infty,\infty]$ be in $L^\infty(S,d\mu)$.
For each $k\ge 1$ we consider the step function $F_k:S\to\R$ given by 
\[
F_k=\sum_{m=1}^{k\cdot 2^k} \tfrac{m}{2^k}\Big(\1_{\big\{x\in S\,|\,\; {m}/{2^k}\le V(x) < (m+1)/{2^k}\big\}} -
\1_{\big\{x\in S\,|\,\; -(m+1)/{2^k}< V(x) \le -{m}/{2^k}\big\}} \Big).
\]
By construction we have $|F_k(x)|\le\|V\|_{L^\infty(S,d\mu)}$ for every $x\in S$, and $F_k\to V$ in $L^\infty(S,d\mu)$.
As shown in the third step there are sequences of smooth functions $V_k^r:S\to\R$ with $|F_k^r(x)|\le\|F_k\|_{L^\infty(S,d\mu)}  \le\|V\|_{L^\infty(S,d\mu)}$ for any $x\in S$ and  
$F_k^r(x)\to F_k(x)$, as $r\to\infty$, for $\mu$-almost all $x\in  S$.
Then we apply Lemma~\ref{Lemma:ultimate_approximation}, which concludes the proof.
\end{proof}

\section{Proof of the Kolmogorov-Chentsov Theorem}\label{appendix_Kolmogorov_Chentsov}

\noindent
In the following we will give a proof of Theorem~\ref{thm_continuousversion} for the convenience of the reader.
For the case that $P(\Omega)=1$, a proof can be found in \cite[Thm.~3.23]{Kal} where condition (\ref{Chentsov_condition}) is required to hold for all $t,s\ge 0$.
We will adapt the proof from \cite{Kal} to the slightly more general situation in Theorem~\ref{thm_continuousversion}.

\begin{proof}[Proof of Theorem~\ref{thm_continuousversion}]
W.~l.~o.~g.~we may assume $T=1$.
For any $n\ge 1$ we set $\mathcal{D}_n:=\left\{0,\frac{1}{2^n},\frac{2}{2^n},\ldots,\frac{n-1}{2^n},1\right\}$ and $\mathcal{D}:=\bigcup_n\mathcal{D}_n$.
We define a measurable function $\xi_n:\Omega\to [0,\infty)$ by
\[ 
\xi_n=\max_{k=1,\ldots,2^n}
\rho\left(X_{k\cdot2^{-n}},X_{(k-1)\cdot2^{-n}}\right). 
\]
If $n_0\in\N$ is large enough, namely if ${2^{-n_0}}<\varepsilon$, we get from (\ref{Chentsov_condition}) that $\E[\xi_n^a]\le \frac{C}{2^{n(1+b)}}$ for all $n\ge n_0$.
We fix $\theta\in(0,\frac{b}{a})$ and we get
\begin{align*}
\E\left[\sum_{n\ge n_0}(2^{\theta\cdot n}\xi_n)^a \right]
&=
\sum_{n\ge n_0} 2^{\theta\cdot n\cdot a} \E\left[\xi_n^a \right] \\
&\le 
\sum_{n\ge n_0} 2^{\theta\cdot n\cdot a} \sum_{k=1}^{2^n}\E\left[\rho\left(X_{k\cdot2^{-n}},X_{(k-1)\cdot2^{-n}}\right)^a \right] \\
&\le 
C\cdot \sum_{n\ge n_0}2^{\theta\cdot n\cdot a} \cdot 2^n\cdot 2^{-n(1+b)} \\
&=
C\cdot \sum_{n\ge n_0}2^{-n(b-\theta\cdot a)} 
<
\infty.
\end{align*}
Hence we can find a null set $\mathcal{N}_\theta\in \mathcal{A}$ such that for all $\omega\not\in\mathcal{N}_\theta$ one has  $\sum_{n\ge n_0}\left(2^{\theta\cdot n}\xi_n(\omega)\right)^a<\infty$.
Therefore, for any $\omega\not\in\mathcal{N}_\theta$ there is a constant $C_1(\omega)>0$ with
\[
\xi_n(\omega)\le C_1(\omega)\cdot 2^{-\theta\cdot n}\quad\mbox{ for all }n\ge 1.
\]
Let $m\ge n_0$.
For $s,t\in\mathcal{D}$ with $|t-s|\le\frac{1}{2^m}$ and any $\omega\not\in\mathcal{N}_\theta$ we conclude, using the triangle inequality,
\[
\rho\left(X_s(\omega),X_t(\omega)\right) \le 2\cdot\sum_{n\ge m}\xi_n(\omega) \le
2\cdot C_1(\omega) \cdot \sum_{n\ge m}
2^{-\theta\cdot n}=C_2(\omega,\theta)\cdot2^{-\theta\cdot m}
\]
for a new constant $C_2(\omega,\theta)>0$.
Now we choose a sequence $\theta_i\in (0,b/a)$ with $\theta_i \nearrow b/a$.
Then $\mathcal{N} := \bigcup_i\mathcal{N}_{\theta_i}$ is again a null set and we have for all $\omega\not\in\mathcal{N}$, for all $i$ and all $0<\delta \le 2^{-n_0}$ that
\[
 \sup_{\substack{s,t\in\mathcal{D}\\|s-t|\le\delta}}\;
\rho\left(X_s(\omega),X_t(\omega)\right)\le C_2(\omega,\theta_i)\cdot \delta^{\theta_i}.
\]
For $\omega\not\in\mathcal{N}$ and $t\in [0,1]$ we set
\[
Y_t(\omega):=\lim_{\substack{s\to t\\ s\in\mathcal{D}}}X_s(\omega).
\]
This limit exists because $(S,\rho)$ is complete by assumption.
For $\omega\not\in\mathcal{N}$ the path $Y_\bullet(\omega)$ is H\"older continuous of any order $\theta_i$ and hence of any order $\theta<\frac{b}{a}$.
For $\omega\in\mathcal{N}$ and $t\in [0,1]$ we simply set $Y_t(\omega):=x_0$.

It remains to show that $(Y_t)_{t\in[0,1]}$ is a version of $(X_t)_{t\in[0,1]}$.
Given $t\in [0,1]$, we choose a sequence $(t_k)_{k\ge 1}$ in $\mathcal{D}$ with $t_k\to t$ as $k\to\infty$. 
If $\omega\not\in\mathcal{N}$ we have $X_{t_k}=Y_{t_k}$.
We set $Z_k:=\rho(X_{t_k},Y_t)$.
Since $(Y_t)_{t\in[0,1]}$ has continuous paths it follows that $Z_k(\omega)\to 0$ for any $\omega\not\in\mathcal{N}$.
By Lemma~\ref{fastsicherstochastisch} we get for any $m\ge 1$ that
\begin{equation}\label{zweiterterm}
 P\left(\omega\,\mid\,\rho(X_{t_k}(\omega),Y_t(\omega))\ge\tfrac{1}{m} \right)\xrightarrow{k\to\infty} 0.
\end{equation}
For any $m\ge 1$ Markov's Inequality and condition (\ref{Chentsov_condition}) imply
\begin{equation}\label{ersterterm}
 P\left(\omega\,\mid\,\rho(X_t(\omega),X_{t_k}(\omega))\ge\tfrac{1}{m} \right)
\le m^a\cdot \E[\rho(X_t,X_{t_k})^a]\xrightarrow{k\to\infty} 0.
\end{equation}
For any $m\ge 1$ and any $x,y,z\in S$ with $\rho(x,y)\ge\tfrac{2}{m}$ the triangle inequality for $\rho$ yields that $\rho(x,z)\ge\tfrac{1}{m}$ or $\rho(z,y)\ge\tfrac{1}{m}$. 
Hence\[
 \{\omega\,\mid\,\rho(X_t(\omega),Y_t(\omega))\ge\tfrac{2}{m} \}
 \subset \{\omega\,\mid\,\rho(X_t(\omega),X_{t_k}(\omega))\ge\tfrac{1}{m} \}
 \cup 
  \{\omega\,\mid\,\rho(X_{t_k}(\omega),Y_t(\omega))\ge\tfrac{1}{m} \}
\]
for all $k\ge 1$.
By (\ref{zweiterterm}) and (\ref{ersterterm}) this implies
\begin{align*}
P\big(\{\omega&\,\mid\,\rho(X_t(\omega),Y_t(\omega))\ge\tfrac{2}{m} \}\big)\\
&\le\;
 P\big(\{\omega\,\mid\,\rho(X_t(\omega),X_{t_k}(\omega))\ge\tfrac{1}{m} \}\big)+
 P\big(\{\omega\,\mid\,\rho(X_{t_k}(\omega),Y_t(\omega))\ge\tfrac{1}{m} \}\big)\xrightarrow{k\to\infty} 0.
\end{align*}
Thus $\{\omega\mid\rho(X_t(\omega),Y_t(\omega))\ge\tfrac{2}{m} \}$ is a null set for every $m\ge 1$, and so is 
$\{\omega\mid X_t(\omega)\ne Y_t(\omega) \}\subset \bigcup\limits_m\left\{ \omega\mid\rho(X_t(\omega),Y_t(\omega))\ge\tfrac{2}{m} \right\}$.
This means that $(Y_t)_{t\in[0,1]}$ is a version of $(X_t)_{t\in[0,1]}$.
\end{proof}

\section{Proof of Lemma~\ref{lem:MetrikVergleich}}\label{app:LemmaProof}

\noindent
Before we prove Lemma~\ref{lem:MetrikVergleich} we show by example that the lemma fails if one drops the assumption of smoothness of the boundary. So the lemma is not as ``obvious'' as it might seem at first glance.

\begin{exm}
Let $M=\R^2$ with the Euclidean metric and let $\overline\Omega = \{(s,t)\mid -1\le s\le 1,\, -1 \le t \le \sqrt{|s|}\}$.
Put $x_\eps:=(-\eps,\sqrt{\eps})$ and $y_\eps:=(\eps,\sqrt{\eps})$, where $\eps\in(0,1]$.
\begin{center}
\psset{unit=2}
\begin{pspicture}(-3,-1.2)(3,1.3)
\psplot[plotpoints=400,linewidth=0.8pt]{-1}{1}{x abs sqrt}
\psline[linewidth=0.8pt](-1,1)(-1,-1)(1,-1)(1,1)
\psdots(0.25,0.5)(0,0.01)(-0.25,0.5)
\psline[linewidth=0.3pt](0.25,0.5)(0,0.01)(-0.25,0.5)(0.25,0.5)
\rput(-0.22,0.7){$x_\eps$}
\rput(0.23,0.7){$y_\eps$}
\rput(0,-0.2){$(0,0)$}
\rput(-0.7,-0.7){$\Omega$}
\end{pspicture}

\bildnummer
\end{center}
Then $\rho^{\R^2}(x_\eps,y_\eps) = \|x_\eps-y_\eps\| = 2\eps$ and $\rhoO(x_\eps,y_\eps) = \|x_\eps\| + \|y_\eps\| = 2\sqrt{\eps(\eps+1)}$.
Hence
$$
\frac{\rhoO(x_\eps,y_\eps)}{\rho^{\R^2}(x_\eps,y_\eps)} =
\sqrt{1+\frac{1}{\eps}}
$$
is unbounded as $\eps\searrow 0$.
\end{exm}

In order to show Lemma~\ref{lem:MetrikVergleich} we need the following elementary comparison result:

\begin{lem}\label{lem:comparison}
Let $f:[0,T] \to \R$ be a $C^2$-function and let $C_1$ and $C_2$ be positive constants such that
$$
\begin{cases}
\ddot{f} \ge -C_1^2 f & \mbox{ on } [0,T], \\
\dot{f}(0) \ge -C_2 f(0), & \\
f(0) > 0.& 
\end{cases}
$$
Then
$$
f(t) \ge f(0)\left(\cos(C_1t)-\frac{C_2}{C_1}\sin(C_1t)\right)
$$
holds for all $t\in[0,T']$ where $T'=\min\{T,\arctan(C_1/C_2)/C_1\}$.
\end{lem}

\begin{proof}
We put $h(t):=\cos(C_1t)-\frac{C_2}{C_1}\sin(C_1t)$. 
Then we have
$$
\begin{cases}
\ddot{h} = -C_1^2 h , &  \\
\dot{h}(0) = -C_2 h(0), & \\
h(0) = 1.& 
\end{cases}
$$
Note that $h>0$ on $[0,\arctan(C_1/C_2)/C_1)$.
Now we define on this interval $v(t):= \dot{h}(t)/h(t)$ and $u(t):=\dot{f}(t)/f(t)$ is defined on the maximal interval $[0,T_0)$ on which $f$ is defined and remains positive.
Then we get
$$
\begin{cases}
\dot{u} \ge -C_1^2 - u^2, &\\
u(0) \ge -C_2 ,
\end{cases}
\quad\mbox{ and }\quad\quad
\begin{cases}
\dot{v} = -C_1^2 - v^2, &\\
v(0) = -C_2 .
\end{cases}
$$
Standard comparison \cite[Thm.~7 on p.~26]{BR} yields $u\ge v$ on the common domain.
Integrating this we find
$$
\log\left(\frac{f(t)}{f(0)}\right) 
= \int_0^t u(s)\, ds 
\ge \int_0^t v(s)\, ds 
= \log\left(\frac{h(t)}{h(0)}\right),
$$
hence
$$
\frac{f(t)}{f(0)} \ge \frac{h(t)}{h(0)} = h(t). 
$$
This is the asserted inequality.
We also see that $f$ remains positive as long as $h$ does.
Hence the inequality holds on the stated interval.
\end{proof}

\begin{proof}[Proof of Lemma~\ref{lem:MetrikVergleich}]
The first inequality is clear because the set of curves in $M$ joining $x$ and $y$ contains the set of such curves in $\overline\Omega$.
Let $\Diag := \{(x,x)\mid x\in\overline\Omega\}$ be the diagonal in $\overline\Omega\times\overline\Omega$.
Then ${\rhoO}/{\rho_M}$ is a continuous positive function on $(\overline\Omega\times\overline\Omega) \setminus \Diag$.
We show that $\rhoO/\rho_M$ can be extended to a continuous function on $\overline\Omega\times\overline\Omega$ by putting it equal to $1$ on the diagonal.
By compactness of $\overline\Omega\times\overline\Omega$, this extension of $\rhoO/\rho_M$ must be bounded and the lemma is proved.

Hence we have to show that
\be
\label{eq:QuotientLimit}
\frac{\rhoO(x_j,y_j)}{\rho_M(x_j,y_j)} \to 1
\quad\mbox{ as }j\to\infty
\ee
for any $x_j,y_j\in\overline\Omega$ such that $x_j\neq y_j$ and $\lim_jx_j = \lim_jy_j =: x$.
If $x\in\Omega$, then $x_j$ and $y_j$ will eventually lie in a convex neighborhood of $x$ entirely contained in $\Omega$.
Then $\rhoO(x_j,y_j)=\rho_M(x_j,y_j)$ and \eqref{eq:QuotientLimit} is clear.

The problematic case occurs when $x\in\partial\Omega$.
For $\eps>0$ we let 
$$
\Omega_\eps := \{x\in M\mid \rho_M(x,\overline\Omega)<\eps\}.
$$
Let $\nu$ be the exterior unit normal field of $\overline\Omega$ along $\partial\Omega$.
For $\eps>0$ sufficiently small we have 
$$
\Omega_\eps = \overline\Omega \cup \{\exp_z(\delta \nu(z))\mid z\in\partial\Omega, 0<\delta<\eps\}
$$
where $\exp$ denotes the Riemannian exponential function of $M$.
\begin{center}
\psset{xunit=.5pt,yunit=.5pt,runit=.5pt}
\begin{pspicture}(450,290)
{
\newrgbcolor{curcolor}{0 0 0}
\pscustom[linewidth=1,linecolor=curcolor]
{
\newpath
\moveto(56.69291815,155.9055107)
\curveto(137.56541415,165.7820907)(195.28892415,103.6126707)(198.42519415,35.4330707)
\curveto(241.04011415,88.8487907)(330.29111415,125.5235407)(389.76378415,141.7322907)
}
}
{
\newrgbcolor{curcolor}{0 0 0}
\pscustom[linewidth=1,linecolor=curcolor]
{
\newpath
\moveto(56.47652715,155.9055107)
\lineto(56.47652715,155.9055107)
\curveto(92.32411415,229.0822607)(206.43038415,245.0344207)(276.16156415,248.0315007)
}
}
{
\newrgbcolor{curcolor}{0 0 0}
\pscustom[linewidth=1,linecolor=curcolor]
{
\newpath
\moveto(276.37795,248.0314957)
\curveto(325.08784,221.0559607)(367.1614,186.4864507)(389.76378,141.7322807)
}
}
{
\newrgbcolor{curcolor}{0 0 0}
\pscustom[linewidth=1.02597833,linecolor=curcolor]
{
\newpath
\moveto(205.51181415,240.9448807)
\curveto(200.65040415,203.1617007)(232.56551415,204.2224707)(250.31081415,185.5432207)
\curveto(275.69045415,158.8278107)(237.76326415,110.9315307)(288.55118415,103.2126007)
}
}
{
\newrgbcolor{curcolor}{0.40000001 0.40000001 0.40000001}
\pscustom[linestyle=none,fillstyle=solid,fillcolor=curcolor]
{
\newpath
\moveto(283.39874,104.9991407)
\curveto(278.42946,106.3127107)(275.10963,107.6750807)(272.07762,109.6450507)
\curveto(266.80677,113.0696507)(263.85116,117.3061007)(261.94494,124.1688707)
\curveto(261.0212,127.4945307)(260.94174,129.0121807)(260.86104,144.8711807)
\curveto(260.79144,158.5548507)(260.61547,162.8774607)(259.9832,166.4389607)
\curveto(258.40671,175.3192107)(256.00787,180.2951607)(250.49236,186.1259107)
\curveto(246.2587,190.6015507)(241.90571,193.7193407)(231.82361,199.4972207)
\curveto(222.3095,204.9496007)(218.1577,207.7682407)(214.7515,211.0874107)
\curveto(210.09486,215.6250507)(207.60626,220.4876007)(206.23812,227.7219307)
\curveto(205.68664,230.6379807)(205.79872,239.2664107)(206.39584,239.8635207)
\curveto(206.85289,240.3205717)(212.01547,241.1630147)(224.32178,242.7887177)
\curveto(240.7243,244.9555447)(255.68324,246.3274297)(269.31829,246.9153477)
\lineto(276.43796,247.2223337)
\lineto(282.73154,243.5731727)
\curveto(321.13057,221.3085407)(349.61635,197.6091307)(370.65365,170.4243107)
\curveto(376.74269,162.5559407)(383.8492,151.5806507)(387.42847,144.5172907)
\lineto(388.6099,142.1858607)
\lineto(383.35967,140.6924607)
\curveto(355.17952,132.6767807)(317.43216,118.2090707)(293.93811,106.4192007)
\curveto(291.37427,105.1326107)(288.69254,104.0915607)(287.97869,104.1057607)
\curveto(287.26485,104.1199607)(285.20387,104.5219807)(283.39874,104.9991407)
\lineto(283.39874,104.9991407)
\closepath
}
}
{
\newrgbcolor{curcolor}{0 0 0}
\pscustom[linewidth=1,linecolor=curcolor]
{
\newpath
\moveto(188.25197,237.8582707)
\curveto(179.16535,205.5118107)(205.51181,195.3385807)(231.02582,174.8865807)
\curveto(257.68735,153.5147307)(219.68504,106.2992107)(264.20472,90.1259807)
}
}
{
\newrgbcolor{curcolor}{0.80000001 0.80000001 0.80000001}
\pscustom[linestyle=none,fillstyle=solid,fillcolor=curcolor]
{
\newpath
\moveto(261.64062,91.8644407)
\curveto(257.01446,93.8866907)(253.53918,96.2162607)(250.33392,99.4436107)
\curveto(246.25453,103.5511107)(244.35415,106.9351807)(242.87153,112.7321707)
\curveto(241.49426,118.1172707)(241.28762,123.8632007)(241.93597,138.7469407)
\curveto(242.66059,155.3813407)(241.38161,162.8860807)(236.64621,169.7861007)
\curveto(234.20699,173.3403207)(229.77454,177.0059507)(215.02927,187.6632707)
\curveto(204.32815,195.3976407)(197.73988,201.2351307)(193.7176,206.5463507)
\curveto(191.43475,209.5607407)(188.87316,214.7234707)(187.96015,218.1501707)
\curveto(187.09659,221.3912207)(187.09451,229.5286107)(187.95615,233.7793907)
\lineto(188.57049,236.8101907)
\lineto(196.3054,238.3380307)
\curveto(200.5596,239.1783407)(204.17368,239.8658707)(204.3367,239.8658707)
\curveto(204.50347,239.8658707)(204.56778,237.9403307)(204.48375,235.4626907)
\curveto(204.0076,221.4234607)(208.67586,212.5190807)(220.75418,204.4279207)
\curveto(222.17285,203.4775707)(226.49941,200.8949907)(230.36877,198.6888607)
\curveto(241.14108,192.5469807)(245.81154,189.2530907)(249.80086,184.9841807)
\curveto(255.20592,179.2003007)(258.17197,172.0541807)(259.11932,162.5331407)
\curveto(259.32062,160.5100207)(259.45122,152.9084407)(259.42121,144.9616507)
\curveto(259.38621,135.6933607)(259.50071,130.0242307)(259.75969,128.2017507)
\curveto(261.20306,118.0449407)(265.11445,111.8515807)(272.88061,107.4258407)
\curveto(275.65032,105.8474507)(281.36281,103.8490607)(284.83189,103.2449507)
\lineto(287.01824,102.8642207)
\lineto(280.59435,99.4701807)
\curveto(277.06121,97.6034607)(272.07838,94.8606607)(269.52139,93.3750807)
\curveto(266.96439,91.8894807)(264.7342,90.6843707)(264.56541,90.6970407)
\curveto(264.39661,90.7097407)(263.08046,91.2350407)(261.64063,91.8644407)
\lineto(261.64063,91.8644407)
\closepath
}
}
\rput(350,150){\psframebox*[framearc=.6]{$\Omega$}}
\rput(260,170){\psframebox*[framearc=.6]{$\Omega_\eps$}}
\rput(100,170){\psframebox*[framearc=.6]{$M$}}
\end{pspicture}

\bildnummer
\end{center}
Moreover, the map $(\delta,z)\mapsto \exp_z(\delta \nu(z))$ is a diffeomorphism $[0,\eps]\times \partial\Omega \to \overline\Omega_\eps \setminus \Omega$.
Denote the Riemannian distance function on $\overline\Omega_\eps$ by $\rho_\eps$.
Then 
\be
\label{eq:QuotientLimit2}
\frac{\rho_\eps(x_j,y_j)}{\rho_M(x_j,y_j)} \to 1
\quad\mbox{ as }j\to\infty
\ee
because $x\in\Omega_\eps$.
It remains to compare $\rhoO(x_j,y_j)$ and $\rho_\eps(x_j,y_j)$.
Clearly, $\rho_\eps\le \rhoO$.
To obtain an inverse inequality, we let $c:[0,1]\to\overline\Omega_\eps$ be a piecewise smooth curve joining $x_j$ and $y_j$.
We deform $c$ to a curve in $\overline\Omega$ by putting
$$
c_\tau(s) := \begin{cases}
             c(s),& \mbox{ if }c(s)\in\overline\Omega ,\\
             \exp_{z}(\tau\delta \nu(z)),  & \mbox{ if }c(s)=\exp_{z}(\delta \nu(z))\mbox{ for some }(\delta,z)\in [0,\eps]\times \partial\Omega .
             \end{cases}
$$

\begin{center}

\psset{xunit=.5pt,yunit=.5pt,runit=.5pt}
\begin{pspicture}(450,290)
{
\newrgbcolor{curcolor}{0 0 0}
\pscustom[linewidth=1,linecolor=curcolor]
{
\newpath
\moveto(56.69291815,155.9055107)
\curveto(137.56541415,165.7820907)(195.28892415,103.6126707)(198.42519415,35.4330707)
\curveto(241.04011415,88.8487907)(330.29111415,125.5235407)(389.76378415,141.7322907)
}
}
{
\newrgbcolor{curcolor}{0 0 0}
\pscustom[linewidth=1,linecolor=curcolor]
{
\newpath
\moveto(56.47652715,155.9055107)
\lineto(56.47652715,155.9055107)
\curveto(92.32411415,229.0822607)(206.43038415,245.0344207)(276.16156415,248.0315007)
}
}
{
\newrgbcolor{curcolor}{0 0 0}
\pscustom[linewidth=1,linecolor=curcolor]
{
\newpath
\moveto(276.37795,248.0314957)
\curveto(325.08784,221.0559607)(367.1614,186.4864507)(389.76378,141.7322807)
}
}
{
\newrgbcolor{curcolor}{0 0 0}
\pscustom[linewidth=1.02597833,linecolor=curcolor]
{
\newpath
\moveto(205.51181415,240.9448807)
\curveto(200.65040415,203.1617007)(232.56551415,204.2224707)(250.31081415,185.5432207)
\curveto(275.69045415,158.8278107)(237.76326415,110.9315307)(288.55118415,103.2126007)
}
}
{
\newrgbcolor{curcolor}{0.40000001 0.40000001 0.40000001}
\pscustom[linestyle=none,fillstyle=solid,fillcolor=curcolor]
{
\newpath
\moveto(283.39874,104.9991407)
\curveto(278.42946,106.3127107)(275.10963,107.6750807)(272.07762,109.6450507)
\curveto(266.80677,113.0696507)(263.85116,117.3061007)(261.94494,124.1688707)
\curveto(261.0212,127.4945307)(260.94174,129.0121807)(260.86104,144.8711807)
\curveto(260.79144,158.5548507)(260.61547,162.8774607)(259.9832,166.4389607)
\curveto(258.40671,175.3192107)(256.00787,180.2951607)(250.49236,186.1259107)
\curveto(246.2587,190.6015507)(241.90571,193.7193407)(231.82361,199.4972207)
\curveto(222.3095,204.9496007)(218.1577,207.7682407)(214.7515,211.0874107)
\curveto(210.09486,215.6250507)(207.60626,220.4876007)(206.23812,227.7219307)
\curveto(205.68664,230.6379807)(205.79872,239.2664107)(206.39584,239.8635207)
\curveto(206.85289,240.3205717)(212.01547,241.1630147)(224.32178,242.7887177)
\curveto(240.7243,244.9555447)(255.68324,246.3274297)(269.31829,246.9153477)
\lineto(276.43796,247.2223337)
\lineto(282.73154,243.5731727)
\curveto(321.13057,221.3085407)(349.61635,197.6091307)(370.65365,170.4243107)
\curveto(376.74269,162.5559407)(383.8492,151.5806507)(387.42847,144.5172907)
\lineto(388.6099,142.1858607)
\lineto(383.35967,140.6924607)
\curveto(355.17952,132.6767807)(317.43216,118.2090707)(293.93811,106.4192007)
\curveto(291.37427,105.1326107)(288.69254,104.0915607)(287.97869,104.1057607)
\curveto(287.26485,104.1199607)(285.20387,104.5219807)(283.39874,104.9991407)
\lineto(283.39874,104.9991407)
\closepath
}
}
{
\newrgbcolor{curcolor}{0 0 0}
\pscustom[linewidth=1,linecolor=curcolor]
{
\newpath
\moveto(188.25197,237.8582707)
\curveto(179.16535,205.5118107)(205.51181,195.3385807)(231.02582,174.8865807)
\curveto(257.68735,153.5147307)(219.68504,106.2992107)(264.20472,90.1259807)
}
}
{
\newrgbcolor{curcolor}{0.80000001 0.80000001 0.80000001}
\pscustom[linestyle=none,fillstyle=solid,fillcolor=curcolor]
{
\newpath
\moveto(261.64062,91.8644407)
\curveto(257.01446,93.8866907)(253.53918,96.2162607)(250.33392,99.4436107)
\curveto(246.25453,103.5511107)(244.35415,106.9351807)(242.87153,112.7321707)
\curveto(241.49426,118.1172707)(241.28762,123.8632007)(241.93597,138.7469407)
\curveto(242.66059,155.3813407)(241.38161,162.8860807)(236.64621,169.7861007)
\curveto(234.20699,173.3403207)(229.77454,177.0059507)(215.02927,187.6632707)
\curveto(204.32815,195.3976407)(197.73988,201.2351307)(193.7176,206.5463507)
\curveto(191.43475,209.5607407)(188.87316,214.7234707)(187.96015,218.1501707)
\curveto(187.09659,221.3912207)(187.09451,229.5286107)(187.95615,233.7793907)
\lineto(188.57049,236.8101907)
\lineto(196.3054,238.3380307)
\curveto(200.5596,239.1783407)(204.17368,239.8658707)(204.3367,239.8658707)
\curveto(204.50347,239.8658707)(204.56778,237.9403307)(204.48375,235.4626907)
\curveto(204.0076,221.4234607)(208.67586,212.5190807)(220.75418,204.4279207)
\curveto(222.17285,203.4775707)(226.49941,200.8949907)(230.36877,198.6888607)
\curveto(241.14108,192.5469807)(245.81154,189.2530907)(249.80086,184.9841807)
\curveto(255.20592,179.2003007)(258.17197,172.0541807)(259.11932,162.5331407)
\curveto(259.32062,160.5100207)(259.45122,152.9084407)(259.42121,144.9616507)
\curveto(259.38621,135.6933607)(259.50071,130.0242307)(259.75969,128.2017507)

\curveto(261.20306,118.0449407)(265.11445,111.8515807)(272.88061,107.4258407)
\curveto(275.65032,105.8474507)(281.36281,103.8490607)(284.83189,103.2449507)
\lineto(287.01824,102.8642207)
\lineto(280.59435,99.4701807)
\curveto(277.06121,97.6034607)(272.07838,94.8606607)(269.52139,93.3750807)
\curveto(266.96439,91.8894807)(264.7342,90.6843707)(264.56541,90.6970407)
\curveto(264.39661,90.7097407)(263.08046,91.2350407)(261.64063,91.8644407)
\lineto(261.64063,91.8644407)
\closepath
}
}
\rput(-5,0){
{
\newrgbcolor{curcolor}{0 0 0.70588237}
\psclip{\psccurve[linestyle=none](212.59843415,212.5984207)(250,200)(250,135)}
\pscustom[linewidth=1.3818897,linecolor=curcolor]
{
\newpath
\moveto(205.51181415,240.9448807)
\curveto(200.65040415,203.1617007)(232.56551415,204.2224707)(250.31081415,185.5432207)
\curveto(275.69045415,158.8278107)(237.76326415,110.9315307)(288.55118415,103.2126007)
}
\endpsclip
}
{
\newrgbcolor{curcolor}{0 0 0.70588237}
\pscustom[linewidth=1.3818897,linecolor=curcolor]
{
\newpath
\moveto(212.59843,212.5984207)
\curveto(233.85827,184.2519707)(259.75986,146.3250107)(276.37795,106.2992107)
}
}
{
\newrgbcolor{curcolor}{0 0 0.70588237}
\pscustom[linewidth=0.99921262,linecolor=curcolor,linestyle=dashed,dash=1.9984252 0.9992126]
{
\newpath
\moveto(212.59843,212.5984207)
\curveto(219.68504,198.4252007)(248.0315,198.4252007)(259.8604,141.7322807)
}
}
{
\newrgbcolor{curcolor}{0 0 0}
\pscustom[linestyle=none,fillstyle=solid,fillcolor=curcolor]
{
\newpath
\moveto(277.87796435,106.2992002)
\curveto(277.87796435,105.47077307)(277.20639148,104.7992002)(276.37796435,104.7992002)
\curveto(275.54953723,104.7992002)(274.87796435,105.47077307)(274.87796435,106.2992002)
\curveto(274.87796435,107.12762732)(275.54953723,107.7992002)(276.37796435,107.7992002)
\curveto(277.20639148,107.7992002)(277.87796435,107.12762732)(277.87796435,106.2992002)
\closepath
}
}
{
\newrgbcolor{curcolor}{0 0 0}
\pscustom[linestyle=none,fillstyle=solid,fillcolor=curcolor]
{
\newpath
\moveto(214.09842334,212.59839453)
\curveto(214.09842334,211.76996741)(213.42685046,211.09839453)(212.59842334,211.09839453)
\curveto(211.76999621,211.09839453)(211.09842334,211.76996741)(211.09842334,212.59839453)
\curveto(211.09842334,213.42682166)(211.76999621,214.09839453)(212.59842334,214.09839453)
\curveto(213.42685046,214.09839453)(214.09842334,213.42682166)(214.09842334,212.59839453)
\closepath
}
}
}
\rput(350,150){\psframebox*[framearc=.6]{$\Omega$}}
\rput(235,220){\psframebox*[framearc=.6]{$x_j$}}
\rput(295,130){\psframebox*[framearc=.6]{$y_j$}}
\rput(280,175){\psframebox*[framearc=.6]{\blue $c_0$}}
\rput(200,175){\psframebox*[framearc=.6]{\blue $c=c_1$}}
\rput(100,170){\psframebox*[framearc=.6]{$M$}}
\end{pspicture}

\bildnummer
\end{center}
Then $c_\tau$ is a piecewise smooth curve in $\overline\Omega_\tau$ joining $x_j$ and $y_j$ with $c_1=c$.
In particular, $c_0$ joins $x_j$ and $y_j$ in $\overline\Omega$.

For any fixed $s_0$, the curve $\tau\mapsto c_\tau(s_0)$ is a geodesic by construction and hence $\frac{\partial c_t}{\partial s}(s_0)$ is a Jacobi field along $\tau\mapsto c_\tau(s_0)$.
Therefore the Jacobi field equation 
$$
\frac{\nabla^2}{\partial \tau^2} \frac{\partial c_\tau}{\partial s}
= 
- R\Big(\frac{\partial c_\tau}{\partial s},\frac{\partial c_\tau}{\partial \tau}\Big)\frac{\partial c_\tau}{\partial \tau}
$$
holds, where $\nabla$ denotes the covariant derivative and $R$ the curvature tensor.
Let $C_1>0$ be such that $|R|\le C_1$ on $\overline\Omega_\eps$.
For fixed $s$ where the second branch in the definition of $c_\tau$ applies, we put $f(\tau) := \left|\frac{\partial c_\tau}{\partial s}(s)\right|^2$.
We compute
\begin{align*}
\ddot{f}(\tau)
&=
\frac{\partial^2}{\partial\tau^2}\<\frac{\partial c_\tau}{\partial s},\frac{\partial c_\tau}{\partial s}\> \\
&=
2 \frac{\partial}{\partial\tau}\<\frac{\nabla}{\partial\tau}\frac{\partial c_\tau}{\partial s},\frac{\partial c_\tau}{\partial s}\> \\
&=
2 \<\frac{\nabla^2}{\partial\tau^2}\frac{\partial c_\tau}{\partial s},\frac{\partial c_\tau}{\partial s}\>
+2 \left|\frac{\nabla}{\partial\tau}\frac{\partial c_\tau}{\partial s}\right|^2 \\
&\ge
-2\<R\Big(\frac{\partial c_\tau}{\partial s},\frac{\partial c_\tau}{\partial \tau}\Big)\frac{\partial c_\tau}{\partial \tau} , \frac{\partial c_\tau}{\partial s}\> \\
&\ge
-2C_1 \eps^2 f(\tau)
\end{align*}
because
$$
\left|\<R\Big(\frac{\partial c_\tau}{\partial s},\frac{\partial c_\tau}{\partial \tau}\Big)\frac{\partial c_\tau}{\partial \tau} , \frac{\partial c_\tau}{\partial s}\>\right|
\le
C_1 \left|\frac{\partial c_\tau}{\partial \tau}\right|^2 \left|\frac{\partial c_\tau}{\partial s}\right|^2
\le
C_1\delta^2\left|\frac{\partial c_\tau}{\partial s}\right|^2
\le
C_1\eps^2 f(\tau) .
$$
Moreover, let $C_2>0$ be a bound for the second fundamental form of $\partial\Omega$, i.e., $|II|\le C_2$.
Then
\begin{align*}
\dot{f}(0)
&=
2 \<\frac{\nabla}{\partial\tau}\frac{\partial c_\tau}{\partial s},\frac{\partial c_\tau}{\partial s}\>\Big|_{\tau=0} \\
&=
2 \<\frac{\nabla}{\partial s}\frac{\partial c_\tau}{\partial \tau},\frac{\partial c_\tau}{\partial s}\>\Big|_{\tau=0} \\
&=
2 \<\frac{\nabla}{\partial s}\left(\delta(s)\nu(c_0(s))\right),c_0'(s)\> \\
&=
2 \<\delta'(s)\nu(c_0(s)) + \delta(s)\nabla_{c_0'(s)}\nu,c_0'(s)\> \\
&=
2 \delta(s)\<\nabla_{c_0'(s)}\nu,c_0'(s)\> \\
&=
2 \delta(s)II(c_0'(s),c_0'(s)) \\
&\ge
-2C_2\eps f(0).
\end{align*}
By Lemma~\ref{lem:comparison} we have
$$
f(1) \ge f(0)\left(\cos(\sqrt{2C_1}\eps)-\frac{\sqrt{2}C_2}{\sqrt{C_1}}\sin(\sqrt{2C_1}\eps)\right)
$$
provided $1\le \arctan(\sqrt{C_1}/(\sqrt{2}C_2))/(\sqrt{2C_1}\eps)$ which is true for sufficiently small $\eps$.
Writing $\cos(\sqrt{2C_1}\eps)-\frac{\sqrt{2}C_2}{\sqrt{C_1}}\sin(\sqrt{2C_1}\eps) = (1-\eta(\eps))^2$ this means
$$
|c'(s)| = |c_1'(s)| \ge (1-\eta(\eps))\cdot |c_0'(s)|
$$
with $\eta(\eps) \to 0$ as $\eps\to 0$.
Hence we have for the lengths of $c$ and $c_0$:
$$
L(c) \ge (1-\eta(\eps))\cdot L(c_0)
$$
and therefore
$$
\rho_\eps(x_j,y_j) \ge (1-\eta(\eps))\cdot \rhoO(x_j,y_j).
$$
Thus 
$$
\limsup_{j\to\infty} \frac{\rhoO(x_j,y_j)}{\rho_M(x_j,y_j)} 
\le
\limsup_{j\to\infty} \frac{\rhoO(x_j,y_j)}{\rho_\eps(x_j,y_j)} 
\limsup_{j\to\infty} \frac{\rho_\eps(x_j,y_j)}{\rho_M(x_j,y_j)} 
\le 
\frac{1}{1-\eta(\eps)} .
$$
Since this holds for any sufficiently small $\eps>0$, we get $\limsup_{j\to\infty} \frac{\rhoO(x_j,y_j)}{\rho_M(x_j,y_j)} \le1$ and thus $\lim_{j\to\infty} \frac{\rhoO(x_j,y_j)}{\rho_M(x_j,y_j)} =1$ as required.
\end{proof}

\end{document}